\documentclass[a4paper,11pt,reqno]{amsart}
\usepackage{etex}
\usepackage[english]{babel}
\usepackage{dsfont,amscd}
\usepackage[cyr]{aeguill}
\usepackage[pdfdisplaydoctitle=true,
            colorlinks=true,
            urlcolor=blue,
            citecolor=blue,
            linkcolor=blue,
            pdfstartview=FitH,
            pdfpagemode=None,
            bookmarksnumbered=true]{hyperref}
\usepackage{graphicx}
\usepackage{amssymb}
\usepackage{changes}
\usepackage[all]{xy}
\usepackage{lscape}
\usepackage{enumerate}
\usepackage{pstricks}
\usepackage{multicol}
\usepackage{float}
\usepackage{todonotes}
\usepackage{supertabular}

\usepackage{tikz}
\usetikzlibrary{shapes}
\usetikzlibrary{fit}

\tikzset{flecheo/.style={->,>=latex,very thick}}
\tikzset{flecheno/.style={very thick}}
\tikzstyle{m}=[circle,draw, minimum size=7mm,inner sep=0pt]

\def\centuryname{century}
\let\up\textsuperscript
\let\markcent\textsc
  \newcommand*{\century}[1]
  {\markcent{\uppercase\expandafter{\romannumeral #1\relax}}\up{th} \centuryname}

\newcommand{\ban}{\begin{equation*}}
\newcommand{\ean}{\end{equation*}}

\newtheorem{thm}{Theorem}[section]
\newtheorem{cor}[thm]{Corollary}
\newtheorem{lem}[thm]{Lemma}
\newtheorem{prop}[thm]{Proposition}
\theoremstyle{definition}
\newtheorem{defn}[thm]{Definition}
\newtheorem{hypo}[thm]{Hypothesis}
\theoremstyle{remark}
\newtheorem{rem}[thm]{Remark}
\newtheorem{exemple}[thm]{Example}
\numberwithin{equation}{section}

\newcommand{\set}[1]{\left\{#1\right\}}

\newcommand{\sldc}{\mathrm{SL}(2,\mathbb{C})}

\newcommand{\GR}{\mathbb{R}}
\newcommand{\GC}{\mathbb{C}}
\newcommand{\vv}{\mathbf{v}}
\newcommand{\uu}{\mathbf{u}}
\newcommand{\ff}{\mathbf{f}}
\newcommand{\bg}{\mathbf{g}}
\newcommand{\bh}{\mathbf{h}}

\newcommand{\xx}{\mathbf{x}}
\newcommand{\kk}{\mathbf{k}}

\DeclareMathOperator{\Sym}{Sym}
\DeclareMathOperator{\Hom}{Hom}
\DeclareMathOperator{\Val}{val}
\DeclareMathOperator{\gr}{gr}

\newcommand{\tr}[3]{( #1,#2 )_{#3}}
\newcommand{\di}[1]{\mathsf{#1}}
\newcommand{\Mdi}[1]{\mathbf{#1}}
\newcommand{\cov}[1]{\mathbf{Cov}(#1)}
\newcommand{\inv}[1]{\mathbf{Inv}(#1)}
\newcommand{\invj}[1]{\mathbf{Inv}_{j}(#1)}
\newcommand{\Sn}[1]{\mathrm{S}_{#1}}

\hypersetup{
    pdfauthor={Marc Olive}, 
    pdftitle={About Gordan's algorithm on binary forms}, 
    pdfsubject={}, 
    pdfkeywords={Classical invariant theory; Covariants; Gordan's algorithm}, 
    pdflang=en, 
    }

\begin{document}
\normalem
\title{About Gordan's algorithm for binary forms}%
\author{Marc Olive}%
\address{Institut de Math\'{e}matiques de Marseille (I2M)
	Aix Marseille Universit\'{e} - CNRS : UMR7373 - Ecole Centrale de Marseille}%
\email{marc.olive@math.cnrs.fr}%

\date{\today}

\begin{abstract}
In this paper, we present a modern version of Gordan's algorithm on binary forms. Symbolic method is reinterpreted in terms of $\sldc$--equivariant homomorphisms defined upon Cayley operator and polarization operator. A graphical approach is thus developed to obtain Gordan's ideal, a central key to get covariant basis of binary forms. To illustrate the power of this method, we obtain for the first time a minimal covariant bases for $\Sn{6}\oplus\Sn{4}$, $\Sn{6}\oplus\Sn{4}\oplus \Sn{2}$ and a minimal invariant bases of $\Sn{8}\oplus\Sn{4}\oplus\Sn{4}$.
\end{abstract}

\maketitle


\begin{scriptsize}
\setcounter{tocdepth}{2}
\tableofcontents
\end{scriptsize}


\section{Introduction}
\label{sec:intro}

\emph{Classical invariant theory} was a very active research field throughout the \century{19}. As pointed out by Parshall~\cite{Par1989}, this field was initiated by Gauss' \emph{Disquisitiones Arithmeticae} (1801) in which he studied a linear change of variables for quadratic forms with integer coefficients. About forty years later, Boole~\cite{Boo1841} established the main purpose of what has become today \emph{classical invariant theory}. Cayley~\cite{Cri1986,Cri1988} deeply investigated this field of research and developed important tools still in use nowadays, such as the \emph{Cayley Omega operator}. During about fifteen years (until Cayley's seventh memoir~\cite{Cay1861} in 1861) the English school of invariant theory, mainly led by Cayley and Sylvester, developed important tools to compute explicit \emph{invariant generators} of binary forms. Thus, the role of calculation deeply influenced this first approach in invariant theory~\cite{Cri1986}.

Meanwhile, a German school principally conducted by Clebsch, Aronhold and Gordan, developed their own approach, using the \emph{symbolic method} (also used with slightly different notations by the English school). In 1868, Gordan, who was called the ``King of invariant theory'', proved that the algebra of \emph{covariants} of any binary forms is always finitely generated~\cite{Gor1868}. As a great part of the mathematical development of that time, such a result was endowed with a constructive proof: the English and the German schools were equally preoccupied by calculation and an exhibition of invariants and covariants. Despite Gordan's constructive proof, Cayley was reluctant to make use of Gordan's approach to obtain a new understanding of invariant theory. That's only in 1903, with the work of Grace--Young~\cite{GY2010}, that the German approach of Gordan and al. became accessible to a wide community of mathematicians. During that time, from 1868 to 1875, Gordan's constructive approach led to several explicit results: first, and without difficulty, Gordan~\cite{Gor1875} computed a bases for the covariants of the quintic and the sextic. Thereafter, he started the computation of a covariant bases for the septimic and the octic. This work was achieved by Von Gall who exhibited a complete covariant bases for the septimic~\cite{vGal1888} and for the octic~\cite{Gall1880}.

In 1890, Hilbert made a critical advance in the field of \emph{invariant theory}. Using a totally new approach~\cite{Hil1993}, which is the cornerstone of today's algebraic geometry, he proved a finiteness theorem in the very general case of a \emph{linear reductive group}~\cite{DK2002}. However, his first proof~\cite{Hil1993} was criticized for not being constructive~\cite{Fis1966}. Facing those critics, Hilbert produced a second proof~\cite{Hil1993}, claimed to be more constructive. This effective approach is nowadays widely used to obtain a finite generating set of invariants~\cite{Shi1967,Dix1982,BP2010,BP2010a}. Let summarize here the three main steps of Hilbert's approach~\cite{Hil1993}.

The first step is to compute the \emph{Hilbert series} of the graded algebra $\mathcal{A}$ of invariants, which is always a rational function by the Hilbert---Serre theorem~\cite{CLO2005}. This Hilbert series\footnote{There exists several methods to compute this Hilbert series~\cite{Bed2011,LP1990,Spr1983} \emph{a priori}.} gives dimensions of each homogeneous space of $\mathcal{A}$. The second step is to exhibit a \emph{homogeneous system of parameters} (hsop) for the algebra $\mathcal{A}$. Finally, the Hochster--Roberts theorem~\cite{HR1974} ensures that the invariant algebra $\mathcal{A}$ is \emph{Cohen--Macaulay}\footnote{Meaning the algebra $\mathcal{A}$ is a finite and free $k[\theta_1,\dotsc,\theta_s]$--module, where $\lbrace \theta_1,\dotsc,\theta_s\rbrace$ is a system of parameters.}. Thanks to that statement, one system of parameters (or at least the knowledge of their degree) altogether with the Hilbert series produce a bound for the degrees of generating invariants. We refer the reader to several references~\cite{Stu2008,BP2010,Dix1982,Der1999,DK2002,DK2008} to get a general and modern approach on this subject.

However, one major weakness of that strategy is that it depends on the knowledge of a system of parameters (or at least their degree). The Noether normalization lemma~\cite{Lan2002} ensures that such a system always exists, but as far as we know, current algorithms to obtain such a system~\cite{Has2008} are not sufficiently effective because of the extensive use of Grobn\"{e}r basis. For the invariant or covariant algebra of binary forms, one has of course the concept of \emph{nullcone} and the Mumford--Hilbert criterion~\cite{DK2002,Bri1996} to check that a given finite family is a system of parameters. But this criterion does not explain how to obtain a system of parameters. Furthermore, in the case of \emph{joint invariants}, that is for the invariant algebra $\inv{V}$ of $V:=\Sn{n_{1}}\oplus\cdots \oplus \Sn{n_{k}}$, such a system of parameters has, in general, a complex shape. Indeed, Brion~\cite{Bri1982} showed that there exists a system of parameters which respects the multi--graduation of $\inv{V}$ only in thirteen cases.

An important motivation for this article was to compare effective approaches in invariant theory since the goal was to compute invariant basis for non trivial joint invariants, such as $\Sn{8}\oplus\Sn{4}\oplus\Sn{4}$ or $\Sn{6}\oplus\Sn{4}\oplus\Sn{2}\oplus\Sn{2}$. Those computations have important applications in \emph{continuum mechanics}~\cite{AF1980} in which occurs invariants of tensor spaces defined on $\GR^{3}$, naturally isomorphic (after complexification) to $\sldc$ spaces of binary forms~\cite{Ste1994,BKO1994}. For instance, to obtain invariants of the \emph{elasticity tensor}~\cite{Bac1970}, Boehler--Kirilov--Onat~\cite{BKO1994} derived from the invariant bases of $\Sn{8}$ (first obtained by Von Gall~\cite{Gall1880} in 1880) a generating set of invariants for the higher dimensional irreducible component of the elasticity tensor. Such an invariant bases can be used to classify the orbit space of the elasticity tensor, as pointed out by Auffray--Kolev--Petitot~\cite{AKP2014}. In a recent paper, we used a joint invariant bases of $\Sn{6}\oplus\Sn{2}$ already obtained by von Gall~\cite{vGal1874} to obtain a new result on an invariant bases of a traceless and totally symmetric third order tensor defined on $\GR^{3}$~\cite{OA2014}. Such an invariant bases is useful in piezoelectricity~\cite{Yan2009} and second--gradient of strain elasticity theory~\cite{Min1965}.

Other interests for effective computations of generating sets of invariants of binary forms arise in \emph{geometrical arithmetic}, illustrated by the work of Lercier--Ritzenthaler~\cite{LR2012} on hyperelliptic curves. We could also cite other areas such as \emph{quantum informatics} with the paper of Luque~\cite{Luq2007} and \emph{recoupling theory}, with the work of Abdesselam and Chipalkatti~\cite{AC2007,AC2009,Abd2012,AC2012} on $6j$ and $9j$--symbols.

Approaches on \emph{effective invariant theory} do not only rely on the algebraic geometry field initially developed by Hilbert. In the case of a single binary form, Olver~\cite{Olv1999} exhibits another constructive approach, which was later generalized for a single $n$-ary form and also specified with a ``running bound'' by Brini--Regonati--Creolis~\cite{BRT2006}. We could also cite Kung--Rota~\cite{KR1984} but the combinatorial approach developed there became increasingly complex for the cases we had to deal with.

As we already noticed, a special case of Gordan's algorithm, stated in theorem~\ref{thm:CovJointsQuadr}, leads to a very easy computation of one covariant bases for $\Sn{6}\oplus\Sn{2}$. Due to this observation, we decided to reformulate Gordan's theorem\footnote{Note that Weyman~\cite{Wey1993} has also reformulated Gordan's method in a modern way and through algebraic geometry but unfortunately, we were unable to extract from it an effective approach. There is also a preprint of Pasechnik~\cite{Pas1996} on this method.} on binary forms in the modern language of operators and $\sldc$ equivariant homomorphisms. We also decided to represent $\sldc$ equivariant homomorphisms with \emph{directed graphs}, in the spirit of the graphical approach developed by Olver--Shakiban~\cite{OS1988}. 

The paper is organized as follows. In~\autoref{sec:AlgCovariants} we recall the mathematical background of classical invariant theory, and we introduce classical operators such as the Cayley operator, polarization operator and the transvectant operator. In~\autoref{sec:molecular-covariants}, we introduce \emph{molecule} and \textit{molecular covariants} which are graphical representations of $\sldc$ equivariant homomorphisms constructed with the use of Cayley and polarization operators. We then give in~\autoref{sec:TransEtMol} important relations between molecular covariants and transvectants. Gordan's algorithm for joint covariants, explained in~\autoref{sec:GordJoint}, produces a finite generating set for $\cov{\mathrm{S}_{m}\oplus \mathrm{S}_{n}}$, knowing a finite system of generators for the covariant algebra $\cov{\mathrm{S}_{m}}$ and $\cov{\mathrm{S}_{n}}$. A second version of Gordan's algorithm, which enables to compute a covariant bases for $\mathrm{S}_{n}$, knowing covariant basis for $\mathrm{S}_{k}$ ($k < n$), is detailed in~\autoref{sec:GordSimp}. We propose in~\autoref{sec:Improvement} some improvement of those two algorithms and in~\autoref{sec:Eff_Comp} we give some illustrations of that method, by (re-)computing a minimal covariant bases for $\Sn{6}\oplus\Sn{2}$ (already done by von Gall~\cite{vGal1874}). We also exhibit for the first time a minimal bases for the joint covariants of $\Sn{6}\oplus\Sn{4}$ (theorem~\ref{thm:Cov_64}), and also a minimal bases for the joint covariants of $\Sn{6}\oplus\Sn{4}\oplus\Sn{2}$ (new, theorem~\ref{thm:Cov_642}). Then we apply the algorithm for a single binary form and give a minimal covariant bases for the binary octics. Note this was already obtained by Von Gall~\cite{vGal1888}, Lercier--Ritzenthaler~\cite{LR2012}, Cr\"{o}ni~\cite{Croe2002} and Bedratyuk~\cite{Bed2008}. Finally we obtain for the first time a minimal invariant bases for $\Sn{8}\oplus\Sn{4}\oplus\Sn{4}$ (theorem~\ref{thm:Inv_Elas_Bin}). Note also that a minimal covariant bases for the binary nonics and decimics will be presented in a forthcoming paper with Lercier~\cite{OL2014}.


\section{Mathematical framework}\label{sec:AlgCovariants}

\subsection{Covariants of binary forms}

\begin{defn}
The complex vector space of $n$-th degree binary forms, noted $\Sn{n}$, is the space of homogeneous polynomials
\begin{equation*}
	\mathbf{f}(\xx):=a_0x^n+\binom{n}{1}a_1x^{n-1}y+\dotsc +\binom{n}{n-1}a_{n-1}xy^{n-1} +a_ny^n,
\end{equation*}
with $\xx:=(x,y)\in \GC^2$ and $a_i\in \mathbb{C}$.
\end{defn}

The natural $\mathrm{SL}_2(\mathbb{C})$ action on $\mathbb{C}^2$ induces a left action on $\Sn{n}$, given by
\begin{equation*}
	(g\cdot \mathbf{f})(\xx):=\mathbf{f}(g^{-1}\cdot \xx),\quad g\in \mathrm{SL}_2(\mathbb{C}).
\end{equation*}
By a space $V$ of binary forms, we mean a direct sum
\begin{equation*}
	V:=\bigoplus_{i=0}^s \Sn{n_i},\quad n_i\in \mathbb{N}
\end{equation*}
where the action of $\mathrm{SL}_2(\mathbb{C})$ is diagonal. One can also define an $\mathrm{SL}_2(\mathbb{C})$ action on the coordinate ring $\mathbb{C}[V\oplus \mathbb{C}^2]$ by
\begin{equation*}
	(g\cdot \mathbf{p})(\ff,\xx):=\mathbf{p}(g^{-1}\cdot \ff,g^{-1}\cdot \xx)  \text{ for } g\in \mathrm{SL}_2(\mathbb{C}), \, \mathbf{p}\in \mathbb{C}[V\oplus \mathbb{C}^2].
\end{equation*}

\begin{defn}
The covariant algebra\footnote{For a general and modern approach on invariant and covariant algebra, we refer to the online text~\cite{KP2000} by Kraft and Procesi.} of a space $V$ of binary forms, noted $\cov{V}$, is the invariant algebra
\begin{equation*}
	\cov{V}:=\mathbb{C}[V\oplus \mathbb{C}^2]^{\mathrm{SL}_2(\mathbb{C})}.
\end{equation*}
\end{defn}

An important result, first established by Gordan~\cite{Gor1868} and then extended by Hilbert~\cite{Hil1993} (for any linear reductive group) is the following.

\begin{thm}
For every space $V$ of binary forms, the covariant algebra $\cov{V}$ is finitely generated, i.e. there exists a finite set $\mathbf{h}_1,\dotsc,\mathbf{h}_N$ in $\cov{V}$, called a \emph{bases}, such that
\begin{equation*}
	\cov{V}=\mathbb{C}[\mathbf{h}_1,\dotsc,\mathbf{h}_N].
\end{equation*}
\end{thm}

There is a natural bi-graduation on the covariant algebra $\cov{V}$:
\begin{itemize}
	\item[$\bullet$] By the \textbf{degree}, which is the polynomial degree in the coefficients of the space $V$;
	\item[$\bullet$] By the \textbf{order} which is the polynomial degree in the variables $\xx$;
\end{itemize}

Let $\mathbf{Cov}_{d,k}(V)$ be the subspace of degree $d$ and order $k$ covariants, and:
\begin{equation*}
  \mathrm{C}_{+}:=\sum_{d+k>0} \mathbf{Cov}_{d,k}(V).
\end{equation*}
Then, $\mathrm{C}_{+}$ is an ideal of the graduated algebra $\cov{V}$. For each $d+k>0$, let $\delta_{d,k}$ be the codimension of $(\mathrm{C}_{+}^2)_{d,k}:=(\mathrm{C}_{+}^2)\cap \mathbf{Cov}_{d,k}(V)$ in $\mathbf{Cov}_{d,k}(V)$. Since the algebra $\cov{V}$ is of finite type, there exists an integer $p$ such that $\delta_{d,k}=0$ for $d+k\geq p$ and we can define the invariant number:
\begin{equation*}
	n(V)=\sum_{d,k} \delta_{d,k} .
\end{equation*}

\begin{defn}
A family $(\mathbf{p}_{1},\dotsc,\mathbf{p}_{s})$ is a \emph{minimal} bases of $\cov{V}$ if its image in the vector space $\mathrm{C}_{+}/\mathrm{C}_{+}^2$ is a bases. In that case we have $s=n(V)$.
\end{defn}

\begin{rem} 
As pointed out by Dixmier--Lazard~\cite{DL1985/86}, a minimal bases is obtained by taking, for each $d,k$, a complement bases of $(\mathrm{C}_{+}^2)_{d,k}$ in $\mathbf{Cov}_{d,k}(V)$. There is a long history of an explicit determination of such a minimal bases for covariant algebras. We give in~\autoref{tab:Min_Bas} some results\footnote{The website \url{http://www.win.tue.nl/$\sim$aeb/math/invar.html} gives a general overview on those results.} obtained from \century{19} to \century{21}. As we know, there is no way to get the invariant $n(V)$ but to exhibit an explicit minimal bases of $\cov{V}$. 
\begin{table}[H]
\begin{tabular}{|c|c|c|}
\hline 
Algebra & $n(V)$ & Explicit minimal bases \\ 
\hline 
$\cov{\Sn{5}}$ & $23$ & Gordan~\cite{Gor1868} \\ 
$\cov{\Sn{6}}$ & $26$ & Gordan~\cite{Gor1868} \\
$\cov{\Sn{7}}$ & $147$ & Dixmier--Lazard~\cite{DL1985/86} \\
$\cov{\Sn{8}}$ & $69$ & Cr\"oni~\cite{Croe2002} \\
$\cov{\Sn{6}\oplus\Sn{2}}$ & $99$ & Von Gall~\cite{vGal1888} \\
\hline 
\end{tabular} 
\caption{Minimal basis of covariant algebras.}\label{tab:Min_Bas}
\end{table}
\end{rem}

\subsection{Bidifferential operators and transvectants}

Recall that $\Sn{n}$ is an \emph{irreducible} $\sldc$ representation~\cite{FH1991}. The \emph{Clebsch--Gordan decomposition}~\cite{FH1991} of a tensor product is the $\sldc$ irreducible decomposition  
\begin{equation*}\label{eq:CB_Decomp}
	\Sn{n}\otimes\Sn{p}\simeq \bigoplus_{r=0}^{\min(n,p)} \Sn{n+p-2r}.
\end{equation*}
We then deduce that, for each $0\leq r\leq \min(n,p)$, there is only one (up to a scale factor) \emph{Clebsch-Gordan projector}
\begin{equation*}\label{def:CG_projecteur}
	\pi_{r}: \Sn{n}\otimes \Sn{p}\longrightarrow \Sn{n+p-2r}, \quad \ff\otimes \bg \mapsto \tr{\mathbf{f}}{\mathbf{g}}{r}:=\pi_r(\ff\otimes \bg).
\end{equation*}

Such a projector is called a \emph{transvectant}. To have an explicit formula for transvectants, we use bi-differential operators:
\begin{itemize}
\item[$\bullet$] the \emph{Cayley operator}~\cite{Olv1999}, which is a bi-differential operator acting on the tensor product of complex analytic functions $\mathbf{f}(\xx_{\alpha})\bg(\xx_{\beta})$:
\begin{equation*}
\Omega_{\alpha\beta}(\mathbf{f}(\xx_{\alpha})\bg(\xx_{\beta})):=\frac{\partial \mathbf{f}}{x_{\alpha}}\frac{\partial \mathbf{g}}{y_{\beta}}-\frac{\partial \mathbf{f}}{y_{\alpha}}\frac{\partial \mathbf{g}}{x_{\beta}};
\end{equation*}
\item[$\bullet$] the \emph{polarization operator}\footnote{This operator is called \emph{scalling process} in~\cite{Olv1999}.} acting on a complex analytic function $\mathbf{f}(\xx_{\alpha})$:
\begin{equation*}
	\sigma_{\alpha}(\ff(\xx_{\alpha}))=x\frac{\partial \ff}{\partial x_{\alpha}}+y\frac{\partial \ff}{\partial y_{\alpha}}.
\end{equation*}
\end{itemize}

Both the Cayley and polarization operators commute with the $\mathrm{SL}_2(\mathbb{C})$ action (see~\cite{Olv1999} for instance).

\begin{defn}
\label{def:Transvectant}
Given two binary forms $\ff\in \Sn{n}$ and $\bg\in \Sn{p}$, their \emph{transvectant} of index $r\geq 0$, noted $\tr{\ff}{\mathbf{g}}{r}$, is defined to be
\begin{equation*}
	\tr{\mathbf{f}}{\mathbf{g}}{r}:=
	\begin{cases}
	\Omega_{\alpha\beta}^r\sigma_{\alpha}^{n-r}\sigma_{\beta}^{p-r}(\mathbf{f}(\xx_{\alpha})\mathbf{g}(\xx_{\beta})) \text{ if } 0\leq r\leq \min (n,p) \\
	0 \text{ else}
	\end{cases}.
\end{equation*}
\end{defn}

\begin{rem}
Note that the definition given in~\cite{Olv1999} uses a scale factor and a trace operator:
\begin{align*}
	\tr{\mathbf{f}}{\mathbf{g}}{r}=(n-r)!(p-r)!\left[\Omega_{\alpha\beta}^{r}\ff(\xx_{\alpha})\bg(\xx_{\beta})\right]_{\vert \xx_{\alpha}=\xx_{\beta}=\xx}.
\end{align*}
On the other hand, Gordan's definition~\cite{GY2010} corresponds to
\begin{equation*}
	\frac{1}{n!}\frac{1}{p!}\tr{\mathbf{f}}{\mathbf{g}}{r}.
\end{equation*}
This last expression is very simple when applied to \emph{powers of linear forms}. Indeed, if
\begin{equation*}
	\mathbf{a}_{\xx_{\alpha}}^{n}:=(a_{0}x_{\alpha}+a_{1}y_{\alpha})^{n},\quad \mathbf{b}_{\xx_{\beta}}^{p}:=(b_{0}x_{\beta}+b_{1}y_{\beta})^{p},\quad (ab):=a_{0}b_{1}-a_{1}b_{0},
\end{equation*}
then,
\begin{equation*} 
	\frac{1}{n!}\frac{1}{p!}\tr{\mathbf{a}_{\xx_{\alpha}}^{n}}{\mathbf{b}_{\xx_{\beta}}^{p}}{r}=(ab)^{r}\mathbf{a}_{\xx}^{n-r}\mathbf{b}_{\xx}^{p-r}.
\end{equation*}
Our choice of definition~\ref{def:Transvectant} has the advatange of inducing simple relations on operators and thus on transvectants (see~\ref{Rel:Syz2} for instance).
\end{rem}

\begin{rem}
Take a space of binary forms 
\begin{equation*}
	V=\Sn{n_1}\oplus \dotsc \oplus\Sn{n_s}
\end{equation*}
and consider the set $\mathcal{T}$ containing each $\ff_{i}\in \Sn{n_i}$ and closed under tranvectant operations:
\begin{equation*}
	\mathbf{f}\in \mathcal{T},\mathbf{g}\in \mathcal{T} \Rightarrow \tr{\ff}{\bg}{r}\in \mathcal{T},\quad \forall r\in \mathbb{N}.
\end{equation*}
Then as a classical result~\cite{Pro1998} the covariant algebra $\cov{V}$ is generated by the (infinite) set $\mathcal{T}$. One important issue is then to extract a finite family from that infinite set.  
\end{rem}


\section{Molecular covariants}
\label{sec:molecular-covariants}

Let $\Sym^{d}(V)$ be the space of totally symmetric tensors of order $d$ on $V$. The \emph{Aronhold polarization} induces an isomorphism~\cite{DC1970} between $\mathbf{Cov}_{d,k}(V)$ and the space
\begin{equation*}
	\Hom_{\sldc}(\Sym^{d}(V),\Sn{k})\subset \Hom_{\sldc}(\otimes^{d} V,\Sn{k}).
\end{equation*}
Transvectants, Cayley operator and polarization operator give natural way to obtain $\sldc$--equivariant homomorphisms.
We already saw (definition \ref{def:Transvectant}) that the Clebsch--Gordan projector 
\begin{equation*}
	\pi_{r}: \Sn{n}\otimes \Sn{p}\longrightarrow \Sn{n+p-2r}, 
\end{equation*}
can be written as 
\begin{equation*}
	\Omega_{\alpha\beta}^r\sigma_{\alpha}^{n-r}\sigma_{\beta}^{p-r}.
\end{equation*}
Such a monomial will be represented by the colored directed graph (colored digraph)\footnote{It is important to note that a digraph represents here a morphism and not a bi-differential operator as did Olver--Shakiban~\cite{OS1989}.}:
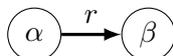
\begin{center}
	\begin{tikzpicture}[baseline={([yshift=-.5ex]current bounding box.center)}]
		\node[m] (P)at(0.2,0.1){$\alpha$};
		\node[m] (Q)at(1.7,0.1){$\beta$};
		\draw[flecheo] (P)--(Q) node[midway,above] {$r$};
	\end{tikzpicture}
\end{center}
where the \emph{atom} $\alpha$ (resp. $\beta$) is colored by $\Sn{n}$ (resp. $\Sn{p}$).

More generaly, let $V= \Sn{n_{1}}\oplus\cdots \oplus\Sn{n_{s}}$ be a space of binary forms. We are going to define equivariant multilinear maps from $V$ to some $\Sn{k}$, corresponding to monomials in the symbols $\Omega_{\alpha\beta},\sigma_{\gamma},\dotsc$ and labelled by \emph{molecules} (colored digraphs). 

More precisely, let $\mathcal{V}(\di{D})=\set{\alpha,\beta,\dotsc,\varepsilon}$  be the set of vertices of a colored digraph $\di{D}$ and $\mathcal{E}(\di{D})$  be its set of edges. Each vertex $\alpha$ of $\di{D}$, also called an \emph{atom}, is colored by a factor $\mathrm{S}(\alpha):=\Sn{n_i}$ of $V$. In that case, the \emph{valence} of $\alpha$ is $\Val(\alpha):=n$. Define $o(e)$, $t(e)$ and $w(e)$ to be respectively the origin, the termination and the weight of an edge $e\in \mathcal{E}(\di{D})$. Finally, we define the \emph{free valence} $\Val_{\di{D}}(\alpha)$ of an atom $\alpha\in \mathcal{V}(\di{D})$ to be:
\begin{equation*}
	\Val_{\di{D}}(\alpha):=\Val(\alpha)-\sum_{\alpha=o(e)\text{ or } \alpha=t(e)} w(e).
\end{equation*}

\begin{defn}\label{def:MolAron}
The $\sldc$--equivariant homomorphism $\phi_{\di{D}}$ defined by the molecule $\di{D}$ is given by
\begin{equation*}
\phi_{\di{D}}:=
\begin{cases}
\prod_{e \in \mathcal{E}(\mathsf{D})} \Omega_{o(e) \; t(e)}^{w(e)}\prod_{\alpha \in \mathcal{V}(\mathsf{D})}\sigma_{\alpha}^{\Val_{\mathsf{D}}(\alpha)} \text{ if } \Val_{\mathsf{D}}(\alpha)\geq 0,\quad \forall \alpha \in \mathcal{V}(\mathsf{D}) \\
0 \text{ else} 
\end{cases}.
\end{equation*}
When $\Val_{\mathsf{D}}(\alpha)\geq 0$ for all $\alpha \in \mathcal{V}(\mathsf{D})$, it maps $\mathrm{S}(\alpha)\otimes \cdots \otimes\mathrm{S}(\varepsilon)$ to $\Sn{k}$, where $k=\Val_{\di{D}}(\alpha)+\dotsc+\Val_{\di{D}}(\epsilon)$.
\end{defn}

There exists \emph{syzygies} on morphisms $\phi_{\di{D}}$ induced by fundamental relations among operators. Let $\alpha$, $\beta$, $\gamma$ and $\delta$ be four atoms.
\begin{enumerate}
\item The first syzygy derives from the equality
\begin{equation*}
\Omega_{\alpha\beta} =-\Omega_{\beta\alpha},
\end{equation*}
which leads to the graphical relation:
\begin{align}\label{Rel:Syz1}
	\begin{tikzpicture}[baseline]
			\node[m] (P)at(0.2,0.1){$\alpha$};
			\node[m] (Q)at(1.7,0.1) {$\beta$};
			\draw[flecheo] (P)--(Q) ;
	\end{tikzpicture}
	&=-\:
	\begin{tikzpicture}[baseline]
		\node[m] (P)at(0.3,0.1){$\alpha$};
		\node[m] (Q)at(1.8,0.1) {$\beta$};
		\draw[flecheo] (Q)--(P) ;
	\end{tikzpicture}
\end{align}
\item The second one comes from the Pl\"{u}cker relation~\cite{Olv1999}:
\begin{equation}\label{eq:DeterminantOmega}
	\Omega_{\alpha\beta}\sigma_{\gamma}=
	\Omega_{\alpha\gamma}\sigma_{\beta}
	+
	\Omega_{\gamma\beta}\sigma_{\alpha},
\end{equation}
which leads to the graphical relation:
\begin{align}\label{Rel:Syz2}
	\begin{tikzpicture}[scale=1.2,baseline={([yshift=-.5ex]current bounding box.center)}]
		\node[m] (P)at(0.2,1.1){$\alpha$};
		\node[m] (Q)at(1.7,1.1){$\beta$};
		\node[m] (R)at(0.95,0.1){$\gamma$};
		\draw[flecheo] (P)--(Q);
	\end{tikzpicture}
	&
	=
	\begin{tikzpicture}[scale=1.2,baseline={([yshift=-.5ex]current bounding box.center)}]
		\node[m] (P)at(0.2,1.1){$\alpha$};
		\node[m] (Q)at(1.7,1.1){$\beta$};
		\node[m] (R)at(0.95,0.1){$\gamma$};
		\draw[flecheo] (P)--(R);
	\end{tikzpicture}
	+
	\begin{tikzpicture}[scale=1.2,baseline={([yshift=-.5ex]current bounding box.center)}]
		\node[m] (P)at(0.2,1.1){$\alpha$};
		\node[m] (Q)at(1.7,1.1){$\beta$};
		\node[m] (R)at(0.95,0.1){$\gamma$};
		\draw[flecheo] (R)--(Q);
	\end{tikzpicture}
\end{align}
\item The third one derives also from a Pl\"{u}cker relation, namely
\begin{equation*}
  \Omega_{\alpha\beta}\Omega_{\gamma\delta} = \Omega_{\alpha\delta}\Omega_{\beta\gamma} + \Omega_{\alpha\gamma}\Omega_{\delta\beta},
\end{equation*}
which leads to the graphical relation:
\begin{align}\label{Rel:Syz3}
	\begin{tikzpicture}[scale=1.2,baseline={([yshift=-.5ex]current bounding box.center)}]
		\node[m] (P)at(0.2,1.2){$\alpha$};
		\node[m] (Q)at(1.7,1.2){$\beta$};
		\node[m] (R)at(1.7,0){$\gamma$};
		\node[m] (S)at(0.2,0){$\delta$};
		\draw[flecheo] (P)--(Q);
		\draw[flecheo] (S)--(R);
	\end{tikzpicture}
	&=
	\begin{tikzpicture}[scale=1.2,baseline={([yshift=-.5ex]current bounding box.center)}]
		\node[m] (P)at(0.2,1.2){$\alpha$};
		\node[m] (Q)at(1.7,1.2){$\beta$};
		\node[m] (R)at(1.7,0){$\gamma$};
		\node[m] (S)at(0.2,0){$\delta$};
		\draw[flecheo] (P)--(S);
		\draw[flecheo] (Q)--(R);
	\end{tikzpicture}
	+
	\begin{tikzpicture}[scale=1.2,baseline={([yshift=-.5ex]current bounding box.center)}]
		\node[m] (P)at(0.2,1.2){$\alpha$};
		\node[m] (Q)at(1.7,1.2){$\beta$};
		\node[m] (R)at(1.7,0){$\gamma$};
		\node[m] (S)at(0.2,0){$\delta$};
		\draw[flecheo] (P)--(R);
		\draw[flecheo] (S)--(Q);
	\end{tikzpicture}
\end{align}
\end{enumerate}

\begin{rem}
By syzygie~\ref{Rel:Syz1}, we have 
\begin{align*}
	\begin{tikzpicture}[baseline]
			\node[m] (P)at(0.2,0.1){$\alpha$};
			\node[m] (Q)at(1.7,0.1) {$\beta$};
			\draw[flecheo] (P)--(Q) node[midway,above] {$2$};
	\end{tikzpicture}
	&=
	\begin{tikzpicture}[baseline]
		\node[m] (P)at(0.3,0.1){$\alpha$};
		\node[m] (Q)at(1.8,0.1) {$\beta$};
		\draw[flecheo] (Q)--(P) node[midway,above] {$2$};
	\end{tikzpicture}
\end{align*}
thus for even weighed edges, we will not specify orientation: 
\begin{align*}
	\begin{tikzpicture}[baseline]
			\node[m] (P)at(0.2,0.1){$\alpha$};
			\node[m] (Q)at(1.7,0.1) {$\beta$};
			\draw[flecheno] (P)--(Q) node[midway,above] {$2$};
	\end{tikzpicture}
	&:=
	\begin{tikzpicture}[baseline]
		\node[m] (P)at(0.3,0.1){$\alpha$};
		\node[m] (Q)at(1.8,0.1) {$\beta$};
		\draw[flecheo] (P)--(Q) node[midway,above] {$2$};
	\end{tikzpicture}
\end{align*}
\end{rem}

For each atom $\alpha\in \mathcal{V}(\mathsf{D})$, let $\ff_{\alpha}\in \mathrm{S}(\alpha)$ and consider one covariant
\begin{equation*}
\label{defn:Molecules_Et_Covariants}
\phi_{\di{D}}\left( \bigotimes_{\alpha \in \mathcal{V}(\mathsf{D})} \ff_{\alpha} \right)\in \cov{V}.
\end{equation*}
This defines a map from the set of molecules to $\cov{V}$. A \emph{molecular covariant} $\Mdi{D}$ is then defined to be an image of a molecule by this map, and in that case a binary form $\ff_{\alpha}\in \mathrm{S}(\alpha)=\Sn{n}$ is said to be an \emph{atom} of valence $n$ in $\Mdi{D}$. The following result is known as the first fundamental theorem for binary forms~\cite{KR1984,Olv1999}.
\begin{thm}\label{thm:FFT}
Given a space $V=\Sn{n_1}\oplus \dotsc \oplus\Sn{n_s}$ of binary forms, the covariant algebra $\cov{V}$ is generated by the (infinite) family of molecular covariants.
\end{thm}


\section{Transvectants on molecular covariants}
\label{sec:TransEtMol}

First observe that a transvectant $\tr{\ff_{\alpha}}{\ff_{\beta}}{r}$ is represented by a simple molecular covariant:
\begin{equation*}
		\begin{tikzpicture}[baseline={([yshift=-.5ex]current bounding box.center)}]
		\node[m] (P)at(0.2,0.1){$\ff_{\alpha}$};
		\node[m] (Q)at(1.7,0.1){$\ff_{\beta}$};
		\draw[flecheo] (P)--(Q) node[midway,above] {$r$};
	\end{tikzpicture}
\end{equation*}

Now, to obtain general relations between iterated transvectants and molecular covariants, we need to specify some operations on molecular covariants.

\begin{defn}\label{def:Rlink}
Let $\Mdi{D}$ and $\Mdi{E}$ be two molecular covariants. Let $r \geq 0$ be an integer and $\nu(r)$ be a symbol, we define the molecular covariant $\mathbf{M}^{\nu(r)}$, graphically noted
\begin{center}
	\begin{tikzpicture}[scale=1.5,baseline={([yshift=.1ex]current bounding box.center)}]
		\node[ellipse,draw,minimum width=40pt] (P)at(0.2,0.1){$\Mdi{D}$};
		\node[ellipse,draw,minimum width=40pt] (Q)at(2,0.1){$\Mdi{E}$};
		\draw[flecheo] (P)--(Q) node[midway,above] {$\nu(r)$};
	\end{tikzpicture}
\end{center}
to be a new molecular covariant obtained by linking $\Mdi{D}$ and $\Mdi{E}$ with $r$ edges in a given way $\nu(r)$.
\end{defn}

\begin{exemple}
Given atoms $\ff_{\alpha},\dotsc,\ff_{\epsilon}$ of valence greater than $4$, let
\begin{align*}
\Mdi{D}=
	\begin{tikzpicture}[scale=1.5,baseline={([yshift=-0.2cm]current bounding box.center)}]
		\node[m] (P)at(1.9,0.1){$\ff_{\beta}$};
		\node[m] (Q)at(3.4,0.1){$\ff_{\gamma}$};
		\node[m] (R)at(0.4,0.1){$\ff_{\alpha}$};
		\draw[flecheno] (P)--(Q) node[midway,above] {$2$};
		\draw[flecheo] (P)--(R);			
	\end{tikzpicture}
	&\text{ and }
\Mdi{E}=
	\begin{tikzpicture}[scale=1.5,baseline={([yshift=-.5ex]current bounding box.center)}]
			\node[m] (P)at(0.2,0.1){$\ff_{\delta}$};
			\node[m] (Q)at(1.7,0.1){$\ff_{\epsilon}$};
			\draw[flecheo] (P)--(Q);
	\end{tikzpicture}
\end{align*}
we can define
\begin{align*}
	\begin{tikzpicture}[scale=1.5,baseline={([yshift=-0.3cm]current bounding box.center)}]
		\node[ellipse,draw,minimum width=40pt] (P)at(0.2,0.1){$\Mdi{D}$};
		\node[ellipse,draw,minimum width=40pt] (Q)at(2,0.1){$\Mdi{E}$};
		\draw[flecheo] (P)--(Q) node[midway,above] {$\nu_1(2)$};
	\end{tikzpicture}
	&
	=\begin{tikzpicture}[scale=1.5,baseline={([yshift=-.5ex]current bounding box.center)}]
		\node[m] (P)at(1.9,0.1){$\ff_{\beta}$};
		\node[m] (Q)at(3.4,0.1){$\ff_{\gamma}$};
		\node[m] (R)at(0.4,0.1){$\ff_{\alpha}$};
		\node[m] (S)at(0.9,-1.1){$\ff_{\delta}$};
		\node[m] (T)at(2.6,-1.1){$\ff_{\epsilon}$};
		\draw[flecheno] (P)--(Q) node[midway,above] {$2$};
		\draw[flecheo] (P)--(R);			
		\draw[flecheo] (S)--(T);
		\draw[flecheno] (R)--(S) node[pos=0.5,left] {$2$};
	\end{tikzpicture}
\end{align*}
or
\begin{align*}
	\begin{tikzpicture}[scale=1.5,baseline={([yshift=-0.3cm]current bounding box.center)}]
		\node[ellipse,draw,minimum width=40pt] (P)at(0.2,0.1){$\Mdi{D}$};
		\node[ellipse,draw,minimum width=40pt] (Q)at(2,0.1){$\Mdi{E}$};
		\draw[flecheo] (P)--(Q) node[midway,above] {$\nu_2(2)$};
	\end{tikzpicture}
	&=
	\begin{tikzpicture}[scale=1.5,baseline={([yshift=-.5ex]current bounding box.center)}]
		\node[m] (P)at(1.9,0.1){$\ff_{\beta}$};
		\node[m] (Q)at(3.4,0.1){$\ff_{\gamma}$};
		\node[m] (R)at(0.4,0.1){$\ff_{\alpha}$};
		\node[m] (S)at(0.9,-1.1){$\ff_{\delta}$};
		\node[m] (T)at(2.6,-1.1){$\ff_{\epsilon}$};
		\draw[flecheno] (P)--(Q) node[midway,above] {$2$};
		\draw[flecheo] (P)--(R);			
		\draw[flecheo] (S)--(T);
		\draw[flecheo] (R)--(S);
		\draw[flecheo] (R)--(T);
	\end{tikzpicture}
\end{align*}
\end{exemple}

By a direct application of Leibnitz formula, we have~\cite{Olv1999}:

\begin{prop}\label{prop:DecompTrans}
Let $\Mdi{D},\Mdi{E}$ be two molecular covariants and $r\geq 0$ be an integer. Then the transvectant $\tr{\Mdi{D}}{\Mdi{E}}{r}$ is a linear combination of molecular covariants\footnote{The covariant $\mathbf{M}^{\nu(r)}$ is called a \emph{term} in~\cite{GY2010}.} $\mathbf{M}^{\nu(r)}$ with rational positive coefficients\footnote{There is explicit expression of those coefficients in \cite{Olive2014}.}, for each possible link $\nu(r)$ between $\Mdi{D}$ and $\Mdi{E}$:
\begin{align}\label{eq:DecompTrans}
\tr{\Mdi{D}}{\Mdi{E}}{r}
&=
\sum_{\nu(r)} a_{\nu(r)}\mathbf{M}^{\nu(r)},\quad a_{\nu(r)}\in \mathbb{Q}^{+}.
\end{align}
\end{prop}

\begin{exemple}
Let $\ff_{\alpha},\dotsc,\ff_{\delta}$ be atoms of valence greater than $4$,
\begin{align*}
	\Mdi{D}
	&=
	\begin{tikzpicture}[scale=1.5,baseline={([yshift=-0.2cm]current bounding box.center)}]
		\node[m] (P)at(1.9,0.1){$\ff_{\beta}$};
		\node[m] (Q)at(3.4,0.1){$\ff_{\gamma}$};
		\node[m] (R)at(0.4,0.1){$\ff_{\alpha}$};
		\draw[flecheno] (P)--(Q) node[midway,above] {$2$};
		\draw[flecheo] (P)--(R);			
	\end{tikzpicture}
	\text{ and }
	\Mdi{E}=
	\begin{tikzpicture}[scale=1.5,baseline={([yshift=-.5ex]current bounding box.center)}]
			\node[m] (P)at(0.2,0.1){$\ff_{\delta}$};
	\end{tikzpicture}
\end{align*}
We have thus:
\begin{align*}
	\tr{\Mdi{D}}{\Mdi{E}}{2}
	&=a_{\nu(1)}
	\begin{tikzpicture}[scale=1.2,baseline={([yshift=-0.2cm]current bounding box.center)}]
		\node[m] (P)at(1.9,0.1){$\ff_{\beta}$};
		\node[m] (Q)at(3.4,0.1){$\ff_{\gamma}$};
		\node[m] (R)at(0.4,0.1){$\ff_{\alpha}$};
		\node[m] (S)at(1.9,-1.1){$\ff_{\delta}$};
		\draw[flecheno] (P)--(Q) node[midway,above] {$2$};
		\draw[flecheo] (P)--(R);			
		\draw[flecheno] (R)--(S) node[midway,below] {$2$};
	\end{tikzpicture}
	+a_{\nu(2)}
	\begin{tikzpicture}[scale=1.2,baseline={([yshift=-0.2cm]current bounding box.center)}]
		\node[m] (P)at(1.9,0.1){$\ff_{\beta}$};
		\node[m] (Q)at(3.4,0.1){$\ff_{\gamma}$};
		\node[m] (R)at(0.4,0.1){$\ff_{\alpha}$};
		\node[m] (S)at(1.9,-1.1){$\ff_{\delta}$};
		\draw[flecheno] (P)--(Q) node[midway,above] {$2$};
		\draw[flecheo] (P)--(R);			
		\draw[flecheno] (P)--(S) node[midway,right] {$2$};
	\end{tikzpicture}\\	
	&+a_{\nu(3)}
	\begin{tikzpicture}[scale=1.2,baseline={([yshift=-0.2cm]current bounding box.center)}]
		\node[m] (P)at(1.9,0.1){$\ff_{\beta}$};
		\node[m] (Q)at(3.4,0.1){$\ff_{\gamma}$};
		\node[m] (R)at(0.4,0.1){$\ff_{\alpha}$};
		\node[m] (S)at(1.9,-1.1){$\ff_{\delta}$};
		\draw[flecheno] (P)--(Q) node[midway,above] {$2$};
		\draw[flecheo] (P)--(R);			
		\draw[flecheno] (Q)--(S) node[midway,below] {$2$};
	\end{tikzpicture}
	+a_{\nu(4)}
	\begin{tikzpicture}[scale=1.2,baseline={([yshift=-0.2cm]current bounding box.center)}]
		\node[m] (P)at(1.9,0.1){$\ff_{\beta}$};
		\node[m] (Q)at(3.4,0.1){$\ff_{\gamma}$};
		\node[m] (R)at(0.4,0.1){$\ff_{\alpha}$};
		\node[m] (S)at(1.9,-1.1){$\ff_{\delta}$};
		\draw[flecheno] (P)--(Q) node[midway,above] {$2$};
		\draw[flecheo] (P)--(R);			
		\draw[flecheo] (R)--(S);
		\draw[flecheo] (P)--(S);
	\end{tikzpicture}\\		
	&+a_{\nu(5)}
	\begin{tikzpicture}[scale=1.2,baseline={([yshift=-0.2cm]current bounding box.center)}]
		\node[m] (P)at(1.9,0.1){$\ff_{\beta}$};
		\node[m] (Q)at(3.4,0.1){$\ff_{\gamma}$};
		\node[m] (R)at(0.4,0.1){$\ff_{\alpha}$};
		\node[m] (S)at(1.9,-1.1){$\ff_{\delta}$};
		\draw[flecheno] (P)--(Q) node[midway,above] {$2$};
		\draw[flecheo] (P)--(R);			
		\draw[flecheo] (R)--(S);
		\draw[flecheo] (Q)--(S);
	\end{tikzpicture}		
	+a_{\nu(6)}
	\begin{tikzpicture}[scale=1.2,baseline={([yshift=-0.2cm]current bounding box.center)}]
		\node[m] (P)at(1.9,0.1){$\ff_{\beta}$};
		\node[m] (Q)at(3.4,0.1){$\ff_{\gamma}$};
		\node[m] (R)at(0.4,0.1){$\ff_{\alpha}$};
		\node[m] (S)at(1.9,-1.1){$\ff_{\delta}$};
		\draw[flecheno] (P)--(Q) node[midway,above] {$2$};
		\draw[flecheo] (P)--(R);			
		\draw[flecheo] (P)--(S);
		\draw[flecheo] (Q)--(S);
	\end{tikzpicture}	
\end{align*}
\end{exemple}

\begin{defn}
Given a molecular covariant $\Mdi{D}$, and an integer $k\ge 0$, we define\footnote{This operation is called \textit{convolution} in~\cite{GY2010}.} $\overline{\Mdi{D}}^{\mu(k)}$ to be the molecular covariant obtained by adding $k$ edges on $\Mdi{D}$ in a certain way $\mu(k)$.
\end{defn}

\begin{exemple}
Given atoms $\ff_{\alpha},\ff_{\beta},\ff_{\gamma}$ of valence greater than $4$ and the molecular covariant
\begin{align*}
	\Mdi{D}
	&=	
	\begin{tikzpicture}[scale=1.5,baseline={([yshift=-.5ex]current bounding box.center)}]
		\node[m] (P)at(0.2,1.1){$\ff_{\alpha}$};
		\node[m] (Q)at(1.7,1.1){$\ff_{\beta}$};
		\node[m] (R)at(0.95,0.1){$\ff_{\gamma}$};
		\draw[flecheo] (P)--(Q) node[midway,above] {$2$};
		\draw[flecheo] (P)--(R) ;
	\end{tikzpicture}
\end{align*}
we can consider
\begin{align*}
	\overline{\Mdi{D}}^{\mu_1(2)}
	&=	
	\begin{tikzpicture}[scale=1.5,baseline={([yshift=-.5ex]current bounding box.center)}]
		\node[m] (P)at(0.2,1.1){$\ff_{\alpha}$};
		\node[m] (Q)at(1.7,1.1){$\ff_{\beta}$};
		\node[m] (R)at(0.95,0.1){$\ff_{\gamma}$};
		\draw[flecheo] (P)--(Q) node[midway,above] {$3$};
		\draw[flecheo] (P)--(R) ;
		\draw[flecheo] (Q)--(R) ;		
	\end{tikzpicture}
	\text{ or }
	\overline{\Mdi{D}}^{\mu_2(2)}=
	\begin{tikzpicture}[scale=1.5,baseline={([yshift=-.5ex]current bounding box.center)}]
		\node[m] (P)at(0.2,1.1){$\ff_{\alpha}$};
		\node[m] (Q)at(1.7,1.1){$\ff_{\beta}$};
		\node[m] (R)at(0.95,0.1){$\ff_{\gamma}$};
		\draw[flecheo] (P)--(Q) node[midway,above] {$2$};
		\draw[flecheo] (P)--(R) ;
		\draw[flecheo] (Q)--(R) node[midway,right] {$2$} ;		
	\end{tikzpicture}
\end{align*}
\end{exemple}

\begin{prop}\label{prop:TermAndTrans}
Let $\Mdi{D},\Mdi{E}$ be two molecular covariants and $r\geq 0$ be an integer. Then for every molecular covariant $\mathbf{M}^{\nu(r)}$ in the decomposition~\ref{eq:DecompTrans} of $\tr{\Mdi{D}}{\Mdi{E}}{r}$, we have
\begin{equation*}
	\mathbf{M}^{\nu(r)}=\lambda\tr{\Mdi{D}}{\Mdi{E}}{r}+\sum_{k_1,k_2,r'}
\lambda_{k_1,k_2,r'}	\tr{
	\overline{\Mdi{D}}^{\mu_1(k_1)}}{\overline{\Mdi{E}}^{\mu_2(k_2)}}
	{r'},\quad \lambda>0
\end{equation*}
with $k_1+k_2+r'=r$ being constant and $r'<r$.
\end{prop}

\begin{proof}[Sketch of proof]
This proof is based on induction on $r$. When $r=1$ take a molecular covariant $\mathbf{M}^{\nu(1)}$ in $\tr{\Mdi{D}}{\Mdi{E}}{1}$. In this molecular covariant, there is a link between an atom $\ff_{\alpha_{1}}$ in $\Mdi{D}$ and an atom $\ff_{\beta_{1}}$ in $\Mdi{E}$. Let $\mathbf{M}^{\mu(1)}$ be another molecular covariant in $\tr{\Mdi{D}}{\Mdi{E}}{1}$, with a link between an atom $\ff_{\alpha_{2}}\neq\ff_{\alpha_{1}}$ in $\Mdi{D}$ and an atom $\ff_{\beta_{1}}\neq\ff_{\beta_{2}}$ in $\Mdi{E}$. By relation~\ref{Rel:Syz2} we have
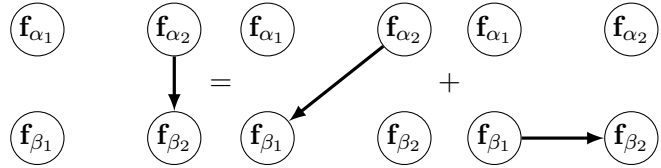
\begin{align*}
	\begin{tikzpicture}[scale=1.2,baseline={([yshift=-.5ex]current bounding box.center)}]
		\node[m] (P)at(0.2,1.2){$\ff_{\alpha_{1}}$};
		\node[m] (Q)at(1.7,1.2){$\ff_{\alpha_{2}}$};
		\node[m] (R)at(1.7,0){$\ff_{\beta_{2}}$};
		\node[m] (S)at(0.2,0){$\ff_{\beta_{1}}$};
		\draw[flecheo] (Q)--(R);
	\end{tikzpicture}
	&=
	\begin{tikzpicture}[scale=1.2,baseline={([yshift=-.5ex]current bounding box.center)}]
		\node[m] (P)at(0.2,1.2){$\ff_{\alpha_{1}}$};
		\node[m] (Q)at(1.7,1.2){$\ff_{\alpha_{2}}$};
		\node[m] (R)at(1.7,0){$\ff_{\beta_{2}}$};
		\node[m] (S)at(0.2,0){$\ff_{\beta_{1}}$};
		\draw[flecheo] (Q)--(S);
	\end{tikzpicture}	+
	\begin{tikzpicture}[scale=1.2,baseline={([yshift=-.5ex]current bounding box.center)}]
		\node[m] (P)at(0.2,1.2){$\ff_{\alpha_{1}}$};
		\node[m] (Q)at(1.7,1.2){$\ff_{\alpha_{2}}$};
		\node[m] (R)at(1.7,0){$\ff_{\beta_{2}}$};
		\node[m] (S)at(0.2,0){$\ff_{\beta_{1}}$};
		\draw[flecheo] (S)--(R);
	\end{tikzpicture}
\end{align*}
where the last molecular covariant is a transvectant $\tr{\Mdi{D}}{\overline{\Mdi{E}}^{1}}{0}$. By the same relation~\ref{Rel:Syz2}:
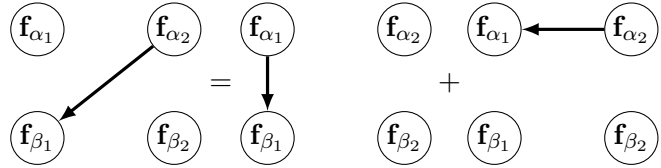
\begin{align*}
	\begin{tikzpicture}[scale=1.2,baseline={([yshift=-.5ex]current bounding box.center)}]
		\node[m] (P)at(0.2,1.2){$\ff_{\alpha_{1}}$};
		\node[m] (Q)at(1.7,1.2){$\ff_{\alpha_{2}}$};
		\node[m] (R)at(1.7,0){$\ff_{\beta_{2}}$};
		\node[m] (S)at(0.2,0){$\ff_{\beta_{1}}$};
		\draw[flecheo] (Q)--(S);
	\end{tikzpicture}
	&=
	\begin{tikzpicture}[scale=1.2,baseline={([yshift=-.5ex]current bounding box.center)}]
		\node[m] (P)at(0.2,1.2){$\ff_{\alpha_{1}}$};
		\node[m] (Q)at(1.7,1.2){$\ff_{\alpha_{2}}$};
		\node[m] (R)at(1.7,0){$\ff_{\beta_{2}}$};
		\node[m] (S)at(0.2,0){$\ff_{\beta_{1}}$};
		\draw[flecheo] (P)--(S);
	\end{tikzpicture}	+
	\begin{tikzpicture}[scale=1.2,baseline={([yshift=-.5ex]current bounding box.center)}]
		\node[m] (P)at(0.2,1.2){$\ff_{\alpha_{1}}$};
		\node[m] (Q)at(1.7,1.2){$\ff_{\alpha_{2}}$};
		\node[m] (R)at(1.7,0){$\ff_{\beta_{2}}$};
		\node[m] (S)at(0.2,0){$\ff_{\beta_{1}}$};
		\draw[flecheo] (Q)--(P);
	\end{tikzpicture}	
\end{align*}
where the last molecular covariant is a transvectant $\tr{\overline{\Mdi{D}}^{1}}{\Mdi{E}}{0}$. Thus every molecular covariant of $\tr{\Mdi{D}}{\Mdi{E}}{1}$ is expressible in terms of $\mathbf{M}^{\nu(1)}$ and a linear combination of $\tr{\overline{\Mdi{D}}^{a_{1}}}{\overline{\Mdi{E}}^{a_{2}}}{0}$. All coefficients $a_{\nu}$ of~\ref{eq:DecompTrans} being positives, this conclude the case $r=1$.
\end{proof}

\begin{exemple}
Given $V=\Sn{n}$ $(n\geq 4)$ and the molecular covariants:
\begin{align*}
	\Mdi{D}
	&=
	\begin{tikzpicture}[baseline={([yshift=-.2cm]current bounding box.center)}]
			\node[m] (P)at(0.2,0.1){$\ff_{\alpha}$};
			\node[m] (Q)at(1.7,0.1){$\ff_{\beta}$};
			\draw[flecheno] (P)--(Q) node[midway,above] {$2$};
	\end{tikzpicture}
	\text{ and }
	\Mdi{E}=
	\begin{tikzpicture}[scale=1.5,baseline={([yshift=-0.15cm]current bounding box.center)}]
			\node[m] (P)at(0.2,0.1){$\ff_{\gamma}$};
	\end{tikzpicture}
\end{align*}
we can consider the transvectant $\tr{\Mdi{D}}{\Mdi{E}}{2}$ and the molecular covariant:
\begin{align*}
	\Mdi{M}
	&=
	\begin{tikzpicture}[baseline={([yshift=-.5ex]current bounding box.center)}]
			\node[m] (P)at(0.2,0.1){$\ff_{\alpha}$};
			\node[m] (Q)at(1.7,0.1){$\ff_{\beta}$};
			\node[m] (R)at(0.2,-1.1){$\ff_{\gamma}$};
			\draw[flecheno] (P)--(Q) node[midway,above] {$2$};
			\draw[flecheno] (P)--(R) node[midway,left] {$2$};
	\end{tikzpicture}
\end{align*}
By proposition~\ref{prop:TermAndTrans}:
\begin{align*}
	\Mdi{M}
	&=\lambda
	\left(
	\begin{tikzpicture}[baseline={([yshift=-0.2cm]current bounding box.center)}]
			\node[m] (P)at(0.2,0.1){$\ff_{\alpha}$};
			\node[m] (Q)at(1.7,0.1){$\ff_{\beta}$};
			\draw[flecheno] (P)--(Q) node[midway,above] {$2$};
	\end{tikzpicture}
	,
	\begin{tikzpicture}[scale=1.5,baseline={([yshift=-0.15cm]current bounding box.center)}]
			\node[m] (P)at(0.2,0.1){$\ff_{\gamma}$};
	\end{tikzpicture}
	\right)_{2}
	+\lambda_1 
	\left(
	\begin{tikzpicture}[baseline={([yshift=-0.2cm]current bounding box.center)}]
			\node[m] (P)at(0.2,0.1){$\ff_{\alpha}$};
			\node[m] (Q)at(1.7,0.1){$\ff_{\beta}$};
			\draw[flecheno] (P)--(Q) node[midway,above] {$3$};
	\end{tikzpicture}
	,
	\begin{tikzpicture}[scale=1.5,baseline={([yshift=-0.15cm]current bounding box.center)}]
			\node[m] (P)at(0.2,0.1){$\ff_{\beta}$};
	\end{tikzpicture}
	\right)_{1}\\
	&+\lambda_2
	\left(
	\begin{tikzpicture}[baseline={([yshift=-0.2cm]current bounding box.center)}]
			\node[m] (P)at(0.2,0.1){$\ff_{\alpha}$};
			\node[m] (Q)at(1.7,0.1){$\ff_{\beta}$};
			\draw[flecheno] (P)--(Q) node[midway,above] {$4$};
	\end{tikzpicture}
		,
	\begin{tikzpicture}[scale=1.5,baseline={([yshift=-0.15cm]current bounding box.center)}]
			\node[m] (P)at(0.2,0.1){$\ff_{\gamma}$};
	\end{tikzpicture}
	\right)_{0}
\end{align*}
\end{exemple}

\begin{cor}\label{cor:Trans_gen}
Let $V$ be a space of binary forms. Then every molecular covariant can be written in terms of transvectants.
\end{cor}

\begin{proof}
This is a simple induction on the number $d$ of atoms in a molecular covariant $\mathbf{M}$. For $d=1$, there is nothing to prove. Given $d>1$, we can write $\mathbf{M}$ as
\begin{center}
\begin{tikzpicture}[scale=1.5,baseline={([yshift=-0.3cm]current bounding box.center)}]
		\node[ellipse,draw,minimum width=40pt] (P)at(0.2,0.1){$\Mdi{D}$};
		\node[ellipse,draw,minimum width=40pt] (Q)at(2,0.1){$\ff$};
		\draw[flecheo] (P)--(Q) node[midway,above] {$\nu(r)$};
	\end{tikzpicture}
\end{center}
where $\ff\in V$, $r\geq 0$ and $\Mdi{D}$ is a molecular covariant with $d-1$ atoms. We conclude by induction and using proposition~\ref{prop:DecompTrans}.
\end{proof}

\begin{rem}
By corolary \ref{cor:Trans_gen}, we deduce theorem \ref{thm:FFT} from the fact that transvectants generate the covariant algebra of binary forms.  
\end{rem}


\section{Gordan's algorithm for joint covariants}
\label{sec:GordJoint}

Gordan's algorithm for joint covariants produces a finite generating set for $\cov{V_{1}\oplus V_{2}}$, knowing a finite family of generators for $\cov{V_{1}}$ and $\cov{V_{2}}$.

Let $V_{1}$ and $V_{2}$ be two spaces of binary forms and
\begin{equation*}
\mathrm{A}:=\lbrace \ff_{1},\cdots,\ff_{p}\rbrace\subset \cov{V_{1}},\quad \mathrm{B}:=\lbrace \bg_{1},\cdots,\bg_{q}\rbrace\subset \cov{V_{2}},
\end{equation*}
be finite families of generators for $\cov{V_{1}}$ and $\cov{V_{2}}$ respectively. 

\begin{lem}\label{lem:Cov_Joint_Trans}
$\cov{V_{1}\oplus V_{2}}$ is generated by transvectants
\begin{equation*}
	\tr{\mathbf{U}}{\mathbf{V}}{r},
\end{equation*}
with $r$ non-negative integer and
\begin{equation*}
	\mathbf{U}:=\ff_1^{\alpha_1}\dotsc\ff_p^{\alpha_p},\qquad \mathbf{V}:=\bg_1^{\beta_1}\dotsc\bg_q^{\beta_q},\qquad \alpha_i,\beta_j\in \mathbb{N}
\end{equation*}
\end{lem}

\begin{proof}
By theorem~\ref{thm:FFT}, $\cov{V_{1}\oplus V_{2}}$ is generated by molecular covariants $\Mdi{M}$ with atoms in $V_{1}$ and $V_{2}$. Just isolate in $\Mdi{M}$ atoms $\ff_i\in \mathrm{A}$ (resp. $\bg_j\in \mathrm{B}$) to form a molecular covariant $\Mdi{D}$ (resp. $\Mdi{E}$) such that 
\begin{center}
$\Mdi{M}=$
\begin{tikzpicture}[scale=1.5,baseline={([yshift=-.2cm]current bounding box.center)}]
		\node[ellipse,draw,minimum width=40pt] (P)at(0.2,0.1){$\mathbf{D}$};
		\node[ellipse,draw,minimum width=40pt] (Q)at(2,0.1){$\mathbf{E}$};
		\draw[flecheo] (P)--(Q) node[midway,above] {$\nu(r)$};
	\end{tikzpicture}
	$,\quad \Mdi{D}\in \cov{V_{1}},\quad \Mdi{E}\in \cov{V_{2}}$
\end{center}
which is a molecular covariant $\mathbf{M}^{\nu(r)}$ in the decomposition of a transvectant $\tr{\Mdi{D}}{\Mdi{E}}{r}$. By proposition~\ref{prop:TermAndTrans}, $\mathbf{M}^{\nu(r)}$ is a linear combination of 
\begin{equation*}
	\tr{\Mdi{D}}{\Mdi{E}}{r} \text{ and }
	\tr{
	\overline{\Mdi{D}}^{\mu_1(k_1)}}{\overline{\Mdi{E}}^{\mu_2(k_2)}}
	{r'},
\end{equation*}
with $\overline{\Mdi{D}}^{\mu_1(k_1)}\in \cov{V_{1}}$ and $\overline{\Mdi{E}}^{\mu_2(k_2)}\in \cov{V_{2}}$. By hypothesis, all covariants in $\cov{V_{1}}$ (resp. $\cov{V_{2}}$) can be recast using monomials in the $\ff_i's$ (resp. $\bg_j's$). 
\end{proof}

Define $a_i$ (resp. $b_j$) to be the order of the covariant $\ff_i$ (resp. $\bg_j$). Now, to each non--vanishing transvectant
\begin{equation*}
	\tr{\mathbf{U}}{\mathbf{V}}{r},
\end{equation*}
we can associate an integer solution $\kappa:=(\boldsymbol{\alpha},\boldsymbol{\beta},u,v,r)$ of the linear Diophantine system
\begin{equation}\label{Syst:Gordan_Linear}
S(\mathrm{A},\mathrm{B}):\quad \begin{cases}
	a_{1}\alpha_{1} + \dotsc + a_{p}\alpha_{p} = u+r, \\
	b_{1}\beta_{1} + \dotsc + b_{q}\beta_{q} = v+r,
	\end{cases}.
\end{equation}
Conversely, to each integer solution $\kappa$ of $S(\mathrm{A},\mathrm{B})$ we can associate a well defined transvectant $\tr{\mathbf{U}}{\mathbf{V}}{r}$. Recall an integer solution $\kappa$ of $S(\mathrm{A},\mathrm{B})$ is \emph{reducible} if we can decompose $\kappa$ as a sum of non--trivial solutions. Recall also that a non constant covariant $\mathbf{h}\in \mathbf{Cov}_{d,k}(V)$ is said to be reducible if $\mathbf{h}$ is in the algebra generated by covariants of degree $d'\leq d$ and order $k'\leq k$, with $d'<d$ or $k'<k$.  

\begin{lem}\label{lem:IntRed_EtTransRed}
If $\kappa=(\boldsymbol{\alpha},\boldsymbol{\beta},u,v,r)$ is a reducible integer solution of $S(\mathrm{A},\mathrm{B})$, then there exists a reducible molecular covariant $\mathbf{M}^{\nu(r)}$ in the decomposition of $\tr{\mathbf{U}}{\mathbf{V}}{r}$.
\end{lem}

\begin{proof}
Take the integer solution $\kappa=\kappa_1+\kappa_2$ to be reducible, with
\begin{equation*}
	\kappa_i=(\boldsymbol{\alpha}^i,\boldsymbol{\beta}^i,u^i,v^i,r^i) \text{ solution of } \eqref{Syst:Gordan_Linear}.
\end{equation*}
Thus we can write $\mathbf{U}=\mathbf{U}_1\mathbf{U}_2$ and $\mathbf{V}=\mathbf{V}_1\mathbf{V}_2$ and there exists $\nu(r),\nu_1(r^1)$ and $\nu_2(r^2)$ such that
\begin{align*}\label{Prop:MolReductible}
	\begin{tikzpicture}[scale=1.5,baseline={([yshift=-.2cm]current bounding box.center)}]
		\node[ellipse,draw,minimum width=40pt] (P)at(0.2,0.1){$\mathbf{U}$};
		\node[ellipse,draw,minimum width=40pt] (Q)at(2,0.1){$\mathbf{V}$};
		\draw[flecheo] (P)--(Q) node[midway,above] {$\nu(r)$};
	\end{tikzpicture}
	&=
	\begin{tikzpicture}[scale=1.5,baseline={([yshift=-.2cm]current bounding box.center)}]
		\node[ellipse,draw,minimum width=40pt] (P)at(0.2,0.1){$\mathbf{U}_1$};
		\node[ellipse,draw,minimum width=40pt] (Q)at(2,0.1){$\mathbf{V}_1$};
		\draw[flecheo] (P)--(Q) node[midway,above] {$\nu_1(r^1)$};
	\end{tikzpicture}
	\:
	\begin{tikzpicture}[scale=1.5,baseline={([yshift=-.2cm]current bounding box.center)}]
		\node[ellipse,draw,minimum width=40pt] (P)at(0.2,0.1){$\mathbf{U}_2$};
		\node[ellipse,draw,minimum width=40pt] (Q)at(2,0.1){$\mathbf{V}_2$};
		\draw[flecheo] (P)--(Q) node[midway,above] {$\nu_2(r^2)$};
	\end{tikzpicture}
\end{align*}
which is a reducible molecular covariant in the decomposition of $\tr{\mathbf{U}}{\mathbf{V}}{r}$.
\end{proof}

\begin{rem}\label{rem:TransNon_Red}
If an integer solution associated to a transvectant $\tr{\mathbf{U}}{\mathbf{V}}{r}$ is reducible, this does not implie that such a transvectant is a reducible one. Lemma~\ref{lem:IntRed_EtTransRed} only states that such a transvectant can be decomposed in terms which \emph{contain} a reducible transvectant. For instance, take $\ff\in \Sn{6}$, $\mathrm{A}=\mathrm{B}:=\lbrace \ff\rbrace$ and the transvectants
\begin{equation*}
	\tr{\ff^{\alpha_{1}}}{\ff^{\beta_{1}}}{5}. 
\end{equation*}
Then the solution $(\alpha_{1},\beta_{1},u,v,5)=(2,1,7,1,5)$ is a reducible one:
\begin{equation*}
	(2,1,7,1,5)=(1,1,1,1,5)+(1,0,6,0,0)
\end{equation*}
We directly observe that the transvectant
\begin{equation}\label{eq:TranNon_Red}
	\tr{\ff^{2}}{\ff}{5}
\end{equation}
contains the molecular covariant
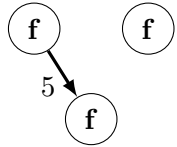
\begin{center}
	\begin{tikzpicture}
		\begin{scope}[xshift=2.5cm]
			\node[draw,circle] (P)at(0.5,1.1){$\ff$};
			\node[draw,circle] (Q)at(2,1.1){$\ff$};
			\node[draw,circle] (R)at(1.25,-0.1){$\ff$};
			\draw[flecheo] (P)--(R) node[pos=0,below=0.2cm] {5};
		\end{scope}
	\end{tikzpicture}
\end{center}
which is a null covariant. Observe that property~\ref{prop:TermAndTrans} implies that transvectant \eqref{eq:TranNon_Red} is a linear combination of transvectants
\begin{equation*}
	\tr{\tr{\ff}{\ff}{4}}{\ff}{1},\quad \tr{\tr{\ff}{\ff}{3}}{\ff}{2}=0,\quad \tr{\tr{\ff}{\ff}{2}}{\ff}{3},\quad \tr{\tr{\ff}{\ff}{1}}{\ff}{4}=0,
\end{equation*}
and one can finally show that
\begin{equation*}
	\tr{\ff^{2}}{\ff}{5}=\frac{65}{66}\tr{\tr{\ff}{\ff}{4}}{\ff}{1},
\end{equation*}
where $\tr{\tr{\ff}{\ff}{4}}{\ff}{1}$ is an irreducible covariant, as being in the covariant bases of $\Sn{6}$ (see table~\ref{table:CovS6}).
\end{rem}

Nevertheless, we have the following result:

\begin{lem}\label{lem:CovBound}
Let $\mathbf{a} := \max(a_{i})$, $\mathbf{b}:= \max(b_{j})$ and
\begin{equation*}
	\mathbf{U}:=\ff_{1}^{\alpha_{1}}\dotsc \ff_{p}^{\alpha_{p}},\quad \mathbf{V}:=\bg_{1}^{\beta_{1}}\dotsc \bg_{q}^{\beta_{q}}.
\end{equation*}
Let $u=$\emph{Ord}$(\mathbf{U})-r$ and $v=$\emph{Ord}$(\mathbf{V})-r$. If
\begin{equation}\label{eq:condition-reducibility}
  u+v \ge \mathbf{a}+\mathbf{b},
\end{equation}
then, the transvectant $\tr{\mathbf{U}}{\mathbf{V}}{r}$ is reducible.
\end{lem}

\begin{proof}
Condition~\eqref{eq:condition-reducibility} implies that $u \ge \mathbf{a}$ or $v \ge \mathbf{b}$. Thus the transvectant $\tr{\mathbf{U}}{\mathbf{V}}{r}$ contains a reducible molecular covariant $\mathbf{M}^{\nu(r)}$ (the corresponding integer solution $\kappa$ is thus not minimal). By virtue of proposition~\ref{prop:TermAndTrans}, this transvectant is a linear combination of the term $\mathbf{M}^{\nu(r)}$ and transvectants
\begin{equation*}
 \tr{\overline{\mathbf{U}}^{\mu(k_{1})}}{\overline{\mathbf{U}}^{\mu(k_{2})}}{r^{\prime}},
\end{equation*}
where $r^{\prime} < r$ and $k_{1} + k_{2} = r-r^{\prime}$.
Note that, since both families $\mathrm{A}$ and $\mathrm{B}$ are supposed to be generator sets, we have
\begin{equation*}
  \overline{\mathbf{U}}^{\mu(k_{1})} = \ff_{1}^{\alpha^{\prime}_{1}} \dotsc \ff_{p}^{\alpha^{\prime}_{p}}, \qquad \overline{\mathbf{V}}^{\mu(k_{2})} = \bg_{1}^{\beta^{\prime}_{1}} \dotsc \bg_{q}^{\beta^{\prime}_{q}},
\end{equation*}
where, moreover, the order of the transvectant $\tr{\overline{\mathbf{U}}^{\mu(k_{1})}}{\overline{\mathbf{V}}^{\mu(k_{2})}}{r^{\prime}}$ is $u^{\prime} + v^{\prime} = u + v$. Since we have supposed that $u + v \ge \mathbf{a}+\mathbf{b}$, we get that $u^{\prime} + v^{\prime} \ge \mathbf{a}+\mathbf{b}$ and the proof is achieved by a recursive argument on the index of the transvectant $r$.
\end{proof}

\begin{rem}
The statement $u+v\geq \mathbf{a}+\mathbf{b}$ cannot be replaced by the weaker hypothesis $u \ge \mathbf{a}$ or $v \ge \mathbf{b}$. For instance in remark~\ref{rem:TransNon_Red}, for $\ff\in \Sn{6}$ and $\mathbf{h}:=\tr{\ff^2}{\ff}{5}$, we have $u=7\geq 6$ but
$\mathbf{h}$ is not reducible.
\end{rem}

Lemma~\ref{lem:CovBound} is closely related to:

\begin{cor}
Let $\mathbf{F}\in \cov{V}$ of order $s$ and $\lbrace \mathbf{F}_{1},\cdots,\mathbf{F}_{k}\rbrace\subset \cov{V}$ be a family of homogeneous covariants. Let $t_{i}$ be the order of $\mathbf{F}_{i}$ and $\mathbf{t}=\max(t_{i})$. For a given integer $r$, if
\begin{equation*}
	\sum_{i=1}^{k}t_i\geq \mathbf{a} +2r,
\end{equation*}
then the transvectant $\tr{\mathbf{F}_{1}\dotsc\mathbf{F}_{k}}{\mathbf{F}}{r}$ is reducible.
\end{cor}

\begin{proof}
Let $\ff_{1},\dotsc,\ff_{p}$ be a covariant bases of $\cov{V}$, each $\ff_{i}$'s being a homogeneous covariant of order $a_{i}$. Then, each covariant $\mathbf{F}_{j}$ is a linear combination of monomials $\ff_{i_{1}}^{\alpha_{i_{1}}}\dotsc \ff_{i_{l}}^{\alpha_{i_{l}}} $ with $a_{i}\leq t_{j} \leq \mathbf{t}$. Thus $\mathbf{F}_{1}\dotsc\mathbf{F}_{k}$ is a covariant expressible in terms of monomials $\mathbf{U}$ in the $\ff_{i}$'s with 
\begin{equation*}
\text{Ord}(\mathbf{U})=\sum_{i=1}^{k}t_i \text{ and } \max(a_{i})\leq \mathbf{t}. 
\end{equation*}
We have also $\mathbf{F}=\ff_{j_{1}}^{\beta_{j_{1}}}\dotsc \ff_{j_{m}}^{\beta_{j_{m}}}$ with $\max(a_{j})\leq s$.
By lemma~\ref{lem:CovBound}, each transvectant $\tr{\mathbf{U}}{\mathbf{V}}{r}$ is thus a reducible covariant.
\end{proof}

Take back $\mathrm{A}$ and $\mathrm{B}$ to be finite generator sets of $V_1$ and $V_2$, respectively. We know that there exists a finite family of \emph{irreducible integer solutions} of the system $S(\mathrm{A},\mathrm{B})$ \eqref{Syst:Gordan_Linear} (see~\cite{Sta2012,Sta1983,Stu2008} for details on linear Diophantine systems). 

\begin{thm}\label{thm:CovJoints}
The algebra $\cov{V_{1}\oplus V_{2}}$ is generated by the finite family $\mathrm{C}$ of transvectants 
\ban
	\tr{\mathbf{U}}{\mathbf{V}}{r}
\ean
corresponding to irreducible solutions of the linear Diophantine system $S(\mathrm{A},\mathrm{B})$ \eqref{Syst:Gordan_Linear}.
\end{thm}

\begin{proof}
Let first remark that each $\ff_i$ (resp. each $\bg_j$) corresponds to an irreducible solution of $S(\mathrm{A},\mathrm{B})$. Thus $\mathrm{A}\subset\mathrm{C}$ and $\mathrm{B}\subset\mathrm{C}$.

From lemma~\ref{lem:Cov_Joint_Trans}, we know that $\cov{V_{1}\oplus V_{2}}$ is generated by transvectants $\tr{\mathbf{U}}{\mathbf{V}}{r}$ where $\mathbf{U}$ (resp. $\mathbf{V}$) is a monomial in $\GC[\mathrm{A}]$ (resp. $\GC[\mathrm{B}]$) and $r$ is a non--negative integer. 

The proof is by induction on $r$. When $r=0$, for $\mathrm{A}\subset\mathrm{C}$ and $\mathrm{B}\subset\mathrm{C}$, we know that the conclusion is true. 

Now, let $r>0$ and $\tr{\mathbf{U}}{\mathbf{V}}{r}$ be a transvectant which corresponds to a reducible integer solution\footnote{For simplicity, we can suppose that $\kappa$ is the sum of two irreducible solutions.}
\begin{equation*}
\kappa=\kappa_{1}+\kappa_{2},\quad \kappa_{i}\text{ irreducible}. 
\end{equation*}
As in lemma~\ref{lem:IntRed_EtTransRed}, there exists a molecular covariant $\mathbf{M}^{\nu(r)}$ in the decomposition of $\tr{\mathbf{U}}{\mathbf{V}}{r}$ which can be written as
\begin{equation*}
\mathbf{M}^{\nu(r)}=\mathbf{M}_{1}^{\nu_{1}(r_{1})}\mathbf{M}_{2}^{\nu_{2}(r_{2})},\quad r_i\leq r.
\end{equation*}
But $\mathbf{M}_{i}^{\nu_{i}(r_{i})}$ is a term in a transvectant $\boldsymbol{\tau}^{i}=\tr{\mathbf{U}_i}{\mathbf{V}_i}{r_i}\in \mathrm{C}$. Then by proposition~\ref{prop:TermAndTrans}, $\tr{\mathbf{U}}{\mathbf{V}}{r}$ is a linear combination of a product of $\boldsymbol{\tau}^{i}$'s and transvectants
\begin{equation}\label{eq:TransPreuveCovJoints}
	\tr{\overline{\mathbf{U}}^{\mu}}{\overline{\mathbf{V}}^{\mu'}}{r'},
	\tr{\overline{\mathbf{U}}_{i}^{\nu_i}}{\overline{\mathbf{V}}_{i}^{\nu'_i}}{r'_i}
	\quad r'<r,r'_i<r
\end{equation}
Now $\overline{\mathbf{U}}^{\mu},\overline{\mathbf{U}}_{i}^{\nu_i}\in \cov{V_1}$ (resp. $\mathbf{V},\mathbf{V}_i$ and $\cov{V_2}$). Therefore the transvectants~\eqref{eq:TransPreuveCovJoints} are linear combinations of
\begin{equation*}
	\tr{\mathbf{U}'}{\mathbf{V}'}{r''},\quad	r''<r,
\end{equation*}
where $\mathbf{U'}$ (resp. $\mathbf{V}'$) is a monomial in the $\ff_{i}$'s (resp. $\bg_{j}$). Thus, by induction on $r$, the algebra $\cov{V_1\oplus V_2}$ is generated by the finite family $\mathrm{C}$.
\end{proof}

Note that lemma~\ref{lem:CovBound} gives a bound on the order of each element of a minimal bases of joint covariants:

\begin{cor}
Let $V = \mathrm{S}_{n_{1}} \oplus \dotsb \oplus \mathrm{S}_{n_{s}}$. If $\mu_{i}$ is the maximal order of a minimal bases for $\mathrm{S}_{n_{i}}$, then, for each element $\mathbf{h}$ of a minimal bases for $V$, we get
\begin{equation*}
  \mathrm{ord}(\mathbf{h}) \le \sum_{i=1}^{s} \mu_{i}.
\end{equation*}
\end{cor}

\begin{exemple}\label{ex:Cov34}
We can directly use theorem~\ref{thm:CovJoints} to get a covariant bases of $\Sn{3}\oplus\Sn{4}$. The same result has been obtained by Popoviciu--Brouwer~\cite{BP2012} with more computations. Let $\uu\in \Sn{3}$ and $\vv\in \Sn{4}$. Recall that:
\begin{itemize}
\item[$\bullet$] The algebra $\cov{\Sn{3}}$ is generated by the three covariants~\cite{GY2010}:
\begin{equation*}
\uu\in\Sn{3},\quad \mathbf{h}_{2,2}:=\tr{\uu}{\uu}{2}\in \Sn{2},\quad \mathbf{h}_{3,3}:=\tr{\uu}{\mathbf{h}_{2,2}}{1}\in \Sn{3}
\end{equation*}
and one invariant $\Delta:=\tr{\uu}{\mathbf{h}_{3,3}}{3}$;
\item[$\bullet$] The algebra $\cov{\Sn{4}}$ is generated by the three covariants~\cite{GY2010}: 
\begin{equation*}
\vv\in \Sn{4},\quad \mathbf{k}_{2,4}:=\tr{\vv}{\vv}{2}\in \Sn{4},\quad \mathbf{k}_{3,6}:=\tr{\vv}{\mathbf{k}_{2,4}}{1}\in \Sn{6}
\end{equation*}
and the two invariants $i:=\tr{\vv}{\vv}{4}$, $j:=\tr{\vv}{\mathbf{k}_{2,4}}{4}$.
\end{itemize}

We then have to solve the linear Diophantine system
\begin{equation}\label{Syst:Cov64}
(S)\: : \:
\begin{cases}
	2\alpha_{1}+3\alpha_{2}+3\alpha_{3}&=u+r \\
	4\beta_{1}+4\beta_{2}+6\beta_{3}&=v+r
\end{cases}.
\end{equation}
Using Normaliz package in Macaulay 2~\cite{BI2010}, this leads to $104$ solutions. The associated covariants form a family of covariants of maximum total degree $d+k=18$. The Hilbert series of $\cov{\Sn{4}\oplus\Sn{3}}$, computed using Bedratyuk's Maple package~\cite{Bed2011}, is given by
\begin{align*}
	H(z)&=1+{z}^{2}+2\,{z}^{3}+5\,{z}^{4}+10\,{z}^{5}+18\,{z}^{6}+31\,{z}^{7}+55\,{z}^{8}+92\,{z}^{9}\\
	&+144\,{z}^{10}+223\,{z}^{11}+341\,{z}^{12}+499\,{z}^{13}+725\,{z}^{14}+1031\,{z}^{15}\\
	&+1436\,{z}^{16}+1978\,{z}^{17}+2685\,{z}^{18}+\dotsc
\end{align*}
By scripts written in Macaulay 2~\cite{M2}, we reduced the family of $104$ generators to a minimal set of $63$ generators given in table~\ref{table:Cov34}, which has also been obtained by Popoviciu--Brouwer~\cite{BP2012}.
\begin{table}[H]
\begin{equation*}
\setlength{\arraycolsep}{1pt}
\begin{array}{c||ccccccc|c|c|c}
 d/o &   0 &  1 &  2 &   3 &   4 &   5 &  6	&  \# 	& Cum \\ \hline\hline
   1 &   - &  - &  - &   1 &   1 &   - &  -	&	2	&  2 \\
   2 &   1 &  1 &  1 &   1 &   1 &   1 &  -	&	6	&  8 \\
   3 &   1 &  1 &  2 &   2 &   1 &   1 &  1	&	9	& 17 \\
   4 &   1 &  2 &  2 &   2 &   1 &   - &  -	& 	8	& 25 \\
   5 &   2 &  3 &  3 &   1 &   1 &   - &  -	&  10	& 35 \\
   6 &   2 &  3 &  2 &   1 &   - &   - &  -	& 	8	& 43 \\
   7 &   3 &  3 &  1 &   - &   - &   - &  -	& 	7	& 50 \\
   8 &   3 &  2 &  - &   - &   - &   - &  -	& 	5	& 55 \\
   9 &   4 &  1 &  - &   - &   - &   - &  -	& 	5	& 60 \\
  10 &   2 &  - &  - &   - &   - &   - &  -	& 	2	& 62 \\
  11 &   1 &  - &  - &   - &   - &   - &  -	& 	1	& 63 \\ \hline
 Tot &  20 & 16 & 11 &   8 &   5 &   2 &  1  	& 		& 63  \\
\end{array}
\end{equation*}\caption{Covariant bases of $\Sn{3}\oplus\Sn{4}$}\label{table:Cov34}
\end{table}
\end{exemple}


\section{Gordan's algorithm for simple covariants}
\label{sec:GordSimp}

There is a second version of Gordan's algorithm which enables to compute a covariant bases for $\Sn{n}$, knowing covariant basis for $\Sn{k}$, ($k < n$). The main idea is, once again, to make use of linear Diophantine system, arguing step by step modulus a chain ideal. 

\subsection{Relatively complete family and Gordan's ideal}
Let 
\ban
\mathrm{A}=\lbrace \ff_{1},\cdots,\ff_{p}\rbrace\subset \cov{\Sn{n}}
\ean
denote a \emph{finite family} of covariants (not necessary a bases). We define $\cov{\mathrm{A}}$ to be the algebra generated by the set which contains the family $\mathrm{A}$ and closed under transvectant operations\footnote{Equivalently (by~\ref{prop:DecompTrans} and~\ref{prop:TermAndTrans}), we can choose the set of all molecular covariants with atoms taken in $\mathrm{A}$.}.

\begin{defn}
Let $I\subset \cov{V}$ be an homogeneous ideal. A family $\mathrm{A}=\lbrace \ff_{1},\cdots,\ff_{p}\rbrace\subset \cov{V}$ of homogeneous covariants is \emph{relatively complete modulo $I$} if every homogeneous covariant $\mathbf{h}\in \cov{\mathrm{A}}$ of degree $d$ can be written
\begin{equation*}
\mathbf{h}=\mathbf{p}(\ff_{1},\dotsc,\ff_{p})+\mathbf{h}_I \text{ with } \mathbf{h}_I\in I,
\end{equation*}
where $\mathbf{p}(\ff_{1},\dotsc,\ff_{p})$ and $\mathbf{h}_I$ are degree $d$ homogeneous covariants.
\end{defn}

\begin{rem}
The notion of \emph{relatively complete} family is weaker than the one of \emph{generator set}. For instance, take $\uu\in \Sn{3}$ and as in example~\ref{ex:Cov34}
\begin{equation*}
	\mathbf{h}_{2,2}:=\tr{\uu}{\uu}{2}\in \Sn{2},\quad \mathbf{h}_{3,3}:=\tr{\uu}{\mathbf{h}_{2,2}}{1}\in \Sn{3}.
\end{equation*}
Take now the invariant $\tilde{\Delta}:=\tr{\mathbf{h}_{2,2}}{\mathbf{h}_{2,2}}{2}$. The family $\mathrm{A}_1=\lbrace \uu,\mathbf{h}_{2,2},\mathbf{h}_{3,3},\tilde{\Delta}\rbrace$ is also a covariant bases of $\cov{\mathrm{A}_1}=\cov{\Sn{3}}$ and is thus a relatively complete family modulo $I=\set{0}$. Now, let
\begin{equation*}
	\mathrm{A}_2:=\set{\mathbf{h}_{2,2},\tilde{\Delta}}.
\end{equation*}
We have $\cov{\mathrm{A}_2}\subsetneq \cov{\Sn{3}}$, but $\mathrm{A}_2$ is exactly a covariant bases~\cite{GY2010} of the quadratic form $\mathbf{h}_{2,2}\in \Sn{2}$, thus $\mathrm{A}_2$ is a relatively complete family modulo $I=\set{0}$ but is not a covariant bases of $\cov{\Sn{2}}$.
\end{rem}

Let now $\di{D}$ be a molecule upon $\Sn{n}$ (recall such a molecule represents an $\sldc$ equivariant homomorphism), the \emph{grade} of $\di{D}$, denoted $\gr(\di{D})$, is the maximum weight of the edges of $\di{D}$ :
\begin{equation*}
	\gr(\di{D}):=\max_{e\in \mathcal{E}(\di{D})} w(e).
\end{equation*}

\begin{defn}
Let $r$ be an integer; we define $\mathcal{G}_r$ to be the set of all molecular covariants issued from a molecule $\di{D}$ (cf. \eqref{defn:Molecules_Et_Covariants}) with grade at least $r$.
\end{defn}

As a first observation, it is clear that $\mathcal{G}_r=\set{0}$ for $r>n$. Furthermore, we  have
\begin{equation}\label{InclusionGrade}
	\mathcal{G}_{i+1}\subset \mathcal{G}_i \text{ for all } i.
\end{equation}

\begin{defn}[Gordan's ideals]\label{def:Gord_Ideal}
Let $r$ be an integer. We define the Gordan ideal $I_r$ to be the homogeneous\footnote{Such an ideal is clearly an homogeneous ideal as being generated by homogeneous elements.} ideal generated by $\mathcal{G}_r$; we write
\begin{equation*}
	I_r:=\langle \mathcal{G}_r \rangle.
\end{equation*}
\end{defn}

\begin{rem}\label{rem:Gord_Ideal_Another_Family}
Gordan's ideal $I_r$ (for $r\leq n$) is also generated by the set of transvectants
\ban
	\tr{\bh}{\tr{\ff}{\ff}{r_1}}{r_2},\quad  r_1\geq r,\quad r_2\in \mathbb{N}^*,	
\ean
where $\bh\in \cov{\Sn{n}}$ is an homogeneous covariant. This is a direct application of propositions~\ref{prop:DecompTrans} and~\ref{prop:TermAndTrans}. Because $\tr{\ff}{\ff}{r_1}=0$ for $r_1$ odd, such a family can also be written as the family of transvectants
\ban
	\tr{\mathbf{h}}{\mathbf{H}_{2k}}{r'},\quad \mathbf{H}_{2k}:=\tr{\ff}{\ff}{2k},\quad 2k\geq r,\quad r'\geq 0.
\ean
\end{rem}

We directly observe that:
\begin{itemize}
\item[$\bullet$] $I_r=\lbrace 0\rbrace$ for all $r>n$;
\item[$\bullet$] By equation~\ref{InclusionGrade}, $I_{r+1}\subset I_r$ for every integer $r$.
\item[$\bullet$] By remark~\ref{rem:Gord_Ideal_Another_Family}: 
\begin{equation}\label{eq:Grade_Impair}
	I_{2k-1}=I_{2k},\quad \forall 2k\leq n. 
\end{equation}
\end{itemize}

By the property~\ref{prop:DecompTrans}, Gordan's ideals are stable by transvectant operations:

\begin{lem}\label{lem:IdeauxStables}
Let $\mathbf{h}_r\in I_r$, and $\mathbf{h}\in \cov{\Sn{n}}$ be a covariant. Then for every integer $r'\geq 0$,
\begin{equation*}
	\tr{\mathbf{h}}{\mathbf{h}_r}{r'}\in I_r.
\end{equation*}
\end{lem}

\begin{rem}\label{rem:IdeauxInvStable}
Suppose that $\Delta\in \cov{\Sn{n}}$ is an invariant. Then the ideal $\langle \Delta \rangle$ is also stable by transvectant operations, since
\begin{equation*}
	\tr{\mathbf{h}}{\Delta\mathbf{k}}{r}=\Delta \tr{\mathbf{h}}{\mathbf{k}}{r}.
\end{equation*}
\end{rem}

Given two finite families $\mathrm{A}$ and $\mathrm{B}$ of covariants, let $\kappa^1,\dotsc,\kappa^l$ be the irreducible integer solutions of the linear system $S(\mathrm{A},\mathrm{B})$ \eqref{Syst:Gordan_Linear} and $\boldsymbol{\tau}^i$ be the associated transvectants. Let $\ff\in \Sn{n}$, $\Delta\in \cov{\Sn{n}}$ be an invariant, $k\ge 0$ and
\begin{equation*}
	\mathbf{H}_{2k}:=\tr{\ff}{\ff}{2k}.
\end{equation*}
Finally, write $J_{2k+2}=I_{2k+2}$ or $J_{2k+2}=I_{2k+2}+\langle \Delta \rangle$.

\begin{thm}\label{thm:FamRelComplete}
Suppose that $\mathrm{A}$ is relatively complete modulo $I_{2k}$ and contains the binary form $\ff$. Suppose also that $\mathrm{B}$ is relatively complete modulo $J_{2k+2}$ and contains the covariant $\mathbf{H}_{2k}$. Then the family $\mathrm{C}:=\lbrace \boldsymbol{\tau}^1,\dotsc,\boldsymbol{\tau}^l\rbrace$ is relatively complete modulo $J_{2k+2}$ and
\begin{equation*}
	\cov{\mathrm{C}}=\cov{\mathrm{A}\cup \mathrm{B}}=\cov{\Sn{n}}.
\end{equation*}
\end{thm}

\begin{proof}
We first use the fact that $\cov{\mathrm{A}\cup \mathrm{B}}$ is generated by the (infinite) family of transvectants
\ban
	\tr{\Mdi{D}}{\Mdi{E}}{r}
\ean
where $\Mdi{D}\in \cov{\mathrm{A}}$ is a degree $d$ homogeneous covariant of $\ff\in \Sn{n}$  and $\Mdi{E}\in \cov{\mathrm{B}}$ (the proof is the same as the one for lemma~\ref{lem:Cov_Joint_Trans}). We order such a family using lexicographic order on $(d,r)$. By hypothesis, we can suppose that
\ban
	\Mdi{D}=\mathbf{U}+\mathbf{h}_{2k},\quad \mathbf{h}_{2k}\in I_{2k}
\ean 
where $\mathbf{U}$ is a monomial of degree $d$ in $\GC[\mathrm{A}]$ and $\mathbf{h}_{2k}$ is a homogeneous covariant of degree $d$. Furthermore we can write 
\ban
	\Mdi{E}=\mathbf{V}+\mathbf{h}_{2k+2},\quad \mathbf{h}_{2k+2}\in J_{2k+2}.
\ean
Thus we have 
\begin{equation}\label{eq:Cov_Alg_Simp}
	\tr{\Mdi{D}}{\Mdi{E}}{r}=\tr{\mathbf{U}}{\mathbf{V}}{r}+\tr{\mathbf{h}_{2k}}{\mathbf{V}}{r}+\underbrace{\tr{\mathbf{U}}{\mathbf{h}_{2k+2}}{r}+\tr{\mathbf{h}_{2k}}{\mathbf{h}_{2k+2}}{r}}_{\in J_{2k+2}}.
\end{equation}
The goal here is to prove that such a covariant can be written as
\ban
	\mathbf{p}(\mathrm{C})+\mathbf{h}'_{2k+2},\quad \mathbf{h}'_{2k+2}\in J_{2k+2}.
\ean
But in \eqref{eq:Cov_Alg_Simp} we just have to focus on
\ban
	\tr{\mathbf{U}}{\mathbf{V}}{r}+\tr{\mathbf{h}_{2k}}{\mathbf{V}}{r}.
\ean
We do it by induction on $(d,r)$. For $d=0$, there is nothing to prove. Suppose now that our claim is true up to an integer $d$ and take covariants 
\ban
	\tr{\mathbf{U}}{\mathbf{V}}{r} \text{ and } \tr{\mathbf{h}_{2k}}{\mathbf{V}}{r}
\ean
where $\mathbf{U}$ and $\mathbf{h}_{2k}$ are of degree $d+1$.
\begin{enumerate}
\item If $\tr{\mathbf{U}}{\mathbf{V}}{r}$ corresponds to a reducible solution, using proposition~\ref{prop:TermAndTrans} and the same argument as in the proof of theorem~\ref{thm:CovJoints}, this transvectant decomposes as
\ban
	\mathbf{p}(\mathrm{C}),\quad \tr{\overline{\mathbf{U}}}{\overline{\mathbf{V}}}{r'<r}.
\ean
But $\overline{\mathbf{U}}$ is of degree $d+1$ : we conclude using a direct induction on $r'$.
\item Using remark~\ref{rem:Gord_Ideal_Another_Family}, the covariant $\mathbf{h}_{2k}$ can be written as a linear combination of
\ban
	\tr{\mathbf{M}}{\mathbf{H}_{2j}}{r_1},\quad \mathbf{H}_{2j}:=\tr{\ff}{\ff}{2j},\quad j\geq k,
\ean
where the degree of $\mathbf{M}$ in $\ff$ is strictly less then $d+1$.
The case $j>k$ being obvious, we only focus on the case $j=k$, and then consider transvectants
\ban
	 \tr{\tr{\mathbf{M}}{\mathbf{H}_{2k}}{r_1}}{\mathbf{V}}{r}.
\ean
Using lemma~\ref{lem:Deg_3_Cov} on degree $3$ covariant bases, such a covariant can be written as a linear combination of
\ban
	\tr{\mathbf{M}}{\tr{\mathbf{H}_{2k}}{\mathbf{V}}{r'_1}}{r'}
\ean
but $\tr{\mathbf{H}_{2k}}{\mathbf{V}}{r'_1}\in \cov{\mathrm{B}}$, thus we have to consider transvectants
\ban
	\tr{\mathbf{M}}{\mathbf{V}'}{r'}
\ean
where degree of $\mathbf{M}$ in $\ff$ is strictly less than $d+1$ : we conclude using induction on $d$.
\end{enumerate}
Thus, for all couple $(d,r)$, our claim is true.
\end{proof}

\subsection{The algorithm}\label{subsec:The_Alg}

Take $V=\Sn{n}$ $(n>2)$ and $\ff\in \Sn{n}$. By corollary~\ref{cor:InitGordan}, the family $\mathrm{A}_{0}:=\set{ \ff}$ is relatively complete modulo $I_{2}$. This means that every covariant $\mathbf{h}\in \cov{\Sn{n}}$ can be written as
\begin{equation*}
\mathbf{h}=\mathbf{p}(\ff)+\mathbf{h}_2 \text{ with } \mathbf{h}_2\in I_2.
\end{equation*}

Take now the covariant $\mathbf{H}_{2}=\tr{\ff}{\ff}{2}$ of order $2n-4$:
\begin{itemize}
\item[$\bullet$] If $2n-4>n$, we take $B_{0}:=\lbrace \mathbf{H}_{2}\rbrace$ which is, by lemma~\ref{lem:Ordre>n}, relatively complete modulo $I_{4}$; applying theorem~\ref{thm:FamRelComplete} we get a family $\mathrm{A}_{1}:=\mathrm{C}$ relatively complete modulo $I_{4}$.
\item[$\bullet$] If $2n-4=n$, we take $B_{0}:=\lbrace \mathbf{H}_{2},\Delta\rbrace$ which is, by  lemma~\ref{lem:Ordre=n}, relatively complete modulo $I_{4}+\langle \Delta \rangle$; where $\Delta$ is the invariant
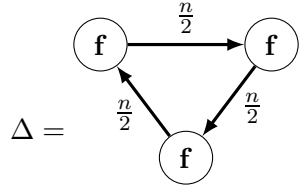
\begin{align*}
	\Delta
	&=
	\begin{tikzpicture}[scale=1.5,baseline={([yshift=-4ex]current bounding box.center)}]
			\node[m] (P)at(0.2,1.1){$\mathbf{f}$};
			\node[m] (Q)at(1.7,1.1){$\mathbf{f}$};
			\node[m] (R)at(0.95,0.1){$\mathbf{f}$};
			\draw[flecheo] (P)--(Q) node[midway,above] {$\frac{n}{2}$};
			\draw[flecheo] (Q)--(R) node[pos=0.6,right=0.1] {$\frac{n}{2}$};
			\draw[flecheo] (R)--(P) node[pos=0.3,left=0.1] {$\frac{n}{2}$};			
	\end{tikzpicture}
\end{align*}
In that case, by applying theorem~\ref{thm:FamRelComplete}, we can take $\mathrm{A}_{1}$ to be $\mathrm{C}\cup \lbrace \Delta \rbrace$. A direct induction on the degree of the covariant shows that $\mathrm{A}_{1}$ is relatively complete modulo $I_{4}$.
\item[$\bullet$] If $2n-4<n$, we suppose already known a covariant bases of $\Sn{2n-4}$; we then take $\mathrm{B}_0$ to be this bases, which is finite and relatively complete modulo $I_{4}$ (because relatively complete modulo $\set{0}$); we directly apply theorem~\ref{thm:FamRelComplete} to get $\mathrm{A}_{1}:=\mathrm{C}$.
\end{itemize}

Let now be given by induction a family $\mathrm{A}_{k-1}$ containing $\ff$, finite and relatively complete modulo $I_{2k}$. We consider the covariant $\mathbf{H}_{2k}=\tr{\ff}{\ff}{2k}$. Then:

\begin{itemize}
\item[$\bullet$] If $\mathbf{H}_{2k}$ is of order $p>n$, we take $B_{k-1}:=\lbrace \mathbf{H}_{2k}\rbrace$ which, by lemma~\ref{lem:Ordre>n}, is relatively complete modulo $I_{2k+2}$. By theorem~\ref{thm:FamRelComplete} we take $\mathrm{A}_{k}:=\mathrm{C}$.
\item[$\bullet$] If $\mathbf{H}_{2k}$ is of order $p=n$, we take $B_{k-1}:=\lbrace \mathbf{H}_{2k},\Delta\rbrace$ which, by  lemma~\ref{lem:Ordre=n}, is relatively complete modulo $I_{2k+2}+\langle \Delta \rangle$; where $\Delta$ is the invariant
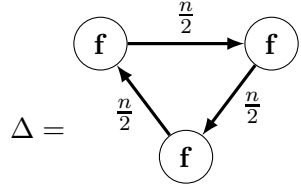
\begin{align*}
	\Delta
	&=
	\begin{tikzpicture}[scale=1.5,baseline={([yshift=-4ex]current bounding box.center)}]
			\node[m] (P)at(0.2,1.1){$\mathbf{f}$};
			\node[m] (Q)at(1.7,1.1){$\mathbf{f}$};
			\node[m] (R)at(0.95,0.1){$\mathbf{f}$};
			\draw[flecheo] (P)--(Q) node[midway,above] {$\frac{n}{2}$};
			\draw[flecheo] (Q)--(R) node[pos=0.6,right=0.1] {$\frac{n}{2}$};
			\draw[flecheo] (R)--(P) node[pos=0.3,left=0.1] {$\frac{n}{2}$};			
	\end{tikzpicture}
\end{align*}
In that case, by applying theorem~\ref{thm:FamRelComplete}, we can take $\mathrm{A}_{k}$ to be $\mathrm{C}\cup \lbrace \Delta \rbrace$. A direct induction on the degree of the covariant shows that $\mathrm{A}_{k}$ is relatively complete modulo $I_{2k+2}$.
\item[$\bullet$] If $\mathbf{H}_{2k}$ is of order $p<n$, we suppose already known a covariant bases of $\Sn{p}$; we then take $\mathrm{B}_{k-1}$ to be this bases, which is relatively complete modulo $I_{2k+2}$ (because relatively complete modulo $\set{0}$); we directly apply theorem~\ref{thm:FamRelComplete} to get $\mathrm{A}_{k}:=\mathrm{C}$.
\end{itemize}

Thus in each case, we have defined the family $\mathrm{A}_{k}$. Now, depending on $n$'s parity:
\begin{itemize}
\item[$\bullet$] If $n=2q$ is even, we know that the family $\mathrm{A}_{q-1}$ is relatively complete modulo $I_{2q}$; furthermore the family $B_{q-1}$ only contains the invariant $\Delta_q:=\lbrace \ff,\ff\rbrace_{2q}$; finally we observe that $\mathrm{A}_{p}$ is given by
\begin{equation*}
\mathrm{A}_{p}:=\mathrm{A}_{p-1}\cup \lbrace \Delta_q \rbrace
\end{equation*}
and is relatively complete modulo $I_{2q+2}=\lbrace 0\rbrace$; thus it is a covariant bases.
\item[$\bullet$] If $n=2q+1$ is odd, the family $\mathrm{B}_{q-1}$ contains the quadratic form $\mathbf{H}_{2q}:=\lbrace \ff,\ff\rbrace_{2q}$; we know then that the family $\mathrm{B}_{q-1}$ is given by the covariant $\mathbf{H}_{2q}$ and the invariant $\delta_{q}:=\tr{\mathbf{H}_{2q}}{\mathbf{H}_{2q}}{2}$. By theorem~\ref{thm:FamRelComplete}, the family $\mathrm{A}_{q}:=\mathrm{C}$ is relatively complete modulo $I_{2q+2}=\lbrace 0\rbrace$ and is thus a covariant bases.
\end{itemize}

\section{Improvment of Gordan's algorithm}\label{sec:Improvement}

Using Gordan's algorithm, one gets a finite set of generators. In general, such a family is not minimal, as shown in example~\ref{ex:Cov34} for the algebra $\cov{\Sn{3}\oplus\Sn{4}}$. A classical way to get a minimal bases once given a finite bases is to make use of \emph{Hilbert series}~\cite{Kem2010} and then reduce the family degree per degree. 

\subsection{Hilbert series}\label{subsec:Hilb_Ser}

Recall here that, for a $\mathbb{C}$ graduated algebra of finite type
\ban
	\mathcal{A}=\bigoplus_{k\geq 0} \mathcal{A}_{k},\quad \mathcal{A}_0:=\GC
\ean
where each homogeneous space $\mathcal{A}_{k}$ is of finite dimension $a_{k}$. The Hilbert series associated to $\mathcal{A}$ is the formal series
\ban
	H_{\mathcal{A}}(z):=\sum_{k\geq 0}a_{k}z^k.	 
\ean
For covariant algebras, the Hilbert series can be computed \emph{a priori}, using for example Bedratyuk's Maple package~\cite{Bed2011}. Suppose now that we know a finite bases $\mathcal{F}$ (with homogeneous elements) of the algebra $\mathcal{A}$, which is not necessary minimal. Suppose also that we know a finite minimal family\footnote{This means that for $k'<k$, we have $(\GC[\mathcal{F}_k])_{k'}=\mathcal{A}_{k'}$ and if we take a strict subfamily $\mathcal{G}\varsubsetneq\mathcal{F}_k$ this property is no more true.} $\mathcal{F}_k$ up to a degree $k$. To get up to degree $k+1$:
\begin{itemize}
\item[$\bullet$] Compute a bases for the subspace of $\mathcal{A}_{k+1}$ spanned by elements of $\mathcal{F}_k$; 
\item[$\bullet$] If this dimension's subspace is strictly less than $a_{k+1}$, choose homogeneous elements in $\mathcal{F}$ such that we get a subspace with dimension exactly $a_{k+1}$.
\end{itemize}
We obtain this way a finite minimal bases. But a major weakness of this strategy is that we work with homogeneous spaces which can be of huge dimensions. For instance, for the invariant algebra of $V=\Sn{8}\oplus \Sn{4}\oplus \Sn{4}$ (see \autoref{subsec:Inv_Bas_Elast}), Gordan's algorithm produces a degree $49$ invariant, and for such a homogeneous space we have
\ban
	\dim \mathbf{Cov}_{49,0}(V)=103\: 947\: 673\: 173. 
\ean
which is far beyond our computation means. 

To get one step further, we thus propose to use some algebraic tools to improve Gordan's algorithm. The main idea is to make use of \emph{relations} on covariant algebras. Note here that this idea was suggested by R. Lercier. 

\subsection{Relations on weighted mononomials}
Let 
\ban
	x_1>x_2>\dotsc >x_p
\ean
be indeterminates and $\mathcal{A}=\GC[x_1,\dotsc,x_p]$ be a graduated algebra of finite type. Consider also the lexicographic order on monomials of $\mathcal{A}$. We write $\mathbf{m}_1\mid \mathbf{m}_2$ whenever the monomial $\mathbf{m}_1$ divide monomial $\mathbf{m}_2$.

\begin{hypo}\label{hypo:Rel_Mon_1facteur}
There exists a finite family $I\subset \set{1,\dotsc,p-1}$ of distinct integers and for each $i\in I$ a relation
\begin{equation}\label{Rel:Ri}
	(\mathcal{R}_i),\quad x_i^{a_i}=\sum_{k=0}^{a_i-1} x_i^k\mathbf{p}_k(x_{i+1},\dotsc,x_p),\quad a_i\in \mathbb{N}^* 
\end{equation}
where $\mathbf{p}_k$ is some polynomial. We write $\mathbf{m}_i:=x_i^{a_i}$.
\end{hypo}

\begin{lem}\label{lem:Red_1facteur}
Under hypothesis \ref{hypo:Rel_Mon_1facteur}, the algebra $\mathcal{A}$ is generated by the family of monomials $\mathbf{m}$ such that
\ban
	\mathbf{m}_i\nmid \mathbf{m},\quad \forall i\in I
\ean
\end{lem}

\begin{proof}
We first order the finite family $I=\set{i_1,i_2,\dotsc,i_s}$ such that 
\ban
	x_{i_1}>x_{i_2}>\dotsc
\ean
We then get a direct proof by induction on $s$. 
\end{proof}

\begin{hypo}\label{hypo:Rel_Mon_2facteur}
There exists a finite family $J$ and for each $j\in J$ a relation
\begin{equation}\label{Rel:R'j}
	(\mathcal{R}'_j),\quad x_{j_b}^{b_{j_b}}x_{j_c}^{c_{j_c}}=\mathbf{p}(x_{j_c+1},\dotsc,x_p),\quad b_{j_b},c_{j_c}\in \mathbb{N}^* 
\end{equation}
where $x_{j_b}>x_{j_c}$ and $\mathbf{p}$ is some polynomial. We write $\mathbf{m}'_{j}:=x_{j_b}^{b_{j_b}}x_{j_c}^{c_{j_c}}$.
\end{hypo}

\begin{lem}\label{lem:Reduce_By_Weighted_Relations}
Under hypotheses \ref{hypo:Rel_Mon_1facteur} and \ref{hypo:Rel_Mon_2facteur}, the algebra $\mathcal{A}$ is generated by the family of monomials $\mathbf{m}$ such that
\ban
	\mathbf{m}_i\nmid \mathbf{m},\quad \mathbf{m}'_{j}\nmid \mathbf{m},\quad \forall i\in I,\quad \forall j\in J
\ean
\end{lem}

\begin{proof}
We first order $J=\set{j_1,\dotsc,j_l}$ such that
\ban
	\mathbf{m}'_{j_1}>\mathbf{m}'_{j_2}>\dotsc
\ean	
From lemma~\ref{lem:Red_1facteur} we can take a monomial $\mathbf{m}$ such that $\mathbf{m}_i\nmid \mathbf{m}$ for all $i\in I$. Now Suppose that $\mathbf{m}'_{j}\mid \mathbf{m}$ for one given $j\in J$ which means that
\ban
	\mathbf{m}=x_1^{r_1}\dotsc x_{j_b}^{r_{j_b}}\dotsc x_{j_c}^{r_{j_c}}\dotsc x_{p}^{r_p},\quad r_{j_b}\geq b_{j_b},\quad r_{j_c}\geq c_{j_c}.
\ean
Using relation $\mathcal{R}'_{j}$~\eqref{Rel:R'j} we then have
\ban
	\mathbf{m}=x_1^{r_1}\dotsc x_{j_b}^{r'_{j_b}}\dotsc x_{j_c}^{r'_{j_c}}\mathbf{p}(x_{j_c+1},\dotsc ,p),\quad r'_{j_b}< b_{j_b} \text{ or } r'_{j_c}< c_{j_c}.
\ean
We can also suppose that no monomial in $\mathbf{p}(x_{j_c+1},\dotsc ,p)$ is divided by $\mathbf{m}_i$, otherwise we use some relation $\mathcal{R}_i$~\eqref{Rel:Ri}.

Suppose now that the lemma is true for a given family $J=\set{j_1,\dotsc ,j_l}$. Take $j>j_l$ and
\ban
	(\mathcal{R}'_j),\quad x_{j_b}^{b_{j_b}}x_{j_c}^{c_{j_c}}=\mathbf{p}(x_{j_c+1},\dotsc,x_p),\quad b_{j_b},c_{j_c}\in \mathbb{N}^*. 
\ean
Let
\ban
	\mathbf{m}=x_1^{r_1}\dotsc x_{j_b}^{r_{j_b}}\dotsc x_{j_c}^{r_{j_c}}\dotsc x_{p}^{r_p},\quad r_{j_b}\geq b_{j_b},\quad r_{j_c}\geq c_{j_c}
\ean
such that $\mathbf{m}_i\nmid \mathbf{m}$ for all $i\in I$ and $\mathbf{m}'_{j_t}\nmid \mathbf{m}$ for all $j_t\in J$. Using relation $\mathcal{R}'_{j}$, we decompose $\mathbf{m}$ in monomials
\ban
	\mathbf{n}=x_1^{r_1}\dotsc x_{j_b}^{r'_{j_b}}\dotsc x_{j_c}^{r'_{j_c}}x_{j_c+1}^{r'_{j_c+1}}\dotsc x_p^{r'_p} \quad r'_{j_b}< b_{j_b} \text{ or } r'_{j_c}< c_{j_c}.
\ean
We also know that $r'_{j_b}\leq r_{j_b}$ and $r'_{j_c}\leq r_{j_c}$.

Using lemma~\ref{lem:Red_1facteur}, we can suppose that no any monomial $\mathbf{m}_i$ divides $\mathbf{n}$. If now some monomial $\mathbf{m}'_{k}$ divide $\mathbf{m}'$, we must have $k_c\geq j_c+1$. We thus conclude using relation $\mathcal{R}'_{k}$ and lemma~\ref{lem:Red_1facteur}.
\end{proof}

\subsubsection{Application to joint covariant algorithm}
Let $\mathrm{A}:=\set{ \ff_1,\dotsc,\ff_p}$ (resp. $\mathrm{B}:=\set{\bg_1,\dotsc,\bg_q}$) be a covariant bases of $\Sn{m}$ (resp. $\Sn{n}$). We introduce an order
\ban
	\ff_p < \ff_{p-1} < \dotsc < \ff_1
\ean
on the covariants $\ff_i\in \mathrm{A}$ and we define a lexicographic order on $\cov{\Sn{m}}=\GC[\mathrm{A}]$. We also suppose that there exists relations $\mathcal{R}_i$ and $\mathcal{R}'_{j}$ as in hypotheses~\ref{hypo:Rel_Mon_1facteur} and \ref{hypo:Rel_Mon_2facteur}.

\begin{exemple}\label{ex:Rel_Cov4}
As noted in example~\ref{ex:Cov34} the algebra $\cov{\Sn{4}}$ is generated by $\vv,\mathbf{k}_{2,4},\mathbf{k}_{3,6}$ and the two invariants $i,j$ (where $\vv\in \Sn{4}$). Let 
\ban
	\mathbf{k}_{3,6}>\mathbf{k}_{2,4}>\vv > i > j
\ean
be an order on $\cov{\Sn{4}}$. In that case we have one relation (obtained by a direct computation)
\begin{equation}\label{eq:SyzS4}
	\mathcal{R}_1:\quad 12\mathbf{k}_{3,6}^2=-6\mathbf{k}_{2,4}^3-2j\vv^3+3i\vv^2\mathbf{k}_{2,4}.
\end{equation}
\end{exemple}

Recall now that theorem \ref{thm:CovJoints} applied to $\cov{\Sn{m}\oplus\Sn{n}}$ gives us a finite bases $\mathrm{C}$ of transvectants
\ban
	\tr{\mathbf{U}}{\mathbf{V}}{r}
\ean
related to irreducible solutions $(\boldsymbol{\alpha},\boldsymbol{\beta},u,v,r)$ of the Diophantine system $S(\mathrm{A},\mathrm{B})$ \eqref{Syst:Gordan_Linear}.  

\begin{thm}\label{thm:CovJointsRed}
The algebra $\cov{\Sn{m}\oplus \Sn{n}}$ is generated by the finite subfamily of $\mathrm{C}$ 
\ban
	\tr{\tilde{\mathbf{U}}}{\mathbf{V}}{r}\in \mathrm{C},\quad \mathbf{m}_i\nmid \tilde{\mathbf{U}},\quad \mathbf{m}'_{j}\nmid 
	\tilde{\mathbf{U}},\quad \forall i\in I,\quad \forall j\in J.
\ean
\end{thm}

\begin{proof}
Take one transvectant 
\ban
	\tr{\mathbf{U}}{\mathbf{V}}{r}
\ean
of the family $\mathrm{C}$. Using lemma~\ref{lem:Reduce_By_Weighted_Relations}, we can write the monomial $\mathbf{U}\in \GC[\mathrm{A}]$ as a linear combination of  
\ban
	\tilde{\mathbf{U}},\quad \mathbf{m}_i\nmid \tilde{\mathbf{U}},\quad \mathbf{m}'_{j}\nmid \tilde{\mathbf{U}},\quad \forall i\in I,\quad \forall j\in J
\ean
As in the proof of theorem~\ref{thm:CovJoints}, if the transvectant 
\ban
	\tr{\tilde{\mathbf{U}}}{\mathbf{V}}{r}
\ean
corresponds to a reducible solution $\kappa=\kappa_1+\kappa_2$ of the Diophantine system $S(\mathrm{A},\mathrm{B})$ \eqref{Syst:Gordan_Linear}, it will decompose as transvectants $\tr{\mathbf{U}'}{\mathbf{V}}{r_i}$ corresponding to irreducible solution $\kappa_i$ and transvectants of strictly lower indexes. But we necessary have 
\ban
	\mathbf{m}_i\nmid \mathbf{U'},\quad \mathbf{m}'_{j}\nmid \mathbf{U}',\quad \forall i\in I,\quad \forall j\in J
\ean
and thus we conclude using a direct induction on $r$.  
\end{proof}

Suppose now that we also have hypotheses~\ref{hypo:Rel_Mon_1facteur} and \ref{hypo:Rel_Mon_2facteur} for the algebra $\cov{\Sn{n}}$, which leads to relations $\overline{\mathcal{R}}_k$ ($k\in K$) and $\overline{\mathcal{R}}'_l$ ($l\in L$), specified on monomials $\overline{\mathbf{m}}_k$ and $\overline{\mathbf{m}}'_l$. The finite family $\mathrm{C}$ denote once again the finite family of transvectants
\ban
	\tr{\mathbf{U}}{\mathbf{V}}{r}
\ean
related to irreducible solutions $(\boldsymbol{\alpha},\boldsymbol{\beta},u,v,r)$ of the Diophantine system $S(\mathrm{A},\mathrm{B})$ \eqref{Syst:Gordan_Linear}.

The same proof as in theorem~\ref{thm:CovJointsRed} now leads to the more general result:

\begin{thm}\label{thm:CovJointsRedBis}
The algebra $\cov{\Sn{m}\oplus \Sn{n}}$ is generated by the finite subfamily of $\mathrm{C}$ 
\ban
	\tr{\tilde{\mathbf{U}}}{\tilde{\mathbf{V}}}{r}\in \mathrm{C},\quad 
	\begin{cases}
	\mathbf{m}_i\nmid \tilde{\mathbf{U}},\quad \mathbf{m}'_{j}\nmid 
	\tilde{\mathbf{U}},\quad \forall i\in I,\quad \forall j\in J \\
	\overline{\mathbf{m}}_k\nmid \tilde{\mathbf{V}},\quad \overline{\mathbf{m}}'_{l}\nmid 
	\tilde{\mathbf{V}},\quad \forall k\in K,\quad \forall l\in L \\
	\end{cases}.
\ean
\end{thm}

\subsubsection{Application to simple covariant algorithm}
Let take the case when $V=\Sn{n}$. Recall that in that case Gordan's algorithm deals with families $\mathrm{A}_{0},\mathrm{B}_{0},\dotsc$ (see \autoref{subsec:The_Alg}). Consider the case when the family $\mathrm{B}_{k-1}$ is one covariant bases of the binary form 
\ban
	\mathbf{H}_{2k}=\tr{\ff}{\ff}{2k}.
\ean
In that case $\mathbf{H}_{2k}$ is of order $p<n$ and we suppose known that the related covariant basis. As for theorem~\ref{thm:FamRelComplete}, write $\Delta\in \cov{\Sn{n}}$ to be an invariant and $J_{2k+2}=I_{2k+2}$ or $J_{2k+2}=I_{2k+2}+\langle \Delta \rangle$. Write $\mathrm{A}:=\mathrm{A}_{k-1}$, $\mathrm{B}:=\mathrm{B}_{k-1}$ and note $\mathrm{C}$ to be the finite family of transvectants 
\ban
	\tr{\mathbf{U}}{\mathbf{V}}{r}
\ean
related to irreducible solutions $(\boldsymbol{\alpha},\boldsymbol{\beta},u,v,r)$ of the Diophantine system $S(\mathrm{A},\mathrm{B})$ \eqref{Syst:Gordan_Linear}. 

Finally, suppose that we have hypotheses~\ref{hypo:Rel_Mon_1facteur} and~\ref{hypo:Rel_Mon_2facteur} on the bases $\mathrm{B}$ of the algebra $\cov{\Sn{p}}$, with relations $\mathcal{R}_i$ ($i\in I$) and $\mathcal{R}'_j$ ($j\in J$) on monomials $\mathbf{m}_i$ and $\mathbf{m}'_j$. In that case we have:

\begin{thm}\label{thm:FamRelCompleteBis}
The subfamily $\tilde{\mathrm{C}}$ of $\mathrm{C}$ given by 
\ban
	\tr{\mathbf{U}}{\tilde{\mathbf{V}}}{r}\in \mathrm{C},\quad \mathbf{m}_i\nmid \tilde{\mathbf{V}},\quad \mathbf{m}'_{j}\nmid 
	\tilde{\mathbf{V}},\quad \forall i\in I,\quad \forall j\in J.
\ean
is relatively complete modulo $J_{2k+2}$ and
\begin{equation*}
	\cov{\tilde{\mathrm{C}}}=\cov{\mathrm{A}\cup \mathrm{B}}=\cov{\Sn{n}}.
\end{equation*}
\end{thm}

\begin{proof}
We just take back the proof of theorem~\ref{thm:FamRelComplete} and replace every monomials $\mathbf{V}$ with monomials $\tilde{\mathbf{V}}$; we then make use of the same ideas as in the proof of theorem~\ref{thm:CovJointsRed}.
\end{proof}

\subsection{Invariant's ideal and covariant relations}

\begin{lem}\label{lem:Reduc_Inv_Ideal}
Let $\mathbf{U}\in \cov{\Sn{m}}$ be a covariant. If $\mathbf{U}$ is in the ideal generated by invariants of $\Sn{m}$ then every covariant
\ban
	\tr{\mathbf{U}}{\mathbf{V}}{r},\quad r\geq 0,\quad \mathbf{V}\in \Sn{n}
\ean
is reducible. 
\end{lem}

\begin{proof}
We just observe that if $\mathbf{U}=\Delta\mathbf{U}'$, where $\Delta\in \inv{\Sn{m}}$ and $\mathbf{U}'\in \cov{\Sn{m}}$ then
\ban
	\tr{\Delta\mathbf{U}'}{\mathbf{V}}{r}=\Delta\tr{\mathbf{U}'}{\mathbf{V}}{r}
\ean
is reducible. 
\end{proof}

Now, in a more general case, we take back two families $\mathrm{A},\mathrm{B}$ and the family $\mathrm{C}$ of transvectants 
\ban
	\tr{\mathbf{U}}{\mathbf{V}}{r}
\ean
related to irreducible solutions of the Diophantine system $S(\mathrm{A},\mathrm{B})$ \eqref{Syst:Gordan_Linear}.
\begin{lem}\label{lem:Reduc_Int_Sol}
Let $\hat{\mathbf{U}}=\hat{\mathbf{U}}_1+\hat{\mathbf{U}}_2$ be monomial covariants in $\cov{\Sn{m}}$, $\mathbf{V}$ a monomial covariant in $\Sn{n}$ and suppose that the transvectants
\ban
	\tr{\hat{\mathbf{U}}_1}{\mathbf{V}}{r},\quad \tr{\hat{\mathbf{U}}_2}{\mathbf{V}}{r}
\ean
correspond to reducible integer solutions of the linear Diophantine system \eqref{Syst:Gordan_Linear}. Then the transvectant
\ban	
	\tr{\hat{\mathbf{U}}}{\mathbf{V}}{r}
\ean
is expressible in terms of transvectants of the family $\mathrm{C}$ and of transvectants
\ban	
	\tr{\mathbf{U}'}{\mathbf{V}'}{r'},\quad r'<r
\ean
where $\mathbf{U}'$ (resp. $\mathbf{V}'$) is a monomial in $\GC[\mathrm{A}]$ (resp. $\GC[\mathrm{B}]$). 
\end{lem}

\begin{proof}
This is a direct application of proposition \ref{prop:TermAndTrans} and lemma~\ref{lem:IntRed_EtTransRed}.
\end{proof}


\section{Effective computations}\label{sec:Eff_Comp}


\subsection{Covariant bases of $\Sn{6}\oplus \Sn{2}$}
\label{subsec:Cov62}

There is a simple procedure to produce a covariant bases of $V\oplus \Sn{2}$ once we know a covariant bases of $V$, as detailed in the following theorem, which proof can be found in~\cite{GY2010}.
\begin{thm}\label{thm:CovJointsQuadr}
Let $\lbrace\mathbf{h}_1,\dotsc ,\mathbf{h}_s\rbrace$ be a covariant bases of $\cov{V}$, and let $\uu\in\Sn{2}$. Then irreducible covariants of $\cov{V\oplus \Sn{2}}$ are taken from the sets:
\begin{itemize}
\item[$\bullet$] $\lbrace \mathbf{h}_i,\uu^r\rbrace_{2r-1}$ for $i=1\dotsc s$;
\item[$\bullet$] $\lbrace \mathbf{h}_i,\uu^r\rbrace_{2r}$ for $i=1\dotsc s$;
\item[$\bullet$] $\lbrace \mathbf{h}_i\mathbf{h}_j,\uu^r\rbrace_{2r}$ where $\mathbf{h}_i$ is of order $2p+1$ and $\mathbf{h}_j$ is of order $2r-2p-1$.
\end{itemize}
\end{thm}

Write now $\bh=\mathbf{h}_{d,k}$ to be a covariant of degree $d$ and order $k$, taken from the covariant bases of $\Sn{6}$ in table~\ref{table:CovS6}, issued from Grace--Young~\cite{GY2010}, and $\uu$ to be a quadratic form in $\Sn{2}$. By theorem~\ref{thm:CovJointsQuadr} we only have to consider covariants given by
\begin{equation*}
\lbrace \mathbf{h},\uu^r\rbrace_{2r-1} \text{ or } \lbrace \mathbf{h},\uu^r\rbrace_{2r}.
\end{equation*}

\begin{table}[h]\label{table:CovS6}
\begin{tabular}{|c|c|c|c|c|}
\hline
$d$/$k$ & 0 															& 2 															& 4 														& 6 \\	\hline
1 	& 															& 															& 														& $\ff$ \\ \hline
2 	& $\tr{\ff}{\ff}{6}$									&  															& $\mathbf{h}_{2,4}:=\tr{\ff}{\ff}{4}$ 			& 		\\ \hline
3 	& 															& $\mathbf{h}_{3,2}:=\tr{\mathbf{h}_{2,4}}{\ff}{4}$	&  														& $\mathbf{h}_{3,6}:=\tr{\mathbf{h}_{2,4}}{\ff}{2}$ \\ \hline
4 	& $\tr{\mathbf{h}_{2,4}}{\mathbf{h}_{2,4}}{4}$ 		&  															& $\tr{\mathbf{h}_{3,2}}{\ff}{2}$ 				& $\mathbf{h}_{4,6}:=\tr{\mathbf{h}_{3,2}}{\ff}{1}$ \\ \hline
5 	& 															& $\tr{\mathbf{h}_{2,4}}{\mathbf{h}_{3,2}}{2}$ 		& $\tr{\mathbf{h}_{2,4}}{\mathbf{h}_{3,2}}{1}$		& \\ \hline
6 	& $\tr{\mathbf{h}_{3,2}}{\mathbf{h}_{3,2}}{2}$ 		&  															&  														& \begin{tabular}{l}$\mathbf{h}_{6,61}:=\tr{\mathbf{h}_{3,8}}{\mathbf{h}_{3,2}}{2}$ \\ $\mathbf{h}_{6,62}:=\tr{\mathbf{h}_{3,6}}{\mathbf{h}_{3,2}}{1}$ \end{tabular} \\ \hline
7 	& 															& $\tr{\ff}{\mathbf{h}_{3,2}^2}{4}$ 					& $\tr{\ff}{\mathbf{h}_{3,2}^2}{3}$ 				&  \\ \hline
8 	& 															& $\tr{\mathbf{h}_{2,4}}{\mathbf{h}_{3,2}^2}{3}$ 		& 														& \\ \hline
9 	& 															& 															& $\tr{\mathbf{h}_{3,8}}{\mathbf{h}_{3,2}^2 }{4}$	& \\ \hline
10	& $\tr{\mathbf{h}_{3,2}^3}{\ff}{6}$ 							& $\tr{\mathbf{h}_{3,2}^3}{\ff}{5}$ 					& 														& \\ \hline
12 	& 															& $\tr{\mathbf{h}_{3,8}}{\mathbf{h}_{3,2}^3}{6}$ 		& 														& \\ \hline
15 	& $\tr{\mathbf{h}_{3,8}}{\mathbf{h}_{3,2}^4}{8}$ 				& 															& 														& \\ \hline
\end{tabular}
\begin{tabular}{|c|c|c|c|}
\hline
$d$/$k$ & $8$ 																	& $10$ 													& $12$ \\ \hline
$2$ & $\mathbf{h}_{2,8}:=\tr{\ff}{\ff}{2}$ 							& 														& \\ \hline
$3$ & $\mathbf{h}_{3,8}:=\tr{\mathbf{h}_{2,4}}{\ff}{1}$ 				&  														& $\tr{\mathbf{h}_{2,8}}{\ff}{1}$ \\ \hline
$4$ &  																		& $\tr{\mathbf{h}_{2,8}}{\mathbf{h}_{2,4}}{1}$ 	&  \\ \hline
$5$ & $\mathbf{h}_{5,8}:=\tr{\mathbf{h}_{2,8}}{\mathbf{h}_{3,2}}{1}$ 	& 														& \\ \hline
\end{tabular}
\caption{Covariant bases of $\Sn{6}$}
\end{table}

Recall the covariant algebra $\cov{V}:=\cov{\Sn{6}\oplus\Sn{2}}$ is a multi-graded algebra:
\begin{equation*}
	\cov{V}=\bigoplus_{d_{1}\geq 0,d_{2}\geq 0,k\geq 0} \cov{V}_{d_{1},d_{2},k}.
\end{equation*}
where $d_{1}$ is the degree in the binary form $\ff\in \Sn{6}$, $d_{2}$ is the degree in the binary form $\uu\in \Sn{2}$ and $k$ the degree in the variable $\xx\in \mathbb{C}^2$. We can define the Hilbert series:
\begin{equation*}
	\mathcal{H}_{6,2}(z_{1},z_{2},t):=\sum_{d_{1},d_{2},k} \dim(\cov{V}_{d_{1},d_{2},k})z_{1}^{d_{1}}z_{2}^{d_{2}}t^{k},
\end{equation*}
which has been computed using Bedratyuk's Maple package~\cite{Bed2011}. From this Hilbert series and theorem~\ref{thm:CovJointsQuadr}, we finally get a minimal bases of 99 covariants, already obtained by von Gall~\cite{vGal1874}. It's worth noting that, by using theorem~\ref{thm:CovJointsQuadr}, we only had to check invariant homogeneous space's dimensions up to degree 15. Results are summerized in table~\ref{table:Cov62}.

\begin{table}[H]
\setlength{\arraycolsep}{6pt}
\begin{equation*}
\begin{array}{c||ccccccc|c|c|c}
 d/o &   0 &   2 &   4 &   6 &   8 &  10 &  12 &  \# & Cum \\\hline\hline
   1 &   - &   1 &   - &   1 &   - &   - &   - &   2 &   2 \\
   2 &   2 &   - &   2 &   1 &   1 &   - &   - &   6 &   8 \\
   3 &   - &   3 &   2 &   2 &   2 &   - &   1 &  10 &  18 \\
   4 &   4 &   3 &   3 &   4 &   - &   2 &   - &  16 &  34 \\
   5 &   - &   4 &   6 &   - &   3 &   - &   - &  13 &  47 \\
   6 &   5 &   7 &   - &   5 &   - &   - &   - &  17 &  64 \\
   7 &   3 &   1 &   6 &   - &   - &   - &   - &  10 &  74 \\
   8 &   1 &   8 &   - &   - &   - &   - &   - &   9 &  83 \\
   9 &   7 &   - &   1 &   - &   - &   - &   - &   8 &  91 \\
  10 &   1 &   2 &   - &   - &   - &   - &   - &   3 &  94 \\
  11 &   2 &   - &   - &   - &   - &   - &   - &   2 &  96 \\
  12 &   - &   1 &   - &   - &   - &   - &   - &   1 &  97 \\
  13 &   1 &   - &   - &   - &   - &   - &   - &   1 &  98 \\
  14 &   - &   - &   - &   - &   - &   - &   - &   - &  98 \\
  15 &   1 &   - &   - &   - &   - &   - &   - &   1 &  99 \\\hline
 Tot &  27 &  30 &  20 &  13 &   6 &   2 &   1 &     &  99  \\
\end{array}
\end{equation*}
\caption{Minimal covariant bases of $\Sn{6}\oplus\Sn{2}$.}\label{table:Cov62}
\end{table}

\renewcommand{\arraystretch}{1.2}
\tablefirsthead{%
\multicolumn{3}{c}{}\\
}
\tablehead{%
\multicolumn{3}{r}{\small\sl continued from previous page} \\ \hline
}
\tabletail{%
\multicolumn{3}{r}{\small\sl continued on next page}
\\
}
\tablelasttail{\hline}
\begin{supertabular}{|l|ll|}
\hline
\multicolumn{3}{|c|}{Ordre $0$ : $5$ invariants from $\Sn{6}$, $1$ invariant from $\Sn{2}$ and $21$ joint invariants.} \\ \hline
Degree 2\quad		&	$\tr{\ff}{\ff}{6}$									\quad
							$\tr{\mathbf{u}}{\mathbf{u}}{2}$	 				& \\ \hline
Degree 4\quad		&	$\tr{\mathbf{h}_{1,6}}{\mathbf{u}^3}{6}$			\quad
							$\tr{\mathbf{h}_{2,4}}{\mathbf{u}^2}{4}$			\quad
							$\tr{\mathbf{h}_{3,2}}{\mathbf{u}}{2}$ 			&\\ \hline
Degree 6\quad		&	$\tr{\mathbf{h}_{3,6}}{\mathbf{u}^3}{6}$			\quad
							$\tr{\mathbf{h}_{2,8}}{\mathbf{u}^4}{8}$			\quad
							$\tr{\mathbf{h}_{4,4}}{\mathbf{u}^2}{4}$			\quad
							$\tr{\mathbf{h}_{5,2}}{\mathbf{u}}{2}$ 			&\\ \hline
Degree 7\quad		&	$\tr{\mathbf{h}_{5,4}}{\mathbf{u}^2}{4}$			\quad
							$\tr{\mathbf{h}_{3,8}}{\mathbf{u}^4}{8}$			\quad
							$\tr{\mathbf{h}_{4,6}}{\mathbf{u}^3}{6}$ 			&\\ \hline
Degree 8\quad		&	$\tr{\mathbf{h}_{7,2}}{\mathbf{u}}{2}$ 			&\\ \hline
Degree 9\quad		&	$\tr{\mathbf{h}_{7,4}}{\mathbf{u}^2}{4}$			\quad
							$\tr{\mathbf{h}_{6,61}}{\mathbf{u}^3}{6}$			\quad
							$\tr{\mathbf{h}_{4,10}}{\mathbf{u}^5}{10}$			&\\ \hline
						&	$\tr{\mathbf{h}_{5,8}}{\mathbf{u}^4}{8}$			\quad
							$\tr{\mathbf{h}_{8,2}}{\mathbf{u}}{2}$				\quad
							$\tr{\mathbf{h}_{3,12}}{\mathbf{u}^6}{12}$			\quad
							$\tr{\mathbf{h}_{6,62}}{\mathbf{u}^3}{6}$			&\\ \hline
Degree 10\quad	&	$\tr{\mathbf{h}_{3,2}^{3}}{\mathbf{f}}{6}$ 		&\\ \hline
Degree 11\quad	&	$\tr{\mathbf{h}_{10,2}}{\mathbf{u}}{2}$			\quad
							$\tr{\mathbf{h}_{9,4}}{\mathbf{u}^2}{4}$ 			&\\ \hline
Degree 13\quad 	&	$\tr{\mathbf{h}_{12,2}}{\mathbf{u}}{2}$ 			&\\ \hline
Degree 14	\quad	&	$\tr{\mathbf{h}_{3,2}^{4}}{\mathbf{h}_{3,8}}{8}$	& \\ \hline
\multicolumn{3}{|c|}{Order $2$ : $1$ from $\Sn{2}$, $6$ from $\Sn{6}$ and $23$ joint covariants.} \\ \hline
Degree 1\quad		&	$\mathbf{u}$										& \\ \hline
Degree 3\quad		&	$\mathbf{h}_{3,2}$									\quad
							$\tr{\ff}{\mathbf{u}^2}{4}$						\quad
							$\tr{\mathbf{h}_{2,4}}{\mathbf{u}}{2}$				& \\ \hline
Degree 4\quad		&	$\tr{\mathbf{h}_{2,4}}{\mathbf{u}^2}{3}$			\quad
							$\tr{\mathbf{h}_{3,2}}{\mathbf{u}}{1}$				\quad
							$\tr{\ff}{\mathbf{u}^3}{5}$						& \\ \hline
Degree 5\quad		&	$\mathbf{h}_{5,2}$									\quad
							$\tr{\mathbf{h}_{2,8}}{\mathbf{u}^3}{6}$			\quad
							$\tr{\mathbf{h}_{4,4}}{\mathbf{u}}{2}$				\quad
							$\tr{\mathbf{h}_{3,6}}{\mathbf{u}^2}{4}$			& \\ \hline
Degree 6\quad		&	$\tr{\mathbf{h}_{2,8}}{\mathbf{u}^4}{7}$			\quad
							$\tr{\mathbf{h}_{4,4}}{\mathbf{u}^2}{3}$			\quad
							$\tr{\mathbf{h}_{5,2}}{\mathbf{u}}{1}$				\quad
							$\tr{\mathbf{h}_{3,6}}{\mathbf{u}^3}{5}$			& \\ \hline
						&	$\tr{\mathbf{h}_{5,4}}{\mathbf{u}}{2}$				\quad
							$\tr{\mathbf{h}_{3,8}}{\mathbf{u}^3}{6}$			\quad
							$\tr{\mathbf{h}_{4,6}}{\mathbf{u}^2}{4}$			& \\ \hline
Degree 7\quad		&	$\mathbf{h}_{7,2}$									& \\ \hline
Degree 8\quad		&	$\mathbf{h}_{8,2}$									\quad
							$\tr{\mathbf{h}_{7,2}}{\mathbf{u}}{1}$				\quad
							$\tr{\mathbf{h}_{7,4}}{\mathbf{u}}{2}$				\quad
							$\tr{\mathbf{h}_{6,6b}}{\mathbf{u}^2}{4}$			& \\ \hline
						&	$\tr{\mathbf{h}_{5,8}}{\mathbf{u}^3}{6}$			\quad
							$\tr{\mathbf{h}_{3,12}}{\mathbf{u}^5}{10}$			\quad
							$\tr{\mathbf{h}_{4,10}}{\mathbf{u}^4}{8}$			\quad
							$\tr{\mathbf{h}_{6,6a}}{\mathbf{u}^2}{4}$			& \\ \hline
Degree 10\quad	&	$\mathbf{h}_{10,2}$								\quad
							$\tr{\mathbf{h}_{9,4}}{\mathbf{u}}{2}$				& \\ \hline
Degree 12\quad	&	$\mathbf{h}_{12,2}$								& \\	 \hline
\multicolumn{3}{|c|}{Order $4$ : $5$ covariants from $\Sn{6}$ and $15$ joint covariants.} \\ \hline
Degree 2\quad		&	$\mathbf{h}_{2,4}$									\quad
							$\tr{\mathbf{f}}{\uu}{2}$							&\\ \hline
Degree 3\quad		&	$\tr{\mathbf{h}_{2,4}}{\uu}{1}$					\quad
							$\tr{\ff}{\uu^2}{3}$								&\\ \hline
Degree 4\quad		&	$\mathbf{h}_{4,4}$									\quad
							$\tr{\mathbf{h}_{3,6}}{\uu}{2}$					\quad
							$\tr{\mathbf{h}_{2,8}}{\uu^2}{4}$					&\\ \hline
Degree 5\quad		&	$\mathbf{h}_{5,4}$									\quad
							$\tr{\mathbf{h}_{3,8}}{\mathbf{u}^2}{4}$			\quad
							$\tr{\mathbf{h}_{3,6}}{\mathbf{u}^2}{3}$			\quad
							$\tr{\mathbf{h}_{4,4}}{\mathbf{u}}{1}$				\quad
							$\tr{\mathbf{h}_{4,6}}{\mathbf{u}}{2}$				\quad
							$\tr{\mathbf{h}_{2,8}}{\mathbf{u}^3}{5}$			&\\ \hline
Degree 7\quad		&	$\mathbf{h}_{7,4}$									\quad
							$\tr{\mathbf{h}_{6,61}}{\mathbf{u}}{2}$			\quad
							$\tr{\mathbf{h}_{3,12}}{\mathbf{u}^4}{8}$			\quad
							$\tr{\mathbf{h}_{4,10}}{\mathbf{u}^3}{6}$			\quad
							$\tr{\mathbf{h}_{6,62}}{\mathbf{u}}{2}$			\quad
							$\tr{\mathbf{h}_{5,8}}{\mathbf{u}^2}{4}$			&\\ \hline
Degree 9\quad		&	$\mathbf{h}_{9,4}$									&	\\ \hline
\multicolumn{3}{|c|}{Order $6$ : $5$ covariants from $\Sn{6}$ and $8$ joint covariants.} \\ \hline
Degree 1\quad		&	$\ff	$											&\\ \hline
Degree 2\quad		&	$\tr{\ff}{\uu}{1}$									&\\ \hline
Degree 3\quad		&	$\mathbf{h}_{3,6}$									\quad
							$\tr{\mathbf{h}_{2,8}}{\uu}{2}$					&\\ \hline
Degree 4\quad		&	$\mathbf{h}_{4,6}$									\quad
							$\tr{\mathbf{h}_{2,8}}{\uu^2}{3}$					\quad
							$\tr{\mathbf{h}_{3,6}}{\uu}{1}$					\quad
							$\tr{\mathbf{h}_{3,8}}{\uu}{2}$					&\\ \hline
Degree 7\quad		&	$\mathbf{h}_{6,61}$								\quad
							$\mathbf{h}_{6,62}$								\quad
							$\tr{\mathbf{h}_{5,8}}{\uu}{2}$					\quad
							$\tr{\mathbf{h}_{4,10}}{\uu^2}{4}$					\quad
							$\tr{\mathbf{h}_{3,12}}{\uu^3}{6}$					&\\ \hline		
\multicolumn{3}{|c|}{Order $8$ : $3$ covariants from $\Sn{6}$ and $3$ joint covariants.} \\ \hline
Degree 2\quad		&	$\mathbf{h}_{2,8}$									&\\ \hline
Degree 3\quad		&	$\mathbf{h}_{3,8}$									\quad
							$\tr{\mathbf{h}_{2,8}}{\uu}{1}$					&\\ \hline
Degree 5\quad		&	$\mathbf{h}_{5,8}$									\quad
							$\tr{\mathbf{h}_{4,10}}{\uu}{2}$					\quad
							$\tr{\mathbf{h}_{3,12}}{\uu^2}{4}$					&\\ \hline
\multicolumn{3}{|c|}{Order $10$ : $1$ covariant from $\Sn{6}$ and $1$ joint covariant.} \\ \hline
Degree 3 \quad		&	$\mathbf{h}_{4,10}$								\quad
						$\tr{\mathbf{h}_{3,12}}{\uu}{2}$					&\\ \hline
\multicolumn{3}{|c|}{Order $12$ : $1$ covariant from $\Sn{6}$.} \\ \hline						
Degree 3\quad		&	$\mathbf{h}_{3,12}$								&\\ \hline
\end{supertabular}


\subsection{Covariant bases of $\Sn{6}\oplus\Sn{4}$}
\label{subsec:Cov64}

Taking $\ff\in \Sn{6}$ and $\vv\in  \Sn{4}$ we take generators $\mathbf{h}_{d,k}$ of $\cov{\Sn{6}}$ given by~\ref{table:CovS6} and we write
\begin{align*}
\vv\in \Sn{4},\quad	&	\mathbf{k}_{2,4}:=\tr{\vv}{\vv}{2},\quad \mathbf{k}_{3,6}:=\tr{\vv}{\mathbf{k}_{2,4}}{1}, \\
i:=\tr{\vv}{\vv}{4},\quad		& j:=\tr{\vv}{\mathbf{k}_{2,4}}{4}.
\end{align*}
We already know one relation on $\cov{\Sn{4}}$ (see example~\ref{ex:Rel_Cov4}):
\begin{equation}\label{Rel:CovS4}
	\mathcal{R}_1:\quad 12\mathbf{k}_{3,6}^2=-6\mathbf{k}_{2,4}^3-2j\vv^3+3i\vv^2\mathbf{k}_{2,4}.
\end{equation}
and thus we have hypothesis~\ref{hypo:Rel_Mon_1facteur} on that algebra. 

Let put on $\cov{\Sn{6}}$ the order
\begin{align*}
\bh_{2,0}>\bh_{4,0}>\bh_{6,0}>\bh_{10,0}>\bh_{15,0}>\bh_{3,2}>\bh_{2,4}>\bh_{4,4}>\ff>\bh_{3,6}>\\
\bh_{4,6}>\bh_{5,4}>\bh_{2,8}>\bh_{6,61}>\bh_{6,62}>\bh_{3,8}>\bh_{7,4}>\bh_{5,2}>\bh_{7,2}>\\
\bh_{9,4}>\bh_{12,2}>\bh_{10,2}>\bh_{8,2}>\bh_{5,8}>\bh_{4,10}>\bh_{3,12}.
\end{align*}

\begin{lem}\label{lem:Rel_CovS6}
There exists, in the algebra $\cov{\Sn{6}}$, 12 relations as in hypothesis~\ref{hypo:Rel_Mon_1facteur} and one relation as in hypothesis~\ref{hypo:Rel_Mon_2facteur}. The monomials $\mathbf{m}_i$ and $\mathbf{m}'_j$ occurring in those relations are
\begin{align*}
\bh_{12,2}^2,&\quad \bh_{10,2}^2,\quad \bh_{8,2}^2,\quad \bh_{7,2}^2,\quad \bh_{9,4}^2,\quad \bh_{7,4}^2,\quad \bh_{5,4}^2 \\
\bh_{6,61}^2,&\quad \bh_{6,62}^2,\quad \bh_{3,12}^2,\quad \bh_{5,8}^2,\quad \bh_{4,10}^2,\quad \bh_{12,2}\bh_{10,2}. 
\end{align*}
\end{lem}

\begin{proof}
To get one relation $\mathcal{R}$ on a monomial $\mathbf{m}$, we consider the homogeneous space $\mathbf{Cov}_{d,k}$ associated to the monomial $\mathbf{m}$. Take for instance the monomial $\bh_{3,12}^2\in \mathbf{Cov}_{6,24}$. By a direct computation in Macaulay2~\cite{M2}, we get that the space $\mathbf{Cov}_{6,24}$ is spanned by the family
\ban
	\set{\bh_{2,0}\ff^4;\bh_{2,4}\bh_{1,6}^2\bh_{2,8};\ff^3\bh_{3,6};\bh_{2,8}^3;\bh_{3,12}^2}	.
\ean  
Now, writing $\ff=\ff(a_0,\dotsc,a_6,x,y)$ we compute the exact expression of those monomials in $(a_i,x,y)$ and, by computing a kernel, we directly obtain the relation
\ban
	36\bh_{3,12}^2+\bh_{2,0}\ff^4-6\ff^3\bh_{3,6}-9\bh_{2,4}\ff^2\bh_{2,8}+18\bh_{2,8}^3=0.
\ean
We do in the same way for all other relations.  
\end{proof}

We are now able to compute a minimal covariant bases of $\cov{\Sn{6}\oplus\Sn{4}}$:
\begin{itemize}
\item[$\bullet$] As a first step, we use Normaliz package~\cite{BI2010} developed in Macaulay2~\cite{M2} to compute irreducible solutions of the linear Diophantine system associated to the covariant basis $\mathrm{A}$ and $\mathrm{B}$ of $\cov{\Sn{6}}$ and $\cov{\Sn{4}}$. To get such a system, we only have to deal with non-invariant covariant of each bases, which then leads to a system of $21+3+3=27$ unknowns
\ban
\begin{cases}
2(\alpha_1+\dotsc+\alpha_6)+4(\alpha_7+\dotsc+\alpha_{11})+6(\alpha_{12}+\dotsc+\alpha_{16})+\\
\quad\quad\quad\quad 8(\alpha_{17}+\alpha_{18}+\alpha_{19})+10\alpha_{20}+12\alpha_{21}&=u+r \\
\quad\quad\quad\quad 4(\beta_1+\beta_2)+6\beta_3&=v+r
\end{cases}.
\ean
We thus obtain on this first step a finite family of 1732 transvectants. 
\item[$\bullet$] By the known relations~\eqref{Rel:CovS4} and by lemma~\ref{lem:Rel_CovS6}, we can use theorem~\ref{thm:CovJointsRedBis} and thus we get a first reduction process that leads to a finite family of 1134 transvectants. 
\item[$\bullet$] We compute now degree per degree, making use of the \emph{multigraduated Hilbert series} of the algebra 
\ban
	\cov{\Sn{6}\oplus\Sn{4}}=\bigoplus_{d_1,d_2,k} \mathbf{Cov}_{d_1,d_2,k}(\Sn{6}\oplus\Sn{4})
\ean
where $d_1$ is on degree on $\ff\in \Sn{6}$, $d_2$ is the degree on $\vv\in \Sn{4}$ and $k$ is the order of the covariant. Such a multigraduated series can be directly computed using Bedratyuk's Maple package~\cite{Bed2011}.
\end{itemize}

We then organize all the 1134 transvectants using orders, then using degrees, which is summarized in table~\ref{table:Red_CovS6S4}.

\begin{center}
\begin{table}[H]
\begin{tabular}{|c|c|c|c|c|c|c|c|c|c|}
\hline 
Order & 0 & 2 & 4 & 6 & 8 & 10 & 12 & 14 & 16 \\ 
\hline 
\# of trans. & 365 & 462 & 144 & 78 & 46 & 24 & 10 & 4 & 1 \\ 
\hline 
Degrees & 3--30 & 2--30 & 1--23 & 1--15 & 2--12 & 3--9 & 3--8 & 4--7 & 6 \\ 
\hline 
\end{tabular} 
\caption{Finite generating family after reduction process.}\label{table:Red_CovS6S4}
\end{table}
\end{center}

\begin{thm}\label{thm:Cov_64}
The covariant algebra $\cov{\Sn{6}\oplus\Sn{4}}$ is generated by a minimal bases of $194$ elements, summarized in table~\ref{table:Cov64}.
\end{thm}

\begin{proof}
We have to get a minimal family using the 1134 transvectants summerized in table~\ref{table:Red_CovS6S4}. We do this order per order, then degree per degree, using multigraduated Hilbert series of the algebra $\cov{\Sn{6}\oplus\Sn{4}}$. For instance, we have 365 order 0 covariants (which are invariants) from degree 3 up to degree 30. Furthermore, for high level degree, we know exactly which homogeneous spaces $\mathbf{Cov}_{d_1,d_2,0}(\Sn{6}\oplus\Sn{4})$ occur (table~\ref{table:Inv_S64}) in the given family.
\begin{center}
\begin{table}[H]
\begin{tabular}{|c|c|c|c|c|}
\hline 
Degree & $30$ & $28$ & $27$ & $26$ \\ 
\hline 
$d_1$ & $27$ & $25$ & $24$ & $23$ \\ 
\hline 
$d_2$ & $3$ & $3$ & $3$ & $3$ \\ 
\hline 
Dimension & $639$ & $518$ & $534$ & $413$ \\ 
\hline 
\end{tabular} 
\caption{Dimension of homoegenous spaces for high degree invariants.}\label{table:Inv_S64}
\end{table}
\end{center}
Then, using scripts written in Macaulay2~\cite{M2} and algorithm explained in \autoref{subsec:Hilb_Ser}, we get a finite minimal family of $60$ invariants. We then do in the same way for each order. 
\end{proof}

\begin{table}[H]
\begin{equation*}
\setlength{\arraycolsep}{6pt}
\begin{array}{c||ccccccc|c|c|c}
 d/o &   0 &   2 &   4 &   6 &   8 &  10 &  12 &  \# & Cum \\\hline\hline
   1 &   - &   - &   1 &   1 &   - &   - &   - &   2 &   2 \\
   2 &   2 &   1 &   3 &   1 &   2 &   - &   - &   9 &  11 \\
   3 &   2 &   4 &   4 &   5 &   3 &   1 &   1 &  20 &  31 \\
   4 &   4 &   6 &   9 &   5 &   2 &   1 &   - &  27 &  58 \\
   5 &   4 &  12 &  11 &   3 &   1 &   - &   - &  31 &  89 \\
   6 &   9 &  14 &   6 &   2 &   - &   - &   - &  31 & 120 \\
   7 &   9 &  17 &   2 &   - &   - &   - &   - &  28 & 148 \\
   8 &   9 &   7 &   1 &   - &   - &   - &   - &  17 & 165 \\
   9 &   8 &   3 &   1 &   - &   - &   - &   - &  12 & 177 \\
  10 &   5 &   2 &   - &   - &   - &   - &   - &   7 & 184 \\
  11 &   3 &   1 &   - &   - &   - &   - &   - &   4 & 188 \\
  12 &   2 &   1 &   - &   - &   - &   - &   - &   3 & 191 \\
  13 &   1 &   - &   - &   - &   - &   - &   - &   1 & 192 \\
  14 &   1 &   - &   - &   - &   - &   - &   - &   1 & 193 \\
  15 &   1 &   - &   - &   - &   - &   - &   - &   1 & 194 \\\hline
 Tot &   60 & 68 &  38 &  17 &   8 &   2 &   1 &     & 194  \\
\end{array}
\end{equation*}
\caption{Minimal covariant bases of $\Sn{6}\oplus\Sn{4}$}\label{table:Cov64}
\end{table}

We give now transvectants expression of this minimal covariant bases. 

\begin{center}
\renewcommand{\arraystretch}{1.2}
\tablefirsthead{%
\multicolumn{3}{c}{}\\
}
\tablehead{%
\multicolumn{3}{r}{\small\sl continued from previous page} \\ \hline
}
\tabletail{%
\multicolumn{3}{r}{\small\sl continued on next page}
\\
}
\tablelasttail{\hline}
\begin{supertabular}{|l|ll|}
\hline
\multicolumn{3}{|c|}{Order 0 : $5$ invariants from $\Sn{6}$, $2$ invariants from $\Sn{4}$ and $53$ joint invariants} \\ \hline
Degree 3\quad 	&	$\tr{\mathbf{h}_{2,4}}{\vv}{4}$ 													&\\ \hline
Degree 4\quad		&	$\tr{\mathbf{h}_{2,8}}{\vv^2}{8}$ \quad $\tr{\mathbf{h}_{2,4}}{\mathbf{k}_{2,4}}{ 4}$	\quad 									$\tr{\ff}{\mathbf{k}_{3,6}}{6}$													&\\ \hline
Degree 5\quad		& 	$\tr{\mathbf{h}_{3,8}}{\vv^2}{8}$\quad $\tr{\mathbf{h}_{4,4}}{\vv}{4}$ 				\quad 									$\tr{\mathbf{h}_{2,8}}{\vv\cdot\mathbf{k}_{2,4}}{8}$\quad $\tr{\ff^2}{\vv^3}{12}$ 	&\\ \hline
Degree 6\quad		&	$\tr{\mathbf{h}_{3,8}}{\vv\cdot\mathbf{k}_{2,4}}{ 8}$								\quad
							$\tr{\ff^2}{\vv^2\cdot\mathbf{k}_{2,4}}{ 12} 	$									\quad
							$\tr{\mathbf{h}_{2,8}}{\mathbf{k}_{2,4}^2}{8}$ 										\quad	
							$\tr{\mathbf{h}_{3,6}}{\mathbf{k}_{3,6}}{6}$ 										&\\ \hline	
						&	$\tr{\mathbf{h}_{3,12}}{\vv^3}{12}$ 												\quad	
							$\tr{\mathbf{h}_{5,4}}{\vv}{4}$													&\\ \hline
						&	$\tr{\mathbf{h}_{4,4}}{\mathbf{k}_{2,4}}{4}$										\quad
							$\tr{\mathbf{h}_{3,2}\cdot\ff}{\vv^2}{8}$											&\\ \hline
Degree 7\quad		&	$\tr{\mathbf{h}_{3,2}^2}{\vv}{ 4}$													\quad
							$\tr{\mathbf{h}_{5,4}}{\mathbf{k}_{2,4}}{ 4}$ 										\quad
							$\tr{\mathbf{h}_{5,8}}{\vv^2}{8}$ 													\quad
							$\tr{\ff\cdot\mathbf{h}_{3,6}}{\vv^3}{12}$ 										&\\ \hline
						&	$\tr{\ff^2}{\vv\cdot\mathbf{k}_{2,4}^2}{12}$ 										\quad
							$\tr{\mathbf{h}_{3,2}\cdot\ff}{\vv\cdot\mathbf{k}_{2,4}}{ 8}$						&\\ \hline
						&	$\tr{\mathbf{h}_{4,6}}{\mathbf{k}_{3,6}}{ 6}$ 										\quad
							$\tr{\mathbf{h}_{3,12}}{\vv^2\cdot\mathbf{k}_{2,4}}{ 12}$ 							\quad
							$\tr{\mathbf{h}_{3,8}}{\mathbf{k}_{2,4}^2}{8}$  									&\\ \hline
Degree 8\quad		& 	$\tr{\mathbf{h}_{3,2} \mathbf{h}_{2,4}}{\mathbf{k}_{3,6}}{ 6}$						\quad
							$\tr{\mathbf{h}_{3,12}}{\vv\cdot\mathbf{k}_{2,4}^2}{12}$ 							\quad
							$\tr{\mathbf{h}_{3,2} \mathbf{h}_{3,6}}{\vv^2}{8}$ 									\quad
							$\tr{\mathbf{h}_{3,2}^2}{\mathbf{k}_{2,4}}{ 4}$ 									&\\ \hline
						&	$\tr{\mathbf{h}_{7,4}}{\vv}{4}$ 													\quad
							$\tr{\ff\cdot\mathbf{h}_{4,6}}{\vv^3}{12}$ 										&\\ \hline
						&	$\tr{\ff\cdot\mathbf{h}_{3,6}}{\vv^2\cdot\mathbf{k}_{2,4}}{ 12}$ 					\quad
							$\tr{\mathbf{h}_{3,2}\cdot\ff}{\mathbf{k}_{2,4}^2}{8}$ 								\quad
							$\tr{\mathbf{h}_{5,8}}{\vv\cdot\mathbf{k}_{2,4}}{ 8}$								&\\ \hline
Degree 9\quad 	&	$\tr{\mathbf{h}_{7,4}}{\mathbf{k}_{2,4}}{4}$ 										\quad
							$\tr{\mathbf{h}_{3,2}\cdot\mathbf{h}_{5,2}}{\vv}{4}$ 								\quad
							$\tr{\mathbf{h}_{5,2}\cdot\ff}{\vv\cdot\mathbf{k}_{2,4}}{8}$ 						&\\ \hline
						&	$\tr{\mathbf{h}_{3,12}}{\mathbf{k}_{2,4}^3}{12}$ 									\quad
							$\tr{\mathbf{h}_{3,2}\cdot\mathbf{h}_{2,8}}{\vv\cdot\mathbf{k}_{3,6}}{ 10}$ 		&\\ \hline
						&	$\tr{\ff \mathbf{h}_{4,6}}{\vv^2\cdot\mathbf{k}_{2,4}}{ 12}$ 						\quad
							$\tr{\mathbf{h}_{3,6}^2}{\vv^3}{12}$ 												\quad
							$\tr{\mathbf{h}_{3,2}\cdot\mathbf{h}_{4,6}}{\vv^2}{8}$ 								&\\ \hline
Degree 10\quad	&	$\tr{\mathbf{h}_{9,4}}{\vv}{ 4}$ 													\quad
							$\tr{\mathbf{h}_{3,2}\cdot\mathbf{h}_{2,8}}{\mathbf{k}_{2,4} \mathbf{k}_{3,6}}{ 10}$\quad
							$\tr{\mathbf{h}_{5,2}\cdot\mathbf{h}_{3,6}}{\vv^2}{8}$ 								\quad
							$\tr{\ff\cdot\mathbf{h}_{6,61}}{\vv^3}{12}$ 										&\\ \hline
Degree 11	\quad	&	$\tr{\mathbf{h}_{5,2}^2}{\vv}{ 4}$ 												\quad
							$\tr{\ff\cdot\mathbf{h}_{6,62}}{\vv^2\cdot\mathbf{k}_{2,4}}{ 12}$ 					\quad
							$\tr{\mathbf{h}_{3,2}\cdot\mathbf{h}_{6,61}}{\vv^2}{8}$								&\\ \hline
Degree 12\quad	&	$\tr{\mathbf{h}_{3,2} \mathbf{h}_{8,2}}{\vv}{4}$									\quad
							$\tr{\mathbf{h}_{3,2} \mathbf{h}_{6,62}}{\vv \mathbf{k}_{2,4}}{ 8}$ 				&\\ \hline
Degree 13 \quad	&	$\tr{\mathbf{h}_{8,2} \mathbf{h}_{3,6}}{\vv^2}{8}$ 									&\\ \hline
Degree 14 \quad	&	$\tr{\mathbf{h}_{3,2} \mathbf{h}_{10,2}}{\vv}{4}$									& \\ \hline
\multicolumn{3}{|c|}{Order $2$ : $6$ covariants from $\Sn{6}$ and $62$ joint covariants.} \\ \hline
\text{Degree 2}				&	$\tr{\ff}{\vv}{4}$										&\\ \hline
\text{Degree 3}\quad		&	$\mathbf{h}_{3,2}										$\quad
							$\tr{\mathbf{h}_{2,4}}{\vv}{3}						$\quad
							$\tr{\ff}{\mathbf{k}_{2,4}}{4}						$\quad
							$\tr{\ff}{\vv^2}{6}$									&\\ \hline
\text{Degree 4}\quad		& 	$\tr{\mathbf{h}_{2,4}}{\mathbf{k}_{2,4}}{3}			$\quad
							$\tr{\ff}{\vv\cdot\mathbf{k}_{2,4}}{ 6}				$\quad
							$\tr{\ff}{\mathbf{k}_{3,6}}{5}						$\quad
							$\tr{\mathbf{h}_{2,8}}{\vv^2}{7}						$\quad
							$\tr{\mathbf{h}_{3,2}}{\vv}{2}						$\quad
							$\tr{\mathbf{h}_{3,6}}{\vv}{4}$						&\\ \hline
\text{Degree 5}\quad		&	$\mathbf{h}_{5,2}										$\quad
							$\tr{\mathbf{h}_{3,6}}{\mathbf{k}_{2,4}}{4}			$\quad
 							$\tr{\mathbf{h}_{4,4}}{\vv}{3}						$\quad
							$\tr{\mathbf{h}_{3,6}}{\vv^2}{6}$						&\\ \hline
						&	$\tr{\mathbf{h}_{2,8}}{\vv\cdot\mathbf{k}_{2,4}}{7}	$\quad
							$\tr{\ff}{\mathbf{k}_{2,4}^2}{6}						$\quad
							$\tr{\mathbf{h}_{3,2}}{\mathbf{k}_{2,4}}{ 2}$			&\\ \hline					
						&	$\tr{\ff^2}{\vv^3}{11}								$\quad
							$\tr{\mathbf{h}_{4,6}}{\vv}{ 4}						$\quad
							$\tr{\mathbf{h}_{2,8}}{\mathbf{k}_{3,6}}{ 6}			$\quad
 							$\tr{\mathbf{h}_{2,4}}{\mathbf{k}_{3,6}}{ 4}			$\quad
 							$\tr{\mathbf{h}_{3,8}}{\vv^2}{7}$ 						&\\ \hline
\text{Degree 6}\quad		&	$\tr{\ff^2}{\vv^2\cdot\mathbf{k}_{2,4}}{11}			$\quad
 							$\tr{\mathbf{h}_{2,8}}{\vv\cdot\mathbf{k}_{3,6}}{ 8}	$\quad
 							$\tr{\mathbf{h}_{3,2}\cdot\ff}{\vv^2}{7}				$\quad
 							$\tr{\mathbf{h}_{2,8}}{\mathbf{k}_{2,4}^2}{7}$			&\\ \hline
 						&	$\tr{\mathbf{h}_{4,4}}{\mathbf{k}_{2,4}}{ 3}			$\quad
 							$\tr{\mathbf{h}_{4,10}}{\vv^2}{8}						$\quad
							$\tr{\mathbf{h}_{3,12}}{\vv^3}{11}					$\quad
 							$\tr{\mathbf{h}_{5,2}}{\vv}{2}$						&\\ \hline
 						&	$\tr{\mathbf{h}_{4,6}}{\vv^2}{6}						$\quad
 							$\tr{\mathbf{h}_{3,6}}{\vv\cdot\mathbf{k}_{2,4}}{6}	$\quad
							$\tr{\mathbf{h}_{4,6}}{\mathbf{k}_{2,4}}{4}			$\quad
 							$\tr{\mathbf{h}_{3,8}}{\vv\cdot\mathbf{k}_{2,4}}{7}$	&\\ \hline
 						&	$\tr{\mathbf{h}_{3,8}}{\mathbf{k}_{3,6}}{6}			$\quad
 							$\tr{\mathbf{h}_{5,4}}{\vv}{3}$						&\\ \hline
\text{Degree 7}\quad		&	$\mathbf{h}_{7,2}										$\quad
 							$\tr{\mathbf{h}_{2,8}}{\mathbf{k}_{2,4}\cdot\mathbf{k}_{3,6}}{ 8}$\quad
 							$\tr{\mathbf{h}_{6,62}}{\vv}{ 4}						$\quad
 							$\tr{\mathbf{h}_{3,12}}{\vv^2\cdot\mathbf{k}_{2,4}}{11} 	$\quad
 							$\tr{\mathbf{h}_{4,10}}{\vv^3}{10}					$\quad
 							$\tr{\mathbf{h}_{6,61}}{\vv}{ 4}$						&\\ \hline
 						&	$\tr{\ff\cdot\mathbf{h}_{3,6}}{\vv^3}{11}				$\quad
							$\tr{\mathbf{h}_{3,2}^2}{\vv}{ 3}						$\quad
 							$\tr{\mathbf{h}_{5,2}}{\mathbf{k}_{2,4}}{ 2}			$\quad
 							$\tr{\mathbf{h}_{3,8}}{\vv\cdot\mathbf{k}_{3,6}}{8}	$\quad
 							$\tr{\mathbf{h}_{2,4}^2}{\mathbf{k}_{3,6}}{6}			$\quad
 							$\tr{\mathbf{h}_{5,8}}{\vv^2}{7}$						&\\ \hline
 						&	$\tr{\mathbf{h}_{4,6}}{\vv\cdot\mathbf{k}_{2,4}}{6}	$\quad
 							$\tr{\ff^2}{\vv\cdot\mathbf{k}_{2,4}^2}{11}			$\quad
							$\tr{\mathbf{h}_{5,4}}{\mathbf{k}_{2,4}}{3}			$\quad
 							$\tr{\mathbf{h}_{4,10}}{\vv\cdot\mathbf{k}_{2,4}}{ 8}	$\quad
 							$\tr{\mathbf{h}_{4,6}}{\mathbf{k}_{3,6}}{ 5}			$&\\ \hline
\text{Degree 8}\quad		&	$\mathbf{h}_{8,2}										$\quad
							$\tr{\mathbf{h}_{3,2}\cdot\mathbf{h}_{3,6}}{\vv^2}{7}	$\quad
							$\tr{\mathbf{h}_{7,2}}{\vv}{2}						$\quad
							$\tr{\mathbf{h}_{3,2}^2}{\mathbf{k}_{2,4}}{3}			$&\\ \hline
						&	$\tr{\mathbf{h}_{6,61}}{\mathbf{k}_{2,4}}{4} 			$\quad
							$\tr{\mathbf{h}_{6,62}}{\vv^2}{6}						$\quad
							$\tr{\mathbf{h}_{4,10}}{\mathbf{k}_{2,4}^2}{8} 		$&\\ \hline
\text{Degree 9}\quad		&	$\tr{\mathbf{h}_{8,2}}{\vv}{2}						$\quad
							$\tr{\mathbf{h}_{3,2}^2}{\mathbf{k}_{3,6}}{4}			$\quad
							$\tr{\mathbf{h}_{3,2}\cdot\mathbf{h}_{5,2}}{\vv}{3}	$&\\ \hline
\text{Degree 10}\quad	&	$\mathbf{h}_{10,2}									$\quad
							$\tr{\mathbf{h}_{5,2}\cdot\mathbf{h}_{3,6}}{\vv^2}{7} 	$&\\ \hline
\text{Degree 11}\quad	&	$\tr{\mathbf{h}_{5,2}^2}{\vv}{3}						$&\\ \hline
\text{Degree 12}\quad	& 	$\mathbf{h}_{12,2}$									& \\ \hline
\multicolumn{3}{|c|}{Order $4$ : $2$ covariants from $\Sn{4}$, $5$ covariants from $\Sn{6}$ and $31$ joint covariants.} \\ \hline
\text{Degree 1}\quad		&	$\vv 													$&\\ \hline
\text{Degree 2}\quad		&	$\mathbf{k}_{2,4}										$\quad
							$\mathbf{h}_{2,4}										$\quad	
							$\tr{\ff}{\vv}{3}										$&\\ \hline
\text{Degree 3}\quad		&	$\tr{\mathbf{h}_{2,4}}{\vv}{2}						$\quad
							$\tr{\ff}{\vv^2}{5}									$\quad
							$\tr{\mathbf{h}_{2,8}}{\vv}{4}						$\quad
							$\tr{\ff}{\mathbf{k}_{2,4}}{3}						$&\\ \hline
\text{Degree 4}\quad		&	$\mathbf{h}_{4,4}										$\quad
							$\tr{\mathbf{h}_{3,8}}{\vv}{4}						$\quad
							$\tr{\mathbf{h}_{3,2}}{\vv}{1}						$\quad
							$\tr{\mathbf{h}_{2,8}}{\mathbf{k}_{2,4}}{4}			$&\\ \hline
						&	$\tr{\mathbf{h}_{2,4}}{\mathbf{k}_{2,4}}{2}			$\quad
							$\tr{\mathbf{h}_{2,8}}{\vv^2}{6} 						$&\\ \hline
						&	$\tr{\ff}{\mathbf{k}_{3,6}}{4}						$\quad
							$\tr{\ff}{\vv\cdot\mathbf{k}_{2,4}}{5}				$\quad
							$\tr{\mathbf{h}_{3,6}}{\vv}{3}						$&\\ \hline
\text{Degree 5}\quad		&	$\mathbf{h}_{5,4}										$\quad
							$\tr{\mathbf{h}_{2,8}}{\vv\cdot\mathbf{k}_{2,4}}{6}	$\quad
							$\tr{\mathbf{h}_{3,12}}{\vv^2}{8}						$\quad
							$\tr{\mathbf{h}_{4,6}}{\vv}{3}						$\quad
							$\tr{\mathbf{h}_{3,2}}{\mathbf{k}_{2,4}}{1} 			$&\\ \hline
						&	$\tr{\mathbf{h}_{3,6}}{\mathbf{k}_{2,4}}{3}			$\quad
							$\tr{\mathbf{h}_{4,4}}{\vv}{2}						$\quad
							$\tr{\mathbf{h}_{3,8}}{\mathbf{k}_{2,4}}{4}			$\quad
							$\tr{\mathbf{h}_{2,8}}{\mathbf{k}_{3,6}}{5}			$\quad
							$\tr{\ff}{\mathbf{k}_{2,4}^2}{5}						$\quad
							$\tr{\mathbf{h}_{3,6}}{\vv^2}{5}						$&\\ \hline
\text{Degree 6}\quad		&	$\tr{\mathbf{h}_{5,2}}{\vv}{1}						$\quad
							$\tr{\mathbf{h}_{3,12}}{\vv\cdot\mathbf{k}_{2,4}}{8}	$\quad
							$\tr{\mathbf{h}_{5,8}}{\vv}{4}						$\quad
							$\tr{\mathbf{h}_{5,4}}{\vv}{2}						$\quad
							$\tr{\mathbf{h}_{4,6}}{\mathbf{k}_{2,4}}{3}			$\quad
							$\tr{\mathbf{h}_{3,6}}{\mathbf{k}_{3,6}}{4}			$&\\ \hline
\text{Degree 7}\quad		&	$\mathbf{h}_{7,4}										$\quad
							$\tr{\mathbf{h}_{5,8}}{\mathbf{k}_{2,4}}{4}			$&\\ \hline
\text{Degree 8}\quad		&	$\tr{\mathbf{h}_{7,4}}{\vv}{2}						$&\\ \hline	
\text{Degree 9}\quad		&	$\mathbf{h}_{9,4}$										& \\ \hline
\multicolumn{3}{|c|}{Order $6$ : $1$ covariant form $\Sn{4}$, $5$ covariants from $\Sn{6}$ and $11$ joint covariants.} \\ \hline
\text{Degree 1}\quad		&	$\ff 													$&\\ \hline
\text{Degree 2}\quad		&	$\tr{\ff}{\vv}{2} 									$&\\ \hline
\text{Degree 3}\quad		&	$\mathbf{h}_{3,6}										$\quad
							$\mathbf{k}_{3,6}										$\quad
							$\tr{\ff}{\mathbf{k}_{2,4}}{2}						$\quad
							$\tr{\mathbf{h}_{2,4}}{\vv}{1}						$\quad
							$\tr{\mathbf{h}_{2,8}}{\vv}{3}						$&\\ \hline
\text{Degree 4}\quad		&	$\mathbf{h}_{4,6}										$\quad
							$\tr{\mathbf{h}_{2,8}}{\mathbf{k}_{2,4}}{3}			$\quad
							$\tr{\mathbf{h}_{3,8}}{\vv}{3}						$\quad
							$\tr{\mathbf{h}_{2,8}}{\vv^2}{5}						$\quad
							$\tr{\mathbf{h}_{2,4}}{\mathbf{k}_{2,4}}{1}			$&\\ \hline
\text{Degree 5}\quad		&	$\tr{\mathbf{h}_{4,10}}{\vv}{4}						$\quad
							$\tr{\mathbf{h}_{4,4}}{\vv}{1}						$\quad
							$\tr{\mathbf{h}_{4,6}}{\vv}{2}						$&\\ \hline
\text{Degree 6}\quad		&	$\mathbf{h}_{6,61}									$\quad
							$\mathbf{h}_{6,62}$ & \\ \hline
\multicolumn{3}{|c|}{Order $8$ : $3$ covariants form $\Sn{6}$, $5$ joint covariants.} \\ \hline
\text{Degree 2}\quad		&	$\mathbf{h}_{2,8}						$\quad
								$\tr{\ff}{\vv}{1}  					$&\\ \hline
\text{Degree 3}\quad		&	$\mathbf{h}_{3,8} 					$\quad
								$\tr{\ff}{\mathbf{k}_{2,4}}{1}		$\quad
								$\tr{\mathbf{h}_{2,8}}{\vv}{2}		$&\\ \hline
\text{Degree 4}\quad		&	$\tr{\mathbf{h}_{3,8}}{\vv}{2}		$\quad
								$\tr{\mathbf{h}_{3,12}}{\vv}{4}		$&\\ \hline
\text{Degree 5}\quad		&	$\mathbf{h}_{5,8}$						& \\ \hline
\multicolumn{3}{|c|}{Order $10$ : $1$ covariant from $\Sn{6}$ and $1$ joint covariant.} \\ \hline
\text{Degree 3}\quad		&	$\tr{\mathbf{h}_{2,8}}{\mathbf{k}_{1,4}}{1}$ &\\ \hline
\text{Degree 4}\quad		&	$\mathbf{h}_{4,10}$						& \\ \hline
\multicolumn{3}{|c|}{Order $12$ : $1$ covariant from $\Sn{6}$.} \\ \hline
\text{Degree 3}\quad		&	$\mathbf{h}_{3,12}$							& \\ \hline
\end{supertabular}
\end{center}


\subsection{Covariant bases of $\Sn{6}\oplus \Sn{4}\oplus\Sn{2}$}
\label{subsec:Cov642}

Now we have a covariant bases of $\cov{\Sn{6}\oplus\Sn{4}}$ from theorem~\ref{thm:Cov_64}, we can use theorem~\ref{thm:CovJoints} with $V=\Sn{6}\oplus\Sn{4}$. We thus have:

\begin{thm}\label{thm:Cov_642}
The covariant algebra $\cov{\Sn{6}\oplus\Sn{4}\oplus\Sn{2}}$ is generated by a minimal bases of $494$ elements, summarized in table~\ref{table:Cov642}.
\end{thm}

\begin{table}[H]
\begin{equation*}
\setlength{\arraycolsep}{6pt}
\begin{array}{c||ccccccc|c|c|c}
 d/o &   0 &   2 &   4 &   6 &   8 &  10 &  12 &  \# & Cum \\\hline\hline
   1 &   - &   1 &   1 &   1 &   - &   - &   - &   3 &   3 \\
   2 &   3 &   2 &   5 &   2 &   2 &   - &   - &  14 &  17 \\
   3 &   4 &   10 &  9 &   8 &   4 &   1 &   1 &  37 &  54 \\
   4 &   12 &  19 & 20 &  10 &   3 &   2 &   - &  66 & 120 \\
   5 &   15 &  38 & 24 &   6 &   3 &   - &   - &  86 & 206 \\
   6 &   37 &  46 & 12 &   5 &   - &   - &   - & 100 & 306 \\
   7 &   42 &  31 &  7 &   - &   - &   - &   - &  80 & 386 \\
   8 &   38 &  15 &  1 &   - &   - &   - &   - &  54 & 440 \\
   9 &   22 &  4 &   1 &   - &   - &   - &   - &  27 & 467 \\
  10 &   9 &   3 &   - &   - &   - &   - &   - &  12 & 479 \\
  11 &   6 &   1 &   - &   - &   - &   - &   - &   7 & 486 \\
  12 &   3 &   1 &   - &   - &   - &   - &   - &   4 & 490 \\
  13 &   2 &   - &   - &   - &   - &   - &   - &   2 & 492 \\
  14 &   1 &   - &   - &   - &   - &   - &   - &   1 & 493 \\
  15 &   1 &   - &   - &   - &   - &   - &   - &   1 & 494 \\\hline
 Tot & 195 & 171 &  80 &  32 &  12 &   3 &   1 &     & 494  \\
\end{array}
\end{equation*}\caption{Covariant bases of $\Sn{6}\oplus \Sn{4}\oplus\Sn{2}$}\label{table:Cov642}
\end{table}

We give now transvectants expression of this minimal covariant bases. 

\begin{center}
\renewcommand{\arraystretch}{1.2}
\tablefirsthead{%
\multicolumn{3}{c}{}\\
}
\tablehead{%
\multicolumn{3}{r}{\small\sl continued from previous page} \\ \hline
}
\tabletail{%
\multicolumn{3}{r}{\small\sl continued on next page}
\\
}
\tablelasttail{\hline}
\begin{supertabular}{|l|ll|}
\hline
\multicolumn{3}{|c|}{$195$ invariants: $5$ from $\Sn{6}$, $2$ from $\Sn{4}$, $1$ from $\Sn{2}$, $21$ joint invariants of $\Sn{6}\oplus \Sn{2}$ given in~\ref{subsec:Cov62},}\\
\multicolumn{3}{|c|}{$53$ joint invariants of $\Sn{6}\oplus\Sn{4}$ given in~\ref{subsec:Cov64}. There is left $113$ invariants.} \\ \hline
\text{Degree 3}\quad		&	$\tr{\vv}{\uu^2}{4}											$\quad
							$\tr{\tr{\ff}{\vv}{4}}{\uu}{2}								$&\\ \hline
\text{Degree 4}\quad		&	$\tr{\mathbf{k}_{2,4}}{\uu^2}{4}								$\quad
							$\tr{\tr{\ff}{\vv}{3}}{\uu^2}{4}								$\quad
							$\tr{\tr{\ff}{\vv^2}{6}}{\uu}{2}								$\quad
							$\tr{\tr{\ff}{\mathbf{k}_{2,4}}{4}}{\uu}{2}					$\quad
							$\tr{\tr{\mathbf{h}_{2,4}}{\vv}{3}}{\uu}{2}					$&\\ \hline
\text{Degree 5}\quad		&	$\tr{\tr{\ff}{\vv}{2}}{\uu^3}{6}								$\quad
							$\tr{\tr{\ff}{\mathbf{k}_{2,4}}{3}}{\uu^2}{4}					$\quad
							$\tr{\tr{\ff}{\vv^2}{5}}{\uu^2}{4}							$\quad
							$\tr{\tr{\ff}{\mathbf{k}_{3,6}}{5}}{\uu}{2}					$&\\ \hline
						&	$\tr{\tr{\ff}{\vv\cdot\mathbf{k}_{2,4}}{6}}{\uu}{2}			$\quad
							$\tr{\tr{\mathbf{h}_{2,8}}{\vv}{4}}{\uu^2}{4}					$\quad
							$\tr{\tr{\mathbf{h}_{2,4}}{\vv}{2}}{\uu^2}{4}					$\quad
							$\tr{\tr{\mathbf{h}_{2,8}}{\vv^2}{7}}{\uu}{2}					$&\\ \hline
						&	$\tr{\tr{\mathbf{h}_{2,4}}{\mathbf{k}_{2,4}}{3}}{\uu}{2}		$\quad
							$\tr{\tr{\mathbf{h}_{3,6}}{\vv}{4}}{\uu}{2}					$&\\ \hline
						&	$\tr{\tr{\mathbf{h}_{3,2}}{\vv}{2}}{\uu}{2}					$&\\ \hline
\text{Degree 6}\quad		&	$\tr{\mathbf{k}_{3,6}}{\uu^3}{6}								$\quad
							$\tr{\tr{\ff}{\vv}{1}}{\uu^4}{8}								$\quad
							$\tr{\tr{\ff}{\mathbf{k}_{2,4}}{2}}{\uu^3}{6}					$\quad
							$\tr{\tr{\ff}{\mathbf{k}_{3,6}}{4}}{\uu^2}{4}					$&\\ \hline
						&	$\tr{\tr{\ff}{\vv\cdot\mathbf{k}_{2,4}}{5}}{\uu^2}{4}			$\quad
							$\tr{\tr{\ff}{\mathbf{k}_{2,4}^2}{6}}{\uu}{2}					$\quad
							$\tr{\tr{\mathbf{h}_{2,4}}{\vv}{1}}{\uu^3}{6}					$\quad
							$\tr{\tr{\mathbf{h}_{2,8}}{\vv}{3}}{\uu^3}{6}					$&\\ \hline
						&	$\tr{\tr{\mathbf{h}_{2,8}}{\vv^2}{6}}{\uu^2}{4}				$\quad
							$\tr{\tr{\mathbf{h}_{2,8}}{\mathbf{k}_{2,4}}{4}}{\uu^2}{4}		$\quad
							$\tr{\tr{\mathbf{h}_{2,4}}{\mathbf{k}_{2,4}}{2}}{\uu^2}{4}		$\quad
							$\tr{\tr{\mathbf{h}_{2,8}}{\vv\cdot\mathbf{k}_{2,4}}{7}}{\uu}{2}	$&\\ \hline
						&	$\tr{\tr{\mathbf{h}_{2,8}}{\mathbf{k}_{3,6}}{6}}{\uu}{2}		$\quad
							$\tr{\tr{\ff^2}{\vv^3}{11}}{\uu}{2}							$\quad
							$\tr{\tr{\mathbf{h}_{2,4}}{\mathbf{k}_{3,6}}{4}}{\uu}{2}		$\quad
							$\tr{\tr{\mathbf{h}_{3,6}}{\vv}{3}}{\uu^2}{4}					$&\\ \hline
						&	$\tr{\tr{\mathbf{h}_{3,8}}{\vv}{4}}{\uu^2}{4}					$\quad
							$\tr{\tr{\mathbf{h}_{3,2}}{\vv}{1}}{\uu^2}{4}					$\quad
							$\tr{\tr{\mathbf{h}_{3,6}}{\mathbf{k}_{2,4}}{4}}{\uu}{2}		$\quad
							$\tr{\tr{\mathbf{h}_{3,6}}{\vv^2}{6}}{\uu}{2}					$&\\ \hline
						&	$\tr{\tr{\mathbf{h}_{3,2}}{\mathbf{k}_{2,4}}{2}}{\uu}{2}		$\quad
							$\tr{\tr{\mathbf{h}_{3,8}}{\vv^2}{7}}{\uu}{2}					$\quad
							$\tr{\tr{\mathbf{h}_{4,4}}{\vv}{3}}{\uu}{2}					$\quad
							$\tr{\tr{\mathbf{h}_{4,6}}{\vv}{4}}{\uu}{2}					$&\\ \hline
\text{Degree 7}\quad		&	$\tr{\tr{\ff}{\mathbf{k}_{2,4}}{1}}{\uu^4}{8}					$\quad
							$\tr{\tr{\ff}{\mathbf{k}_{2,4}^2}{5}}{\uu^2}{4}				$\quad
							$\tr{\tr{\mathbf{h}_{2,8}}{\vv}{2}}{\uu^4}{8}					$\quad
							$\tr{\tr{\mathbf{h}_{2,8}}{\vv^2}{5}}{\uu^3}{6}				$&\\ \hline
						&	$\tr{\tr{\mathbf{h}_{2,8}}{\mathbf{k}_{2,4}}{3}}{\uu^3}{6}		$\quad
							$\tr{\tr{\mathbf{h}_{2,4}}{\mathbf{k}_{2,4}}{1}}{\uu^3}{6}		$\quad
							$\tr{\tr{\mathbf{h}_{2,8}}{\vv\cdot\mathbf{k}_{2,4}}{6}}{\uu^2}{4}	$\quad
							$\tr{\tr{\mathbf{h}_{2,8}}{\mathbf{k}_{3,6}}{5}}{\uu^2}{4}		$&\\ \hline
						&	$\tr{\tr{\mathbf{h}_{2,8}}{\mathbf{k}_{2,4}^2}{7}}{\uu}{2}		$\quad
							$\tr{\tr{\ff^2}{\vv^2\cdot\mathbf{k}_{2,4}}{11}}{\uu}{2}		$\quad
							$\tr{\tr{\mathbf{h}_{2,8}}{\vv\cdot\mathbf{k}_{3,6}}{8}}{\uu}{2}	$\quad
							$\tr{\tr{\mathbf{h}_{3,8}}{\vv}{3}}{\uu^3}{6}					$&\\ \hline
						&	$\tr{\tr{\mathbf{h}_{3,2}}{\mathbf{k}_{2,4}}{1}}{\uu^2}{4}		$\quad
							$\tr{\tr{\mathbf{h}_{3,8}}{\mathbf{k}_{2,4}}{4}}{\uu^2}{4}		$\quad
							$\tr{\tr{\mathbf{h}_{3,6}}{\vv^2}{5}}{\uu^2}{4}				$\quad
							$\tr{\tr{\mathbf{h}_{3,12}}{\vv^2}{8}}{\uu^2}{4}				$&\\ \hline
						&	$\tr{\tr{\mathbf{h}_{3,6}}{\mathbf{k}_{2,4}}{3}}{\uu^2}{4}		$\quad
							$\tr{\tr{\mathbf{h}_{3,6}}{\vv\cdot\mathbf{k}_{2,4}}{6}}{\uu}{2}	$\quad
							$\tr{\tr{\mathbf{h}_{3,12}}{\vv^3}{11}}{\uu}{2}				$\quad
							$\tr{\tr{\mathbf{h}_{3,8}}{\vv\cdot\mathbf{k}_{2,4}}{7}}{\uu}{2}	$&\\ \hline
						&	$\tr{\tr{\mathbf{h}_{3,8}}{\mathbf{k}_{3,6}}{6}}{\uu}{2}		$\quad
							$\tr{\tr{\mathbf{h}_{4,6}}{\vv}{3}}{\uu^2}{4}					$\quad
							$\tr{\tr{\mathbf{h}_{4,4}}{\vv}{2}}{\uu^2}{4}					$\quad
							$\tr{\tr{\mathbf{h}_{4,6}}{\mathbf{k}_{2,4}}{4}}{\uu}{2}		$&\\ \hline
						&	$\tr{\tr{\mathbf{h}_{4,6}}{\vv^2}{6}}{\uu}{2}					$\quad
							$\tr{\tr{\mathbf{h}_{4,4}}{\mathbf{k}_{2,4}}{3}}{\uu}{2}		$\quad
							$\tr{\tr{\mathbf{h}_{4,10}}{\vv^2}{8}}{\uu}{2}					$\quad
							$\tr{\tr{\mathbf{h}_{3,2}\cdot\ff}{\vv^2}{7}}{\uu}{2}			$&\\ \hline
						&	$\tr{\tr{\mathbf{h}_{5,4}}{\vv}{3}}{\uu}{2}					$\quad
							$\tr{\tr{\mathbf{h}_{5,2}}{\vv}{2}}{\uu}{2}$					& \\ \hline
\text{Degree 8}\quad		&	$\tr{\tr{\mathbf{h}_{2,8}}{\vv}{1}}{\uu^5}{10}					$\quad
							$\tr{\tr{\mathbf{h}_{2,8}}{\mathbf{k}_{2,4}\cdot\mathbf{k}_{3,6}}{8}}{\uu}{2}	$\quad
							$\tr{\tr{\ff^2}{\vv\cdot\mathbf{k}_{2,4}^2}{11}}{\uu}{2}		$\quad
							$\tr{\tr{\mathbf{h}_{3,12}}{\vv}{4}}{\uu^4}{8}					$&\\ \hline
						&	$\tr{\tr{\mathbf{h}_{3,8}}{\vv}{2}}{\uu^4}{8}					$\quad
							$\tr{\tr{\mathbf{h}_{3,12}}{\vv\cdot\mathbf{k}_{2,4}}{8}}{\uu^2}{4}			$\quad
							$\tr{\tr{\mathbf{h}_{3,6}}{\mathbf{k}_{3,6}}{4}}{\uu^2}{4}		$\quad
							$\tr{\tr{\mathbf{h}_{3,8}}{\vv\cdot\mathbf{k}_{3,6}}{8}}{\uu}{2}				$&\\ \hline
						&	$\tr{\tr{\mathbf{h}_{3,12}}{\vv^2\cdot\mathbf{k}_{2,4}}{11}}{\uu}{2}			$\quad
							$\tr{\tr{\mathbf{h}_{4,10}}{\vv}{4}}{\uu^3}{6}					$\quad
							$\tr{\tr{\mathbf{h}_{4,6}}{\vv}{2}}{\uu^3}{6}					$\quad
							$\tr{\tr{\mathbf{h}_{4,4}}{\vv}{1}}{\uu^3}{6}					$&\\ \hline
						&	$\tr{\tr{\mathbf{h}_{4,6}}{\mathbf{k}_{2,4}}{3}}{\uu^2}{4}		$\quad
							$\tr{\tr{\mathbf{h}_{4,6}}{\vv\cdot\mathbf{k}_{2,4}}{6}}{\uu}{2}	$\quad
							$\tr{\tr{\mathbf{h}_{4,10}}{\vv^3}{10}}{\uu}{2}				$\quad
							$\tr{\tr{\ff\cdot\mathbf{h}_{3,6}}{\vv^3}{11}}{\uu}{2}			$&\\ \hline
						&	$\tr{\tr{\mathbf{h}_{4,10}}{\vv\cdot\mathbf{k}_{2,4}}{8}}{\uu}{2}				$\quad
							$\tr{\tr{\mathbf{h}_{4,6}}{\mathbf{k}_{3,6}}{5}}{\uu}{2}						$\quad
							$\tr{\tr{\mathbf{h}_{2,4}^2}{\mathbf{k}_{3,6}}{6}}{\uu}{2}						$\quad
							$\tr{\tr{\mathbf{h}_{5,8}}{\vv}{4}}{\uu^2}{4}					$&\\ \hline
						&	$\tr{\tr{\mathbf{h}_{5,2}}{\vv}{1}}{\uu^2}{4}					$\quad
							$\tr{\tr{\mathbf{h}_{5,4}}{\vv}{2}}{\uu^2}{4}					$\quad
							$\tr{\tr{\mathbf{h}_{5,2}}{\mathbf{k}_{2,4}}{2}}{\uu}{2}		$\quad
							$\tr{\tr{\mathbf{h}_{5,8}}{\vv^2}{7}}{\uu}{2}					$&\\ \hline
						&	$\tr{\tr{\mathbf{h}_{5,4}}{\mathbf{k}_{2,4}}{3}}{\uu}{2}		$\quad
							$\tr{\tr{\mathbf{h}_{6,62}}{\vv}{4}}{\uu}{2}					$\quad
							$\tr{\tr{\mathbf{h}_{3,2}^2}{\vv}{3}}{\uu}{2}					$\quad
							$\tr{\tr{\mathbf{h}_{6,61}}{\vv}{4}}{\uu}{2}					$&\\ \hline
\text{Degree 9}\quad		&	$\tr{\tr{\mathbf{h}_{4,10}}{\mathbf{k}_{2,4}^2}{8}}{\uu}{2}	$\quad
							$\tr{\tr{\mathbf{h}_{5,8}}{\mathbf{k}_{2,4}}{4}}{\uu^2}{4}		$\quad
							$\tr{\tr{\mathbf{h}_{3,2}\cdot\mathbf{h}_{3,6}}{\vv^2}{7}}{\uu}{2}				$\quad
							$\tr{\tr{\mathbf{h}_{6,61}}{\mathbf{k}_{2,4}}{4}}{\uu}{2}		$&\\ \hline
						&	$\tr{\tr{\mathbf{h}_{3,2}^2}{\mathbf{k}_{2,4}}{3}}{\uu}{2}		$\quad
							$\tr{\tr{\mathbf{h}_{6,62}}{\vv^2}{6}}{\uu}{2}					$\quad
							$\tr{\tr{\mathbf{h}_{7,2}}{\vv}{2}}{\uu}{2}					$&\\ \hline
\text{Degree 10}\quad	&	$\tr{\tr{\mathbf{h}_{3,2}^2}{\mathbf{k}_{3,6}}{4}}{\uu}{2}		$\quad
							$\tr{\tr{\mathbf{h}_{7,4}}{\vv}{2}}{\uu^2}{4}					$\quad
							$\tr{\tr{\mathbf{h}_{8,2}}{\vv}{2}}{\uu}{2}					$\quad
							$\tr{\tr{\mathbf{h}_{3,2}\cdot\mathbf{h}_{5,2}}{\vv}{3}}{\uu}{2}	$&\\ \hline
\text{Degree 11}\quad	&	$\tr{\tr{\mathbf{h}_{5,2}\cdot\mathbf{h}_{3,6}}{\vv^2}{7}}{\uu}{2}	$&\\ \hline
\text{Degree 12}\quad	&	$\tr{\tr{\mathbf{h}_{5,2}^2}{\vv}{3}}{\uu}{2}$						& \\ \hline
\multicolumn{3}{|c|}{$171$ covariants of order $2$:$6$ from $\Sn{6}$, $1$ from $\Sn{2}$, $23$ joint covariants of $\Sn{6}\oplus \Sn{2}$ given in~\ref{subsec:Cov62},}\\
\multicolumn{3}{|c|}{$62$ joint covariants of $\Sn{6}\oplus\Sn{4}$ given in~\ref{subsec:Cov64}. There is left $79$ covariants given below:} \\ \hline
\text{Degree 2}\quad		&	$\tr{\vv}{\uu}{2}										$&\\ \hline
\text{Degree 3}\quad		&	$\tr{\vv}{\uu^2}{3}										$\quad
							$\tr{\mathbf{k}_{2,4}}{\uu}{2}							$\quad
							$\tr{\tr{\ff}{\vv}{4}}{\uu}{1}							$\quad
							$\tr{\tr{\ff}{\vv}{3}}{\uu}{2}							$&\\ \hline
\text{Degree 4}\quad		&	$\tr{\tr{1}{\mathbf{k}_{2,4}}{0}}{\uu^2}{3}				$\quad
							$\tr{\tr{\ff}{\vv}{2}}{\uu^2}{4}							$\quad
							$\tr{\tr{\ff}{\vv}{3}}{\uu^2}{3}							$\quad
							$\tr{\tr{\ff}{\mathbf{k}_{2,4}}{3}}{\uu}{2}				$&\\ \hline
						&	$\tr{\tr{\ff}{\vv^2}{5}}{\uu}{2}							$\quad
							$\tr{\tr{\ff}{\mathbf{k}_{2,4}}{4}}{\uu}{1}				$\quad
							$\tr{\tr{\ff}{\vv^2}{6}}{\uu}{1}							$\quad
							$\tr{\tr{\mathbf{h}_{2,4}}{\vv}{2}}{\uu}{2}				$&\\ \hline
						&	$\tr{\tr{\mathbf{h}_{2,8}}{\vv}{4}}{\uu}{2}				$\quad
							$\tr{\tr{\mathbf{h}_{2,4}}{\vv}{3}}{\uu}{1}				$&\\ \hline
\text{Degree 5}\quad		&	$\tr{\mathbf{k}_{3,6}}{\uu^2}{4}							$\quad
							$\tr{\tr{\ff}{\vv}{2}}{\uu^3}{5}							$\quad
							$\tr{\tr{\ff}{\vv}{1}}{\uu^3}{6}							$\quad
							$\tr{\tr{\ff}{\mathbf{k}_{2,4}}{2}}{\uu^2}{4}				$&\\ \hline
						&	$\tr{\tr{\ff}{\mathbf{k}_{2,4}}{3}}{\uu^2}{3}				$\quad
							$\tr{\tr{\ff}{\mathbf{k}_{3,6}}{5}}{\uu}{1}				$\quad
							$\tr{\tr{\ff}{\vv\cdot\mathbf{k}_{2,4}}{5}}{\uu}{2}			$\quad
							$\tr{\tr{\ff}{\vv\cdot\mathbf{k}_{2,4}}{6}}{\uu}{1}			$&\\ \hline
						&	$\tr{\tr{\ff}{\mathbf{k}_{3,6}}{4}}{\uu}{2}				$\quad
							$\tr{\tr{\mathbf{h}_{2,8}}{\vv}{4}}{\uu^2}{3}				$\quad
							$\tr{\tr{\mathbf{h}_{2,4}}{\vv}{1}}{\uu^2}{4}				$\quad
							$\tr{\tr{\mathbf{h}_{2,8}}{\vv}{3}}{\uu^2}{4}				$&\\ \hline 	
						&	$\tr{\tr{\mathbf{h}_{2,4}}{\vv}{2}}{\uu^2}{3}				$\quad
							$\tr{\tr{\mathbf{h}_{2,8}}{\vv^2}{7}}{\uu}{1}				$\quad
							$\tr{\tr{\mathbf{h}_{2,4}}{\mathbf{k}_{2,4}}{3}}{\uu}{1}	$\quad
							$\tr{\tr{\mathbf{h}_{2,8}}{\mathbf{k}_{2,4}}{4}}{\uu}{2}	$&\\ \hline
						&	$\tr{\tr{\mathbf{h}_{2,8}}{\vv^2}{6}}{\uu}{2}				$\quad
							$\tr{\tr{\mathbf{h}_{2,4}}{\mathbf{k}_{2,4}}{2}}{\uu}{2}	$\quad
							$\tr{\tr{\mathbf{h}_{3,2}}{\vv}{1}}{\uu}{2}				$\quad
							$\tr{\tr{\mathbf{h}_{3,6}}{\vv}{4}}{\uu}{1}				$&\\ \hline
						&	$\tr{\tr{\mathbf{h}_{3,2}}{\vv}{2}}{\uu}{1}				$\quad
							$\tr{\tr{\mathbf{h}_{3,6}}{\vv}{3}}{\uu}{2}				$\quad
							$\tr{\tr{\mathbf{h}_{3,8}}{\vv}{4}}{\uu}{2}				$&\\ \hline
\text{Degree 6}\quad		&	$\tr{\tr{\ff}{\mathbf{k}_{2,4}}{1}}{\uu^3}{6}				$\quad
							$\tr{\tr{\ff}{\mathbf{k}_{2,4}^2}{5}}{\uu}{2}				$\quad
							$\tr{\tr{\ff}{\mathbf{k}_{2,4}^2}{6}}{\uu}{1}				$\quad
							$\tr{\tr{\mathbf{h}_{2,8}}{\vv}{2}}{\uu^3}{6}				$&\\ \hline
						&	$\tr{\tr{\mathbf{h}_{2,8}}{\vv^2}{5}}{\uu^2}{4}			$\quad
							$\tr{\tr{\mathbf{h}_{2,8}}{\mathbf{k}_{2,4}}{4}}{\uu^2}{3}	$\quad
							$\tr{\tr{\mathbf{h}_{2,8}}{\mathbf{k}_{2,4}}{3}}{\uu^2}{4}	$\quad
							$\tr{\tr{\mathbf{h}_{2,4}}{\mathbf{k}_{2,4}}{1}}{\uu^2}{4}	$&\\ \hline
						&	$\tr{\tr{\mathbf{h}_{2,8}}{\vv\cdot\mathbf{k}_{2,4}}{7}}{\uu}{1}	$\quad
							$\tr{\tr{\mathbf{h}_{2,8}}{\mathbf{k}_{3,6}}{5}}{\uu}{2}		$\quad
							$\tr{\tr{\mathbf{h}_{2,8}}{\mathbf{k}_{3,6}}{6}}{\uu}{1}		$\quad
							$\tr{\tr{\mathbf{h}_{2,8}}{\vv\cdot\mathbf{k}_{2,4}}{6}}{\uu}{2}	$&\\ \hline
						&	$\tr{\tr{\mathbf{h}_{3,8}}{\vv}{3}}{\uu^2}{4}				$\quad
							$\tr{\tr{\mathbf{h}_{3,6}}{\vv}{3}}{\uu^2}{3}				$\quad
							$\tr{\tr{\mathbf{h}_{3,2}}{\mathbf{k}_{2,4}}{1}}{\uu}{2}	$\quad
							$\tr{\tr{\mathbf{h}_{3,8}}{\mathbf{k}_{2,4}}{4}}{\uu}{2}	$&\\ \hline
						&	$\tr{\tr{\mathbf{h}_{3,6}}{\vv^2}{5}}{\uu}{2}				$\quad
							$\tr{\tr{\mathbf{h}_{3,6}}{\vv^2}{6}}{\uu}{1}				$\quad
							$\tr{\tr{\mathbf{h}_{3,6}}{\mathbf{k}_{2,4}}{3}}{\uu}{2}	$\quad
							$\tr{\tr{\mathbf{h}_{3,2}}{\mathbf{k}_{2,4}}{2}}{\uu}{1}	$&\\ \hline
						&	$\tr{\tr{\mathbf{h}_{3,12}}{\vv^2}{8}}{\uu}{2}				$\quad
							$\tr{\tr{\mathbf{h}_{3,8}}{\vv^2}{7}}{\uu}{1}				$\quad
							$\tr{\tr{\mathbf{h}_{4,6}}{\vv}{3}}{\uu}{2}				$\quad
							$\tr{\tr{\mathbf{h}_{4,6}}{\vv}{4}}{\uu}{1}				$&\\ \hline
						&	$\tr{\tr{\mathbf{h}_{4,4}}{\vv}{2}}{\uu}{2}				$&\\ \hline
\text{Degree 7}\quad		&	$\tr{\tr{\mathbf{h}_{2,8}}{\vv}{1}}{\uu^4}{8}				$\quad
							$\tr{\tr{\mathbf{h}_{3,8}}{\vv}{2}}{\uu^3}{6}				$\quad
							$\tr{\tr{\mathbf{h}_{3,12}}{\vv}{4}}{\uu^3}{6}				$\quad
							$\tr{\tr{\mathbf{h}_{3,12}}{\vv\cdot\mathbf{k}_{2,4}}{8}}{\uu}{2}	$&\\ \hline
						&	$\tr{\tr{\mathbf{h}_{3,6}}{\mathbf{k}_{3,6}}{4}}{\uu}{2}	$\quad
							$\tr{\tr{\mathbf{h}_{4,6}}{\vv}{2}}{\uu^2}{4}				$\quad
							$\tr{\tr{\mathbf{h}_{4,4}}{\vv}{1}}{\uu^2}{4}				$\quad
							$\tr{\tr{\mathbf{h}_{4,10}}{\vv}{4}}{\uu^2}{4}				$&\\ \hline
						&	$\tr{\tr{\mathbf{h}_{4,6}}{\vv^2}{6}}{\uu}{1}				$\quad
							$\tr{\tr{\mathbf{h}_{4,6}}{\mathbf{k}_{2,4}}{3}}{\uu}{2}	$\quad
							$\tr{\tr{\mathbf{h}_{5,4}}{\vv}{2}}{\uu}{2}				$\quad
							$\tr{\tr{\mathbf{h}_{5,8}}{\vv}{4}}{\uu}{2}				$&\\ \hline
						&	$\tr{\tr{\mathbf{h}_{5,4}}{\vv}{3}}{\uu}{1}				$\quad
							$\tr{\tr{\mathbf{h}_{5,2}}{\vv}{1}}{\uu}{2}				$&\\ \hline
\text{Degree 8}\quad		&	$\tr{\tr{\mathbf{h}_{5,8}}{\mathbf{k}_{2,4}}{4}}{\uu}{2}	$&\\ \hline
\text{Degree 9}\quad		&	$\tr{\tr{\mathbf{h}_{7,4}}{\vv}{2}}{\uu}{2}$				& \\ \hline
\multicolumn{3}{|c|}{$80$ covariants of order $4$ : $5$ from $\Sn{6}$, $2$ from $\Sn{4}$, $15$ joint covariants of $\Sn{6}\oplus \Sn{2}$ given in~\ref{subsec:Cov62},}\\
\multicolumn{3}{|c|}{$31$ joint covariants of $\Sn{6}\oplus\Sn{4}$ given in~\ref{subsec:Cov64}. There is left $27$ covariants given below:} \\ \hline
\text{Degree 2}\quad		&	$\tr{\vv}{\uu}{1}							$&\\ \hline
\text{Degree 3}\quad		&	$\tr{\mathbf{k}_{2,4}}{\uu}{1}				$\quad
								$\tr{\tr{\ff}{\vv}{2}}{\uu}{2}				$\quad
								$\tr{\tr{\ff}{\vv}{3}}{\uu}{1}				$&\\ \hline
\text{Degree 4}\quad		&	$\tr{\mathbf{k}_{3,6}}{\uu}{2}				$\quad
								$\tr{\tr{\ff}{\vv}{2}}{\uu^2}{3}				$\quad
								$\tr{\tr{\ff}{\vv}{1}}{\uu^2}{4}				$\quad
								$\tr{\tr{\ff}{\mathbf{k}_{2,4}}{3}}{\uu}{1}	$&\\ \hline
							&	$\tr{\tr{\ff}{\mathbf{k}_{2,4}}{2}}{\uu}{2}	$\quad
								$\tr{\tr{\mathbf{h}_{2,8}}{\vv}{4}}{\uu}{1}	$\quad
								$\tr{\tr{\mathbf{h}_{2,4}}{\vv}{2}}{\uu}{1}	$\quad
								$\tr{\tr{\mathbf{h}_{2,8}}{\vv}{3}}{\uu}{2}	$&\\ \hline
							&	$\tr{\tr{\mathbf{h}_{2,4}}{\vv}{1}}{\uu}{2}	$&\\ \hline
\text{Degree 5}\quad		&	$\tr{\tr{\ff}{\mathbf{k}_{2,4}}{1}}{\uu^2}{4}				$\quad
								$\tr{\tr{\mathbf{h}_{2,8}}{\vv}{2}}{\uu^2}{4}				$\quad
								$\tr{\tr{\mathbf{h}_{2,8}}{\mathbf{k}_{2,4}}{4}}{\uu}{1}		$\quad
								$\tr{\tr{\mathbf{h}_{2,8}}{\vv^2}{5}}{\uu}{2}				$&\\ \hline
							&	$\tr{\tr{\mathbf{h}_{2,8}}{\mathbf{k}_{2,4}}{3}}{\uu}{2}		$\quad
								$\tr{\tr{\mathbf{h}_{2,4}}{\mathbf{k}_{2,4}}{1}}{\uu}{2}		$\quad
								$\tr{\tr{\mathbf{h}_{3,6}}{\vv}{3}}{\uu}{1}					$\quad
								$\tr{\tr{\mathbf{h}_{3,8}}{\vv}{3}}{\uu}{2}					$&\\ \hline
\text{Degree 6}\quad		&	$\tr{\tr{\mathbf{h}_{2,8}}{\vv}{1}}{\uu^3}{6}				$\quad
								$\tr{\tr{\mathbf{h}_{3,8}}{\vv}{2}}{\uu^2}{4}				$\quad
								$\tr{\tr{\mathbf{h}_{3,12}}{\vv}{4}}{\uu^2}{4}				$\quad
								$\tr{\tr{\mathbf{h}_{4,10}}{\vv}{4}}{\uu}{2}					$&\\ \hline
							&	$\tr{\tr{\mathbf{h}_{4,4}}{\vv}{1}}{\uu}{2}					$\quad
								$\tr{\tr{\mathbf{h}_{4,6}}{\vv}{2}}{\uu}{2}					$& \\ \hline
\multicolumn{3}{|c|}{$32$ covariants of order $6$ : $5$ from $\Sn{6}$, $1$ from $\Sn{4}$, $8$ joint covariants of $\Sn{6}\oplus \Sn{2}$ given in~\ref{subsec:Cov62},}\\
\multicolumn{3}{|c|}{$11$ joint covariants of $\Sn{6}\oplus\Sn{4}$ given in~\ref{subsec:Cov64}. There is left $7$ covariants given below:} \\ \hline
\text{Degree 3}\quad		&	$\tr{\tr{\ff}{\vv}{2}}{\uu}{1}					$\quad
								$\tr{\tr{\ff}{\vv}{1}}{\uu}{2}					$&\\ \hline
\text{Degree 4}\quad		&	$\tr{\tr{\ff}{\mathbf{k}_{2,4}}{1}}{\uu}{2}				$\quad
								$\tr{\tr{\mathbf{h}_{2,8}}{\vv}{2}}{\uu}{2}		$&\\ \hline
\text{Degree 5}\quad		&	$\tr{\tr{\mathbf{h}_{2,8}}{\vv}{1}}{\uu^2}{4}	$\quad
								$\tr{\tr{\mathbf{h}_{3,12}}{\vv}{4}}{\uu}{2}		$\quad
								$\tr{\tr{\mathbf{h}_{3,8}}{\vv}{2}}{\uu}{2}		$&\\ \hline
\multicolumn{3}{|c|}{$12$ covariants of order $8$ : $3$ from $\Sn{6}$, $3$ joint covariants of $\Sn{6}\oplus \Sn{2}$ given in~\ref{subsec:Cov62},}\\
\multicolumn{3}{|c|}{$5$ joint covariants of $\Sn{6}\oplus\Sn{4}$ given in~\ref{subsec:Cov64}. There is left $1$ covariant given below:} \\ \hline
\text{degree 4}\quad		&	$\tr{\tr{\mathbf{h}_{2,8}}{\vv}{1}}{\uu}{2}$		& \\ \hline
\end{supertabular}
\end{center}

There is left $3$ covariants of order $10$ : $1$ from $\Sn{6}$, $1$ joint covariant of $\Sn{6}\oplus \Sn{2}$ given in~\ref{subsec:Cov62} and $1$ joint covariant of $\Sn{6}\oplus\Sn{4}$ given in~\ref{subsec:Cov64}. Finally there is $1$ covariant of order $12$ taken from $\Sn{6}$.


\subsection{Covariant bases of $\Sn{8}$}
\label{subsec:Cov8}

We apply here Gordan's algorithm for a simple binary form. 
\begin{enumerate}
\item As a first step $\mathrm{A}_0=\set{\ff}$ for $\ff\in \Sn{8}$. The family $\mathrm{B}_0$ only contains the covariant
\begin{equation*}
\mathbf{h}_{2,12}:=\lbrace \ff,\ff\rbrace_2\in \Sn{12}.
\end{equation*}
\item To obtain $\mathrm{A}_1$ we have to consider transvectants
\begin{equation*}
\tr{\ff^a}{\mathbf{h}_{2,12}^b}{r},
\end{equation*}
which contain no reducible molecular covariants modulo $I_4$. From \eqref{eq:Grade_Impair} we deduce that necessarily $r\leq 2$. Take now a molecule
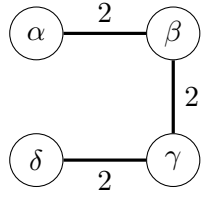
\begin{center}
\begin{tikzpicture}[scale=1.2,baseline={([yshift=-.5ex]current bounding box.center)}]
		\node[m] (P)at(0.2,1.4){$\alpha$};
		\node[m] (Q)at(1.7,1.4){$\beta$};
		\node[m] (R)at(1.7,0){$\gamma$};
		\node[m] (S)at(0.2,0){$\delta$};
		\draw[flecheno] (P)--(Q) node[midway,above] {$2$} ;
		\draw[flecheno] (S)--(R) node[midway,below] {$2$} ;
		\draw[flecheno] (Q)--(R) node[midway,right] {$2$} ;
	\end{tikzpicture}
\end{center}
Using lemma~\ref{lem:RedDeg3} with $e_0=2$ and $e_1=2$, this molecule is of grade $3$ and thus by \eqref{eq:Grade_Impair} of grade $4$.

We can deduce from all this that the family $\mathrm{A}_{1}$ is 
\begin{equation*}
\ff,\quad \mathbf{h}_{2,12},\quad \mathbf{h}_{3,18}:=\lbrace \ff,\mathbf{h}_{2,12}\rbrace_1
\end{equation*}
and family $\mathrm{B}_{1}$ only contains the covariant
\begin{equation*}
\mathbf{h}_{2,8}:=\lbrace \ff,\ff\rbrace_4\in \Sn{8}.
\end{equation*}
\item To get the system $\mathrm{A}_{2}$ we have to consider transvectants
\begin{equation*}
\tr{\ff^{a_1}\mathbf{h}_{2,12}^{a_2}\mathbf{h}_{3,18}^{a_3}}{\mathbf{h}_{2,8}^b}{r}.
\end{equation*}
The same kind of argument as above, using lemma such as lemma~\ref{lem:RedDeg3} leads to~\cite{GY2010,Gor1875}:

\begin{lem}
The family $\mathrm{A}_{2}$ is given by the seven covariants
\begin{align*}
\ff,\quad &\mathbf{h}_{2,8}=\tr{\ff}{\ff}{4},\quad \mathbf{h}_{2,12}=\tr{\ff}{\ff}{2},\quad \mathbf{h}_{3,12}:=\tr{\ff}{\mathbf{h}_{2,8}}{2},\quad \mathbf{h}_{3,14}:=\tr{\ff}{\mathbf{h}_{2,8}}{1}
\end{align*}
\begin{align*}
\mathbf{h}_{3,18}:=\tr{\ff}{\mathbf{h}_{2,12}}{1},\quad 
\mathbf{h}_{4,18}:=\tr{\mathbf{h}_{2,12}}{\mathbf{h}_{2,8}}{1}
\end{align*}
\end{lem}

Recall also that we have to consider the invariant
\begin{equation*}
\tr{\ff}{\mathbf{h}_{2,8}}{8}.
\end{equation*}

Now, the family $\mathrm{B}_{2}$ is given by one covariant bases of
\begin{equation*}
\mathbf{h}_{2,4}:=\tr{\ff}{\ff}{6} \in \Sn{4}.
\end{equation*}

As seen above in~\autoref{subsec:Cov64}, such a covariant bases is given by:
\begin{align*}
\mathbf{h}_{2,4},\quad 
\mathbf{h}_{4,4}:=\tr{\mathbf{h}_{2,4}}{\mathbf{h}_{2,4}}{2},\quad \mathbf{h}_{6,6}:=\tr{\mathbf{h}_{2,4}}{\tr{\mathbf{h}_{2,4}}{\mathbf{h}_{2,4}}{2}}{1}
\end{align*}
and two invariants
\begin{align*}
\mathbf{h}_{4,0}:=\tr{\mathbf{h}_{2,4}}{\mathbf{h}_{2,4}}{4},\quad \mathbf{h}_{6,0}:=\tr{\mathbf{h}_{2,4}}{\tr{\mathbf{h}_{2,4}}{\mathbf{h}_{2,4}}{2}}{4}.
\end{align*}
\item To get the family $\mathrm{B}_{3}$, we have to consider transvectants
\begin{equation*}
\tr{\ff^{a_1}\mathbf{h}_{2,8}^{a_2}\mathbf{h}_{2,12}^{a_3}\mathbf{h}_{3,12}^{a_4}
\mathbf{h}_{3,14}^{a_5}\mathbf{h}_{3,18}^{a_6}\mathbf{h}_{4,18}^{a_7}}{\mathbf{h}_{2,4}^{b_1}
\mathbf{h}_{4,4}^{b_2}\mathbf{h}_{6,6}^{b_3}}{r}
\end{equation*}
which is associated to the integer system
\begin{equation}\label{eq:IntegerS8}
\begin{cases}
8a_1+8a_2+12a_3+12a_4+14a_5+18a_6+18a_7&=u+r \\
4b_1+4b_2+6b_3&=v+r
\end{cases}.
\end{equation}

We also make use of the relation~\eqref{Rel:CovS4} in $\cov{\Sn{4}}$, thus we can apply theorem~\ref{thm:FamRelCompleteBis}. With computations made in Macaulay2~\cite{M2}, we finally get a covariant bases of $\Sn{8}$ given bellow.
\begin{center}
\renewcommand{\arraystretch}{1.2}
\tablefirsthead{%
\multicolumn{3}{c}{}\\
}
\tablehead{%
\multicolumn{3}{r}{\small\sl continued from previous page} \\ \hline
}
\tabletail{%
\multicolumn{3}{r}{\small\sl continued on next page}
\\
}
\tablelasttail{\hline}
\begin{supertabular}{|l|ll|}
\hline
\multicolumn{3}{|c|}{$8$ invariants} \\ \hline
\text{Degree 2}\quad		&	$\mathbf{h}_{2,0}:=\tr{\ff}{\ff}{8}								$&\\ \hline
\text{Degree 3}\quad		&	$\tr{\ff}{\mathbf{h}_{2,8}}{8}									$&\\ \hline
\text{Degree 4}\quad		&	$\tr{\mathbf{h}_{2,4}}{\mathbf{h}_{2,4}}{4}						$&\\ \hline
\text{Degree 5}\quad		&	$\tr{\ff}{\mathbf{h}_{2,4}^2}{8}									$&\\ \hline
\text{Degree 6}\quad		&	$\tr{\mathbf{h}_{4,4}}{\mathbf{h}_{2,4}}{4}						$&\\ \hline
\text{Degree 7}\quad		&	$\tr{\ff}{\mathbf{h}_{2,4}\mathbf{h}_{4,4}}{8}						$&\\ \hline
\text{Degree 8}\quad		&	$\tr{\mathbf{h}_{2,12}}{\mathbf{h}_{2,4}^3}{12}					$&\\ \hline
\text{Degree 9}\quad		&	$\tr{\mathbf{h}_{3,12}}{\mathbf{h}_{2,4}^3}{12}					$&\\ \hline
\text{Degree 10}\quad	&	$\tr{\mathbf{h}_{2,12}}{\mathbf{h}_{2,4}^2\mathbf{h}_{4,4}}{12}	$& \\ \hline 
\multicolumn{3}{|c|}{$14$ covariants of order $2$.} \\ \hline
\text{Degree 5}\quad		&	$\tr{\ff}{\mathbf{h}_{2,4}^2}{7}									$&\\ \hline
\text{Degree 6}\quad		&	$\tr{\mathbf{h}_{2,8}}{\mathbf{h}_{2,4}^2}{7}						$&\\ \hline
\text{Degree 7}\quad		&	$\tr{\ff}{\mathbf{h}_{6,6}}{6}									$\quad
							$\tr{\ff}{\mathbf{h}_{2,4}\mathbf{h}_{4,4}}{7}						$&\\ \hline
\text{Degree 8}\quad		&	$\tr{\mathbf{h}_{2,12}}{\mathbf{h}_{2,4}^3}{11}					$\quad
							$\tr{\mathbf{h}_{2,8}}{\mathbf{h}_{6,6}}{6}						$&\\ \hline
\text{Degree 9}\quad		&	$\tr{\mathbf{h}_{3,14}}{\mathbf{h}_{2,4}^3}{12}					$\quad
							$\tr{\mathbf{h}_{3,12}}{\mathbf{h}_{2,4}^3}{11}					$\quad
							$\tr{\ff}{\mathbf{h}_{4,4}^2}{7}									$&\\ \hline
\text{Degree 10}\quad	&	$\tr{\mathbf{h}_{2,12}}{\mathbf{h}_{2,4}\mathbf{h}_{6,6}}{10}		$\quad
							$\tr{\mathbf{h}_{2,12}}{\mathbf{h}_{2,4}^2\mathbf{h}_{4,4}}{11}	$&\\ \hline
\text{Degree 11}\quad	&	$\tr{\mathbf{h}_{3,18}}{\mathbf{h}_{2,4}^4}{16} 					$\quad
							$\tr{\mathbf{h}_{3,14}}{\mathbf{h}_{2,4}^2\mathbf{h}_{4,4}}{12}	$&\\ \hline
\text{Degree 12}\quad	&	$\tr{\mathbf{h}_{4,18}}{\mathbf{h}_{2,4}^4}{16}					$&\\ \hline
\multicolumn{3}{|c|}{$13$ covariants of order $4$.} \\ \hline
\text{Degree 2}\quad		&	$\mathbf{h}_{2,4}:=\tr{\ff}{\ff}{6}								$&\\ \hline
\text{Degree 3}\quad		&	$\tr{\ff}{\mathbf{h}_{2,4}}{4}									$&\\ \hline
\text{Degree 4}\quad		&	$\mathbf{h}_{4,4}:=\tr{\mathbf{h}_{2,4}}{\mathbf{h}_{2,4}}{2}		$&\\ \hline
						&	$\tr{\mathbf{h}_{2,8}}{\mathbf{h}_{2,4}}{4}						$&\\ \hline
\text{Degree 5}\quad		&	$\tr{\ff}{\mathbf{h}_{4,4}}{4}$									\quad
							$\tr{\ff}{\mathbf{h}_{2,4}^2}{6}									$&\\ \hline
\text{Degree 6}\quad		&	$\tr{\mathbf{h}_{2,12}}{\mathbf{h}_{2,4}^2}{8}$						\quad
							$\tr{\mathbf{h}_{2,8}}{\mathbf{h}_{4,4}}{4}						$&\\ \hline
\text{Degree 7}\quad		&	$\tr{\mathbf{h}_{3,12}}{\mathbf{h}_{2,4}^2}{8}$						\quad
							$\tr{\ff}{\mathbf{h}_{6,6}}{5}									$&\\ \hline
\text{Degree 8}\quad		&	$\tr{\mathbf{h}_{2,12}}{\mathbf{h}_{2,4}\mathbf{h}_{4,4}}{8}$		\quad
							$\tr{\mathbf{h}_{2,12}}{\mathbf{h}_{2,4}^3}{10}					$&\\ \hline
\text{Degree 9}\quad		&	$\tr{\mathbf{h}_{3,14}}{\mathbf{h}_{2,4}^3}{11}					$&\\ \hline
\multicolumn{3}{|c|}{$12$ covariants of order $6$.} \\ \hline
\text{Degree 3}\quad		&	$\tr{\ff}{\mathbf{h}_{2,4}}{3}									$&\\ \hline
\text{Degree 4}\quad		&	$\tr{\mathbf{h}_{2,8}}{\mathbf{h}_{2,4}}{3}						$&\\ \hline
\text{Degree 5}\quad		&	$\tr{\ff}{\mathbf{h}_{4,4}}{3}								$	\quad
							$\tr{\ff}{\mathbf{h}_{2,4}^2}{5}									$&\\ \hline
\text{Degree 6}\quad		&	$\mathbf{h}_{6,6}:=\tr{\mathbf{h}_{4,4}}{\mathbf{h}_{2,4}}{1}	$	\quad
							$\tr{\mathbf{h}_{2,12}}{\mathbf{h}_{2,4}^2}{7}				$		\quad
							$\tr{\mathbf{h}_{2,8}}{\mathbf{h}_{4,4}}{3}						$&\\ \hline
\text{Degree 7}\quad		&	$\tr{\mathbf{h}_{3,14}}{\mathbf{h}_{2,4}^2}{8}				$		\quad
							$\tr{\mathbf{h}_{3,12}}{\mathbf{h}_{2,4}^2}{7}				$		\quad
							$\tr{\ff}{\mathbf{h}_{6,6}}{4}									$&\\ \hline
\text{Degree 8}\quad		&	$\tr{\mathbf{h}_{2,12}}{\mathbf{h}_{6,6}}{6}					$	\quad
							$\tr{\mathbf{h}_{2,12}}{\mathbf{h}_{2,4}\mathbf{h}_{4,4}}{7}		$&\\ \hline
\multicolumn{3}{|c|}{$6$ covariants of order $8$.} \\ \hline
\text{Degree 1}\quad		&	$\ff																$&\\ \hline
\text{Degree 2}\quad		&	$\mathbf{h}_{2,8}													$&\\ \hline
\text{Degree 3}\quad		&	$\tr{\ff}{\mathbf{h}_{2,4}}{2}									$&\\ \hline
\text{Degree 4}\quad		&	$\tr{\mathbf{h}_{2,12}}{\mathbf{h}_{2,4}}{4}						$&\\ \hline
\text{Degree 5}\quad		&	$\tr{\ff}{\mathbf{h}_{4,4}}{2}									$&\\ \hline
\text{Degree 6}\quad		&	$\tr{\mathbf{h}_{2,12}}{\mathbf{h}_{4,4}}{4}						$&\\ \hline
\multicolumn{3}{|c|}{$7$ covariants of order $10$.} \\ \hline
\text{Degree 3}\quad		&	$\tr{\ff}{\mathbf{h}_{2,4}}{1}									$&\\ \hline
\text{Degree 4}\quad		&	$\tr{\mathbf{h}_{2,12}}{\mathbf{h}_{2,4}}{3}$						\quad
							$\tr{\mathbf{h}_{2,8}}{\mathbf{h}_{2,4}}{1}						$&\\ \hline		
\text{Degree 5}\quad		&	$\tr{\mathbf{h}_{3,14}}{\mathbf{h}_{2,4}}{4}	$					\quad
							$\tr{\mathbf{h}_{3,12}}{\mathbf{h}_{2,4}}{3}						$&\\ \hline
						&	$\tr{\ff}{\mathbf{h}_{4,4}}{1}									$&\\ \hline
\text{Degree 6}\quad		&	$\tr{\mathbf{h}_{2,12}}{\mathbf{h}_{4,4}}{3}						$&\\ \hline
\multicolumn{3}{|c|}{$3$ covariants of order $12$.} \\ \hline
\text{Degree 2}\quad		&	$\mathbf{h}_{2,12}												$&\\ \hline
\text{Degree 3}\quad		&	$\mathbf{h}_{3,12}:=\tr{\ff}{\mathbf{h}_{2,8}}{2}					$&\\ \hline		
\text{Degree 4}\quad		&	$\tr{\mathbf{h}_{2,12}}{\mathbf{h}_{2,4}}{2}						$&\\ \hline
\multicolumn{3}{|c|}{$3$ covariants of order $14$.} \\ \hline
\text{Degree 3}\quad		&	$\mathbf{h}_{3,14}:=\tr{\ff}{\mathbf{h}_{2,8}}{1}					$&\\ \hline
\text{Degree 4}\quad		&	$\tr{\mathbf{h}_{2,12}}{\mathbf{h}_{2,4}}{1}						$&\\ \hline		
\text{Degree 5}\quad		&	$\tr{\mathbf{h}_{3,12}}{\mathbf{h}_{2,4}}{1}						$&\\ \hline
\multicolumn{3}{|c|}{$2$ covariants of order $18$.} \\ \hline
\text{Degree 3}\quad		&	$\mathbf{h}_{3,18}:=\tr{\ff}{\mathbf{h}_{2,12}}{1}					$&\\ \hline
\text{Degree 4}\quad		&	$\tr{\mathbf{h}_{2,12}}{\mathbf{h}_{2,8}}{1}						$&\\ \hline		
\end{supertabular}
\end{center}
\end{enumerate}

\subsection{Invariant bases of $\Sn{8}\oplus\Sn{4}\oplus\Sn{4}$}\label{subsec:Inv_Bas_Elast}
\label{subsec:Cov844}

For such an invariant bases, we apply theorem~\ref{thm:CovJointsRed} with $V_1=\Sn{8}$ and $V_2=\Sn{4}\oplus\Sn{4}$. Note here that we make use of a different covariant bases\footnote{Such a covariant bases had been used to obtain a covariant bases of $\cov{\Sn{10}}$~\cite{OL2014}.} of $\cov{\Sn{8}}$ than the one given in~\autoref{subsec:Cov8}.

All the covariant basis of $\cov{\Sn{8}}$ and $\cov{\Sn{4}\oplus\Sn{4}}$ are given in tables~\ref{table:CovS4S4} and~\ref{table:CovS8b}. In those tables, $\bh_{n}$ (resp. $\ff_n$) denotes the covariant of $\Sn{4}\oplus \Sn{4}$ (resp. $\Sn{8}$) defined by line numbered $n$. 

\begin{center}
\begin{table}[H]
\setlength{\tabcolsep}{4pt}
\begin{tabular}{c||cccc|c|c}
 d/o &   0 &   2 &   4 &   6 & \# & Cum \\\hline\hline
   1 &   - &   - &   2 &   - &  2 &   2 \\
   2 &   3 &   1 &   3 &   1 &  8 &   10 \\
   3 &   4 &   2 &   2 &   4 &	12 &  22 \\
   4 &   1 &   3 &   - &   - &	4 &   26 \\
   5 &   - &   2 &   - &   - &  2 &   28 \\ \hline
 Tot &   8 &  8  &   7 &   5 &   &    28  \\
\end{tabular}
\begin{tabular}{|c||c|c||c||c|c|}
\hline
Number & Covariant & $(d_{1},d_{2},o)$ & Number & Covariant & $(d_{1},d_{2},o)$ \\ \hline
1 & $\vv_{1}$					& $(1,0,4)$ & 15 & $\tr{\vv_{1}}{\bh_{8}}{3}$	& $(1,2,2)$ \\ 
2 & $\vv_{2}$					& $(0,1,4)$ & 16 & $\tr{\vv_{2}}{\bh_{7}}{3}$	& $(2,1,2)$ \\ 
3 & $\tr{\vv_{1}}{\vv_{1}}{4}$	& $(2,0,0)$ & 17 & $\tr{\vv_{1}}{\bh_{8}}{2}$	& $(1,2,4)$ \\
4 & $\tr{\vv_{2}}{\vv_{2}}{4}$	& $(0,2,0)$ & 18 & $\tr{\vv_{2}}{\bh_{7}}{2}$	& $(2,1,4)$\\
5 & $\tr{\vv_{1}}{\vv_{2}}{4}$	& $(1,1,0)$ & 19 & $\tr{\vv_{1}}{\bh_{7}}{1}$	& $(3,0,6)$\\
6 & $\tr{\vv_{1}}{\vv_{2}}{3}$	& $(1,1,2)$ & 20 & $\tr{\vv_{2}}{\bh_{8}}{1}$	& $(0,3,6)$ \\
7 & $\tr{\vv_{1}}{\vv_{1}}{2}$	& $(2,0,4)$ & 21 & $\tr{\vv_{1}}{\bh_{8}}{1}$	& $(1,2,6)$\\
8 & $\tr{\vv_{2}}{\vv_{2}}{2}$	& $(0,2,4)$ & 22 & $\tr{\vv_{2}}{\bh_{7}}{1}$	& $(2,1,6)$\\
9 & $\tr{\vv_{1}}{\vv_{2}}{2}$	& $(1,1,4)$ & 23 & $\tr{\bh_{7}}{\bh_{8}}{4}$	& $(2,2,0)$ \\
10 & $\tr{\vv_{1}}{\vv_{2}}{1}$	& $(1,1,6)$ & 24 & $\tr{\bh_{7}}{\bh_{8}}{3}$	& $(2,2,2)$ \\
11 & $\tr{\vv_{1}}{\bh_{7}}{4}$	& $(3,0,0)$ & 25 & $\tr{\bh_{19}}{\vv_{2}}{4}$	& $(3,1,2)$\\
12 & $\tr{\vv_{2}}{\bh_{8}}{4}$	& $(0,3,0)$ & 26 & $\tr{\vv_{1}}{\bh_{20}}{4}$	& $(1,3,2)$\\
13 & $\tr{\vv_{1}}{\bh_{8}}{4}$	& $(1,2,0)$ & 27 & $\tr{\vv_{1}^{2}}{\bh_{20}}{6}$	& $(2,3,2)$ \\
14 & $\tr{\vv_{2}}{\bh_{7}}{4}$	& $(2,1,0)$ & 28 & $\tr{\bh_{19}}{\vv_{2}^{2}}{6}$	& $(3,2,2)$\\ \hline
\end{tabular}\caption{Covariant bases of $\cov{\Sn{4}\oplus\Sn{4}}$.}\label{table:CovS4S4}
\end{table}
\end{center}

\begin{center}
\begin{table}[H]
\begin{tabular}{|c||c|c||c||c|c||c||c|c|}
\hline
Number & Covariant & (d,o) & Number & Covariant & (d,o) & Number & Covariant & (d,o)\\ \hline
1 & $\ff$ 							& $(1,8)$ & 26 & $\tr{\ff_{21}}{\ff}{8}$ & $(5,4)$ & 51 & $\tr{\ff_{41}}{\ff}{4}$ & $(7,6)$\\
2 & $\tr{\ff}{\ff}{8}$ 			& $(2,0)$ & 27 & $\tr{\ff_{20}}{\ff}{7}$ & $(5,4)$ & 52 & $\tr{\ff_{7}\ff_{16}}{\ff}{8}$ & $(8,0)$\\						
3 & $\tr{\ff}{\ff}{6}$ 			& $(2,4)$ & 28 & $\tr{\ff_{22}}{\ff}{8}$ & $(5,6)$ & 53 & $\tr{\ff_{51}}{\ff}{6}$ & $(8,2)$\\
4 & $\tr{\ff}{\ff}{4}$ 			& $(2,8)$ & 29 & $\tr{\ff_{21}}{\ff}{7}$ & $(5,6)$ & 54 & $\tr{\ff_{50}}{\ff}{6}$ & $(8,2)$\\
5 & $\tr{\ff}{\ff}{2}$ 			& $(2,12)$ & 30 & $\tr{\ff_{22}}{\ff}{7}$ & $(5,8)$ & 55 & $\tr{\ff_{51}}{\ff}{5}$ & $(8,4)$\\
6 & $\tr{\ff_{4}}{\ff}{8}$ 		& $(3,0)$ & 31 & $\tr{\ff_{23}}{\ff}{8}$ & $(5,10)$ & 56 & $\tr{\ff_{50}}{\ff}{5}$ & $(8,4)$\\
7 & $\tr{\ff_{5}}{\ff}{8}$ 		& $(3,4)$ & 32 & $\tr{\ff_{22}}{\ff}{6}$ & $(5,10)$ & 57 & $\tr{\ff_{51}}{\ff}{4}$ & $(8,6)$\\
8 & $\tr{\ff_{5}}{\ff}{7}$ 		& $(3,6)$ & 33 & $\tr{\ff_{21}}{\ff}{5}$ & $(5,10)$ & 58 & $\tr{\ff_{50}}{\ff}{4}$ & $(8,6)$\\
9 & $\tr{\ff_{5}}{\ff}{6}$ 		& $(3,8)$ & 34 & $\tr{\ff_{23}}{\ff}{6}$ & $(5,14)$ & 59 & $\tr{\ff_{15}\ff_{16}}{\ff}{8}$ & $(9,0)$\\
10 &$\tr{\ff_{5}}{\ff}{5}$			& $(3,10)$ & 35 & $\tr{\ff_{3}\ff_{7}}{\ff}{8}$ & $(6,0)$ & 60 & $\tr{\ff_{58}}{\ff}{6}$ & $(9,2)$\\
11 &$\tr{\ff_{5}}{\ff}{4}$			& $(3,12)$ & 36 & $\tr{\ff_{33}}{\ff}{8}$ & $(6,2)$ & 61 & $\tr{\ff_{57}}{\ff}{6}$ & $(9,2)$\\
12 &$\tr{\ff_{5}}{\ff}{3}$			& $(3,14)$ & 37 & $\tr{\ff_{33}}{\ff}{7}$ & $(6,4)$ & 62 & $\tr{\ff_{16}\ff_{17}}{\ff}{8}$ & $(9,2)$\\
13 &$\tr{\ff_{5}}{\ff}{1}$			& $(3,18)$ & 38 & $\tr{\ff_{32}}{\ff}{7}$ & $(6,4)$ & 63 & $\tr{\ff_{58}}{\ff}{5}$ & $(9,4)$\\
14 &$\tr{\ff_{9}}{\ff}{8}$			& $(4,0)$ & 39 & $\tr{\ff_{34}}{\ff}{8}$ & $(6,6)$ & 64 & $\tr{\ff_{17}\ff_{25}}{\ff}{8}$ & $(10,0)$\\
15 &$\tr{\ff_{11}}{\ff}{8}$		& $(4,4)$ & 40 & $\tr{\ff_{33}}{\ff}{6}$ & $(6,6)$ & 65 & $\tr{\ff_{17}\ff_{27}}{\ff}{8}$ & $(10,2)$\\
16 &$\tr{\ff_{10}}{\ff}{7}$		& $(4,4)$ & 41 & $\tr{\ff_{32}}{\ff}{6}$ & $(6,6)$ & 66 & $\tr{\ff_{17}\ff_{26}}{\ff}{8}$ & $(10,2)$\\
17 &$\tr{\ff_{12}}{\ff}{8}$		& $(4,6)$ & 42 & $\tr{\ff_{34}}{\ff}{7}$ & $(6,8)$ & 67 & $\tr{\ff_{27}\ff_{29}}{\ff}{8}$ & $(11,2)$\\
18 &$\tr{\ff_{12}}{\ff}{7}$		& $(4,8)$ & 43 & $\tr{\ff_{34}}{\ff}{6}$ & $(6,10)$ & 68 & $\tr{\ff_{27}\ff_{28}}{\ff}{8}$ & $(11,2)$\\
19 &$\tr{\ff_{13}}{\ff}{8}$		& $(4,10)$ & 44 & $\tr{\ff_{7}^{2}}{\ff}{8}$ & $(7,0)$ & 69 & $\tr{\ff_{29}\ff_{38}}{\ff}{8}$ & $(12,2)$ \\ \cline{7-9}
20 &$\tr{\ff_{12}}{\ff}{6}$		& $(4,10)$ & 45 & $\tr{\ff_{43}}{\ff}{8}$ & $(7,2)$ \\
21 &$\tr{\ff_{13}}{\ff}{7}$		& $(4,12)$ & 46 & $\tr{\ff_{42}}{\ff}{7}$ & $(7,2)$\\
22 &$\tr{\ff_{13}}{\ff}{6}$		& $(4,14)$ & 47 & $\tr{\ff_{43}}{\ff}{7}$ & $(7,4)$\\
23 &$\tr{\ff_{13}}{\ff}{4}$		& $(4,18)$ & 48 & $\tr{\ff_{42}}{\ff}{6}$ & $(7,4)$\\
24 &$\tr{\ff_{3}^{2}}{\ff}{8}$		& $(5,0)$ & 49 & $\tr{\ff_{43}}{\ff}{6}$ & $(7,6)$ \\
25 &$\tr{\ff_{20}}{\ff}{8}$ & $(5,2)$ & 50 & $\tr{\ff_{42}}{\ff}{5}$ & $(7,6)$\\ \cline{1-6}
\end{tabular}\caption{Covariant bases of $\cov{\Sn{8}}$.}\label{table:CovS8b}
\end{table}
\end{center}

To get an invariant bases of the invariant algebra $\inv{\Sn{8}\oplus\Sn{4}\oplus\Sn{4}}$, we use the same strategy as the one used for the computation of a covariant bases of $\cov{\Sn{6}\oplus\Sn{4}}$ (see \autoref{subsec:Cov64}). We give now details of the three steps of this strategy.

\subsubsection*{Resolution of the associated linear Diophantine system}

To apply Gordan's algorithm given in theorem~\ref{thm:CovJointsRed}, we have to solve a linear Diophantine system associated to covariant \emph{orders} (and thus excluding invariants) given by tables~\ref{table:CovS4S4} and~\ref{table:CovS8b}.

We thus have a linear Diophantine system with $81$ unknowns:
\ban
(S_e):\left\lbrace
\begin{split}
2x_{1,1}	&+\dotsc+ 2x_{14,1}+4x_{1,2}+\dotsc+	4x_{13,2}+6x_{1,3}+\dotsc +6x_{12,3}+\\
		&8x_{1,4}+\dots +8x_{6,4}+10x_{1,5}+\dotsc +10x_{7,5}+12x_{1,6}+\dotsc +12x_{3,6}+\\
		&\quad 14x_{1,7}+\dotsc +14x_{3,7}+18x_{1,8}+18x_{2,8}=r	 \\
2y_{1,1}&+\dotsc+2y_{8,1}+4y_{1,2}+\dotsc+4y_{7,2}+6y_{1,3}+\dotsc +6y_{5,3}=r
\end{split}
\right.
\ean

To solve such a system, we use Clausen--Fortenbacher's result~\cite{CF1990} on reduced systems. First note $(a_{1},a_{2},\dotsc,a_{8},b_{1},b_{2},b_{3})$ the irreducible solutions of the reduced system: \ban
(S'):\begin{cases}
	2X_{1}+4X_{2}+6X_{3}+8X_{4}+10X_{5}+12X_{6}+14X_{7}+18X_{8}&=r \\
	2Y_{1}+4Y_{2}+6Y_{3}&=r
	\end{cases}	
\ean
We then get irreducible solutions of the initial system $(S_e)$ by solving all the systems
\ban
\begin{cases}
	x_{1,1}+\dotsc+x_{14,1}&=a_{1} \\
	\dotsc \\
	x_{1,8}+x_{2,8}&=a_{8}	 \\
	y_{1,1}+\dotsc +y_{8,1}&=b_{1}\\
	y_{1,2}+\dotsc +y_{7,2}&=b_{2} \\
	y_{1,3}+\dotsc +y_{5,3}&=b_{3}
\end{cases}
\ean
Note also that to get irreducible solutions of the reduced system $(S')$, we used Normaliz package~\cite{BI2010} in Macaulay2~\cite{M2}. Finally we have:

\begin{lem}
The system $(S_e)$ has 695 754 irreducible integer solutions, corresponding to invariants from degree $3$ to degree $49$. 
\end{lem}

\subsubsection*{Relations on $\cov{\Sn{8}}$ and $\cov{\Sn{4}\oplus\Sn{4}}$}

Take here the minimal bases of $\cov{\Sn{8}}$ given in table~\ref{table:CovS8b} and take the lexicographic order:  
\ban
\begin{array}{l}
\bh_{69}> \bh_{67}> \bh_{68}> \bh_{66}> \bh_{65}> \bh_{64}> \bh_{63}> \bh_{62}> \bh_{61}> \bh_{60}> \bh_{59}> \bh_{58}> \\ 
\bh_{57}> \bh_{55}> \bh_{56}> \bh_{53}> \bh_{54}> \bh_{52}> \bh_{49}> \bh_{50}> \bh_{51}> \bh_{48}> \bh_{47}> \bh_{45}>
\\ 
\bh_{46}> \bh_{43}> \bh_{42}> \bh_{40}> \bh_{41}> \bh_{39}> \bh_{37}> \bh_{38}> \bh_{36}> \bh_{34}> \bh_{32}> \bh_{33}>
\\ 
\bh_{31}> \bh_{30}> \bh_{28}> \bh_{29}> \bh_{27}> \bh_{26}> \bh_{25}> \bh_{23}> \bh_{22}> \bh_{21}> \bh_{19}>\bh_{20}> 
\\
\bh_{18}> \bh_{17}> \bh_{15}> \bh_{16}> \bh_{13}> \bh_{12}> \bh_{11}> \bh_{10}> \bh_{9}> \bh_{8}> \bh_{7}> \bh_{5}>  \\
\bh_{4}> \bh_{3}> \bh_{1}
\end{array}
\ean

An algorithm developed by Lercier~\cite{OL2014} leads to:

\begin{lem}
There exists 1723 relations $\mathcal{R}_i$ and $\mathcal{R}'_j$ which verify hypotheses~\ref{hypo:Rel_Mon_1facteur} and~\ref{hypo:Rel_Mon_2facteur}. 
\end{lem}

By theorem~\ref{thm:CovJointsRed}, we thus reduce the family of 695 754 invariants to a family of 508 021 invariants. 

Using scripts written in Macaulay 2 and direct computations, we found:

\begin{lem}\label{lem:Rel_CovS4}
There exists $179$ monomial covariants $\mathbf{m}\in \cov{\Sn{4}\oplus\Sn{4}}$ contained in the invariant ideal of $\cov{\Sn{4}\oplus\Sn{4}}$. Furthermore, there exists $98$ relations 
\ban
	\mathbf{V}=\sum \mathbf{V}_i,\quad \mathbf{V},\mathbf{V}_i \text{ monomials in } \cov{\Sn{4}\oplus\Sn{4}}
\ean
\end{lem}

From this lemma~\ref{lem:Rel_CovS4} and from lemmas~\ref{lem:Reduc_Inv_Ideal} and~\ref{lem:Reduc_Int_Sol}, we get: 

\begin{lem}
All invariants from degree $22$ to degree $49$ are reducibles. 
\end{lem}

\begin{proof}
We have to consider invariants given as transvectants
\ban
	\tr{\mathbf{U}}{\mathbf{V}}{r},\quad r\leq 0
\ean	
where $\mathbf{U}$ (resp. $\mathbf{V}$) is a monomial in $\cov{\Sn{8}}$ (resp. $\cov{\Sn{4}\oplus\Sn{4}}$). From lemma~\ref{lem:Reduc_Inv_Ideal}, we know that every time one monomial $\mathbf{m}$ (given by one of the $179$ first relations of lemma~\ref{lem:Rel_CovS4}) divide $\mathbf{V}$, then the invariant is a reducible one. We get here a first reduction process. For instance, for degree $26$ invariants, we initially have 20 392 invariants, and this first reduction leads to 1822 invariants. We now use lemma~\ref{lem:Reduc_Int_Sol} for a second reduction process. For degree $26$ invariant, we have for example to consider the invariant
\begin{equation}\label{eq:Trans_Rest}
	\tr{\ff_{18}\ff_{12}\ff_{22}}{\bh_{10}^{3}\bh_{19}^{3}}{36}.
\end{equation}
In that case, we have the relation
\ban
	12\bh_{19}^{2}+6\bh_{7}^{3}+2\bh_{11}\bh_{1}^{3}-3\bh_{3}\bh_{1}^{2}\bh_{7}=0
\ean
which leads to consider invariants 
\ban
	\tr{\ff_{18}\ff_{12}\ff_{22}}{\bh_{10}^{3}\bh_{19}\bh_{7}^{3}}{36},\quad \tr{\ff_{18}\ff_{12}\ff_{22}}{\bh_{10}^{3}\bh_{19}\bh_{11}\bh_{1}^{3}}{36},\quad \tr{\ff_{18}\ff_{12}\ff_{22}}{\bh_{10}^{3}\bh_{19}\bh_{3}\bh_{1}^{2}\bh_{7}}{36}
\ean
where $\bh_{11}$ and $\bh_{3}$ are invariants (thus the two last transvectants are reducible). By a direct computation, we can check that the transvectant
\ban
	\tr{\ff_{18}\ff_{12}\ff_{22}}{\bh_{10}^{3}\bh_{19}\bh_{7}^{3}}{36}
\ean
correspond to a reducible integer solution. Using lemma~\ref{lem:Reduc_Int_Sol} we thus deduce that transvectant~\eqref{eq:Trans_Rest} is expressible in terms of reduced invariant and lower index transvectants. We then use the same arguments for lower indexe transvectants, which are all reducible.   
\end{proof}

Now there still remain $257\:770$ invariants, from degree $3$ to degree $21$. Direct computation in the algebra $\cov{\Sn{8}}$ leads to:

\begin{lem}
There exists $4085$ monomial covariants $\overline{\mathbf{m}}\in \cov{\Sn{8}}$ contained in the invariant ideal of $\cov{\Sn{8}}$. Furthermore, there exists $964$ relations 
\ban
	\mathbf{U}=\sum \mathbf{U}_i,\quad \mathbf{U},\mathbf{U}_i \text{ monomials in } \cov{\Sn{8}}
\ean
\end{lem}

Using those relations and lemmas~\ref{lem:Reduc_Inv_Ideal} and~\ref{lem:Reduc_Int_Sol} thus lead to a third reduction process. We can now make use of the multigraduate Hilbert series of $\inv{\Sn{8}\oplus\Sn{4}\oplus\Sn{4}}$ to get our final result. 

\subsubsection*{Finite minimal bases of $\inv{\Sn{8}\oplus\Sn{4}\oplus\Sn{4}}$}

\begin{thm}\label{thm:Inv_Elas_Bin}
The invariant algebra $\inv{\Sn{8}\oplus\Sn{4}\oplus\Sn{4}}$ is generated by a minimal bases of 297 invariants, resumed\footnote{We note here $\invj{V_1\oplus V_2}$ a set of joint invariants of degree $d_1>0$ and $d_2>0$ in $V_1$ and $V_2$.} 
 in table~\ref{table:Inv844}.
\end{thm}

\begin{table}[H]
\setlength{\arraycolsep}{6pt}
\begin{equation*}
\begin{array}{c||cccccc}
 \text{Degree} &  \inv{\Sn{8}} &   \inv{\Sn{4}} & \invj{\Sn{4}\oplus\Sn{4}} & \invj{\Sn{8}\oplus\Sn{4}}  &  \invj{\Sn{8}\oplus\Sn{4}\oplus\Sn{4}} \\\hline\hline
   1 &   -&   1  &   -  & - &  -  \\
   2 &   1 &   -  &  1 & - &  - \\
   3 &   1 &   -  & 2 & 2 &  1 \\
   4 &   1 &   -  & 1 & 4 &  6 \\
   5 &   1 &   -  & - & 7 &  18 \\
   6 &   1 &   -  & - & 10 & 36\\
   7 &   1 &   -  & - & 11 & 53 \\
   8 &   1 &   -  & - & 10 & 45  \\
   9 &   1 &   -  & - & 5 &  10 \\
  10 &   1 &   -  & - & 2 &  2  \\
  11 &   - &   -  & - & 2 &  3 \\ \hline
 Tot &   9 &   1  & 4 & 53 & 174 \\
\end{array}
\end{equation*}
\caption{Minimal bases of $\inv{\Sn{8}\oplus\Sn{4}\oplus\Sn{4}}$}\label{table:Inv844}
\end{table}
\begin{proof}
Let $d_{1}$ be the invariant degree in $\ff\in \Sn{8}$, $d_{2}$ the degree in $\vv_1\in \Sn{4}$ and $d_{3}$ the degree in $\vv_2\in \Sn{4}$. Thus we have
\ban
	\inv{\Sn{8}\oplus\Sn{4}\oplus\Sn{4}}=\bigoplus_{d_{1},d_{2},d_{3}\geq 0} \mathbf{Inv}_{d_{1},d_{2},d_{3}} (\Sn{8}\oplus\Sn{4}\oplus\Sn{4}) 
\ean
Using multigraduated Hilbert series computed by Bedratyuk's Maple package~\cite{Bed2011}, we can compute the minimal bases degree per degree, as explained in \autoref{subsec:Hilb_Ser}. For instance, we have left $740$ degree $12$ invariants. Note that for each of those invariants, we know the associated homogeneous spaces, given in table~\ref{table:EspHomogeneDeg12Elast}.

\begin{table}[H]
\begin{center}
\begin{tabular}{|c|c||c|c||c|c|}
\hline
$d_{1},d_{2},d_{3}$ & Dimension & $d_{1},d_{2},d_{3}$ & Dimension & $d_{1},d_{2},d_{3}$ & Dimension \\ \hline
4, 4, 4 & 1004 	& 3, 9, 0 & 44 	& 8, 4, 0 & 176 \\
6, 3, 3 &1003	& 5, 7, 0 & 126 	& 6, 5, 1 & 494 \\
8, 2, 2 & 544	& 5, 6, 1 & 414	& 6, 4, 2 & 871 \\
10, 1, 1 & 135	& 4, 6, 2 & 611	& 7, 4, 1 & 488 \\
4, 8, 0 & 91		& 4, 5, 3 & 872	& 9, 3, 0 & 131 \\
 3, 4, 5 & 695	& 4, 7, 1 & 290	& 10, 2, 0 & 95 \\
 3, 8, 1 & 157	& 5, 5, 2 & 788	& 7, 2, 3 & 747 \\
 3, 7,2 & 350 	& 5, 3, 4 & 1046 & 8, 3, 1 & 404 \\
 3, 6, 3 & 558 	& 7, 5, 0 & 176	& 9, 1, 2 & 271 \\
 \hline
\end{tabular}\caption{Homogeneous spaces in degree $12$}\label{table:EspHomogeneDeg12Elast}
\end{center}
\end{table}

Using scripts written in Macaulay2~\cite{M2}, we thus checked all homogeneous spaces for the finite family already obtained. This leads to no irreducible invariants for this degree. Such computations had thus been done to homogeneous spaces up to degree $21$.  
\end{proof}

We now give joint invariants of $\Sn{8}\oplus \Sn{4}$. For that purpose, we write $\vv\in \Sn{4}$ and 
\ban
	\kk_{2,4}:=\tr{\vv}{\vv}{2},\quad \kk_{3,6}:=\tr{\vv}{\kk_{2,4}}{1}.
\ean
\begin{center}
\renewcommand{\arraystretch}{1.2}
\tablefirsthead{%
\multicolumn{3}{c}{}\\
}
\tablehead{%
\multicolumn{3}{r}{\small\sl continued from previous page} \\ \hline
}
\tabletail{%
\multicolumn{3}{r}{\small\sl continued on next page}
\\
}
\tablelasttail{\hline}
\begin{supertabular}{|l|ll|}
\hline
\multicolumn{3}{|c|}{$53$ joint invariants of $\invj{\Sn{8}\oplus\Sn{4}}$.} \\ \hline
Degree 3\quad	 	&	$\tr{\ff_3}{\vv}{4}$					\quad
						$\tr{\ff_1}{\vv^2}{8}$ 					&\\ \hline
Degree 4\quad		&	$\tr{\ff_1}{\vv\cdot\kk_{2,4}}{8}$		\quad
						$\tr{\ff_4}{\vv^2}{8}$					\quad
						$\tr{\ff_3}{\kk_{2,4}}{4}$				\quad
						$\tr{\ff_7}{\vv}{4}$					&\\ \hline
Degree 5\quad		& 	$\tr{\ff_1}{\kk_{2,4}^2}{8}$			\quad
						$\tr{\ff_4}{\vv\cdot\kk_{2,4}}{8}$		\quad
						$\tr{\ff_5}{\vv^3}{12}$				\quad
						$\tr{\ff_7}{\kk_{2,4}}{4}$				& \\
					&	$\tr{\ff_9}{\vv^2}{8}$					\quad
						$\tr{\ff_{15}}{\vv}{4}$					\quad
						$\tr{\ff_{16}}{\vv}{4}$ 					&\\ \hline
Degree 6\quad		&	$\tr{\ff_4}{\kk_{2,4}^2}{8} $			\quad
						$\tr{\ff_5}{\vv^2\cdot\kk_{2,4}}{12}$	\quad
						$\tr{\ff_{11}}{\vv^3}{12}$				\quad
						$\tr{\ff_9}{\vv\cdot\kk_{2,4}}{8}$		& \\
					&	$\tr{\ff_8}{\kk_{3,6}}{6}$				\quad
						$\tr{\ff_{15}}{\kk_{2,4}}{4}$			\quad
						$\tr{\ff_{18}}{\vv^2}{8}$				\quad
						$\tr{\ff_{16}}{\kk_{2,4}}{4}$			& \\
					&	$\tr{\ff_{26}}{\vv}{4}$				\quad
						$\tr{\ff_{27}}{\vv}{4}$				&\\ \hline
Degree 7\quad		&	$\tr{\ff_5}{\vv\cdot\kk_{2,4}^2}{12}$	\quad
						$\tr{\ff_{10}}{\vv\cdot\kk_{3,6}}{10}$	\quad
						$\tr{\ff_{11}}{\vv^2\cdot\kk_{2,4}}{12}$	\quad
						$\tr{\ff_{18}}{\vv\cdot\kk_{2,4}}{8}$	& \\
					&	$\tr{\ff_{17}}{\kk_{3,6}}{6}$			\quad
						$\tr{\ff_{21}}{\vv^3}{12}$				\quad
						$\tr{\ff_{30}}{\vv^2}{8}$				\quad
						$\tr{\ff_{27}}{\kk_{2,4}}{4}$			& \\
					&	$\tr{\ff_{26}}{\kk_{2,4}}{4}$			\quad
						$\tr{\ff_{37}}{\vv}{4}$				\quad
						$\tr{\ff_{38}}{\vv}{4}$  									&\\ \hline
Degree 8\quad		& 	$\tr{\ff_{47}}{\vv}{4}$				\quad
						$\tr{\ff_{48}}{\vv}{4}$				\quad
						$\tr{\ff_{37}}{\kk_{2,4}}{4}$				\quad
						$\tr{\ff_{38}}{\kk_{2,4}}{4}$				& \\
					&	$\tr{\ff_{42}}{\vv^2}{8}$				\quad
						$\tr{\ff_{29}}{\kk_{3,6}}{6}$				\quad
						$\tr{\ff_{30}}{\vv\cdot\kk_{2,4}}{8}$				\quad
						$\tr{\ff_{20}}{\vv\cdot\kk_{3,6}}{10}$				& \\
					&	$\tr{\ff_{21}}{\vv^2\cdot\kk_{2,4}}{12}$				\quad
						$\tr{\ff_{11}}{\vv\cdot\kk_{2,4}^2}{12}$								&\\ \hline
Degree 9\quad 	&	$\tr{\ff_8^2}{\vv^3}{12}$				\quad
					$\tr{\ff_{48}}{\kk_{2,4}}{4}$				\quad
					$\tr{\ff_{47}}{\kk_{2,4}}{4}$				\quad
					$\tr{\ff_{55}}{\vv}{4}$				\quad
					$\tr{\ff_{56}}{\vv}{4}$ 								&\\ \hline
Degree 10\quad	&	$\tr{\ff_{56}}{\kk_{2,4}}{4}$			\quad
					$\tr{\ff_{63}}{\vv}{4}$ 										&\\ \hline
Degree 11\quad	&	$\tr{\ff_{63}}{\kk_{2,4}}{4}$		\quad
					$\tr{\ff_{25}^2}{\vv}{4}$									& \\ \hline
\end{supertabular}
\end{center}

Finally, we give joint invariants of $\Sn{8}\oplus \Sn{4}\oplus \Sn{4}$. Recall here that $\bh_{n}$ is defined to be the number $n$ covariant in the covariant bases of $\cov{\Sn{4}\oplus\Sn{4}}$, given in table~\ref{table:CovS4S4}.

\begin{center}
\renewcommand{\arraystretch}{1.2}
\tablefirsthead{%
\multicolumn{3}{c}{}\\
}
\tablehead{%
\multicolumn{3}{r}{\small\sl continued from previous page} \\ \hline
}
\tabletail{%
\multicolumn{3}{r}{\small\sl continued on next page}
\\
}
\tablelasttail{\hline}
\begin{supertabular}{|l|ll|}
\hline
\multicolumn{3}{|c|}{$174$ joints invariant of $\invj{\Sn{8}\oplus\Sn{4}\oplus\Sn{4}}$.} \\ \hline
Degree 3\quad	 	&	$\tr{\ff_{1}}{\bh_{1}\cdot\bh_{2}}{8}$ 				&\\ \hline
Degree 4\quad		&	$\tr{\ff_{1}}{\bh_{1}\cdot\bh_{8}}{8}$ 				\quad
						$\tr{\ff_{1}}{\bh_{2}\cdot\bh_{9}}{8}$				\quad
						$\tr{\ff_{1}}{\bh_{2}\cdot\bh_{7}}{8}$ 				\quad
						$\tr{\ff_{1}}{\bh_{1}\cdot\bh_{9}}{8}$				& \\ &
						$\tr{\ff_{3}}{\bh_{9}}{4}$ 							\quad
						$\tr{\ff_{4}}{\bh_{1}\cdot\bh_{2}}{8}$				&\\ \hline
Degree 5\quad		& 	$\tr{\ff_{1}}{\bh_{8}\cdot\bh_{9}}{8}$ 				\quad
						$\tr{\ff_{1}}{\bh_{2}\cdot\bh_{17}}{8}$ 			\quad
						$\tr{\ff_{1}}{\bh_{7}\cdot\bh_{8}}{8}$ 				\quad
						$\tr{\ff_{1}}{\bh_{2}\cdot\bh_{18}}{8}$ 			& \\ &
						$\tr{\ff_{1}}{\bh_{9}^{2}}{8}$ 						\quad
						$\tr{\ff_{1}}{\bh_{7}\cdot\bh_{9}}{8}$ 				\quad
						$\tr{\ff_{1}}{\bh_{1}\cdot\bh_{18}}{8}$ 			\quad
						$\tr{\ff_{4}}{\bh_{1}\cdot\bh_{8}}{8}$				& \\ &
						$\tr{\ff_{4}}{\bh_{2}\cdot\bh_{9}}{8}$				\quad
						$\tr{\ff_{5}}{\bh_{1}\cdot\bh_{2}^{2}}{12}$ 			\quad
						$\tr{\ff_{3}}{\bh_{17}}{4}$ 						\quad
						$\tr{\ff_{4}}{\bh_{2}\cdot\bh_{7}}{8}$ 				& \\ &
						$\tr{\ff_{3}}{\bh_{18}}{4}$ 						\quad
						$\tr{\ff_{4}}{\bh_{1}\cdot\bh_{9}}{8}$ 				\quad
						$\tr{\ff_{5}}{\bh_{1}^{2}\cdot\bh_{2}}{12}$ 			\quad
						$\tr{\ff_{9}}{\bh_{1}\cdot\bh_{2}}{8}$ 				& \\ &
						$\tr{\ff_{7}}{\bh_{9}}{4}$ 							\quad
						$\tr{\ff_{8}}{\bh_{10}}{6}$ 	 					&\\ \hline
Degree 6\quad		&	$\tr{\ff_{1}}{\bh_{8}\cdot\bh_{17}}{8}$ 			\quad
						$\tr{\ff_{1}}{\bh_{2}\cdot\bh_{6}^{2}}{8}$ 			\quad
						$\tr{\ff_{1}}{\bh_{9}\cdot\bh_{17}}{8}$ 			\quad
						$\tr{\ff_{1}}{\bh_{9}\cdot\bh_{18}}{8}$ 			& \\ &
						$\tr{\ff_{1}}{\bh_{1}\cdot\bh_{6}^{2}}{8}$ 			\quad
						$\tr{\ff_{1}}{\bh_{7}\cdot\bh_{18}}{8}$ 			\quad
						$\tr{\ff_{4}}{\bh_{2}\cdot\bh_{17}}{8}$ 			& \\ &
						$\tr{\ff_{5}}{\bh_{1}\cdot\bh_{2}\cdot\bh_{8}}{12}$ \quad
						$\tr{\ff_{4}}{\bh_{8}\cdot\bh_{9}}{8}$ 				\quad
						$\tr{\ff_{5}}{\bh_{2}^{2}\cdot\bh_{9}}{12}$ 			\quad
						$\tr{\ff_{4}}{\bh_{2}\cdot\bh_{18}}{8}$ 			& \\ &
						$\tr{\ff_{5}}{\bh_{1}\cdot\bh_{2}\cdot\bh_{9}}{12}$ \quad
						$\tr{\ff_{5}}{\bh_{1}^{2}\cdot\bh_{8}}{12}$	 		\quad
						$\tr{\ff_{4}}{\bh_{9}^{2}}{8}$ 						& \\ &
						$\tr{\ff_{4}}{\bh_{7}\cdot\bh_{8}}{8}$ 				\quad
						$\tr{\ff_{5}}{\bh_{2}^{2}\cdot\bh_{7}}{12}$ 			\quad
						$\tr{\ff_{5}}{\bh_{1}^{2}\cdot\bh_{9}}{12}$ 			\quad
						$\tr{\ff_{4}}{\bh_{7}\cdot\bh_{9}}{8}$ 				& \\ &
						$\tr{\ff_{4}}{\bh_{1}\cdot\bh_{18}}{8}$ 			\quad
						$\tr{\ff_{5}}{\bh_{1}\cdot\bh_{2}\cdot\bh_{7}}{12}$ \quad
						$\tr{\ff_{9}}{\bh_{1}\cdot\bh_{8}}{8}$ 				& \\ &
						$\tr{\ff_{8}}{\bh_{21}}{6}$ 						\quad
						$\tr{\ff_{10}}{\bh_{2}\cdot\bh_{10}}{10}$ 			\quad
						$\tr{\ff_{8}}{\bh_{2}\cdot\bh_{6}}{6}$ 				\quad
						$\tr{\ff_{9}}{\bh_{2}\cdot\bh_{9}}{8}$ 				& \\ &
						$\tr{\ff_{11}}{\bh_{1}\cdot\bh_{2}^{2}}{12}$ 			\quad
						$\tr{\ff_{11}}{\bh_{1}^{2}\cdot\bh_{2}}{12}$ 			\quad
						$\tr{\ff_{10}}{\bh_{1}\cdot\bh_{10}}{10}$ 			\quad
						$\tr{\ff_{9}}{\bh_{2}\cdot\bh_{7}}{8}$ 				& \\ &
						$\tr{\ff_{9}}{\bh_{1}\cdot\bh_{9}}{8}$ 				\quad
						$\tr{\ff_{8}}{\bh_{1}\cdot\bh_{6}}{6}$ 				\quad
						$\tr{\ff_{8}}{\bh_{22}}{6}$ 						\quad
						$\tr{\ff_{16}}{\bh_{9}}{4}$ 						& \\ &
						$\tr{\ff_{17}}{\bh_{10}}{6}$ 						\quad
						$\tr{\ff_{18}}{\bh_{1}\cdot\bh_{2}}{8}$		 		\quad
						$\tr{\ff_{15}}{\bh_{9}}{4}$ 						&\\ \hline
Degree 7\quad		&	$\tr{\ff_{5}}{\bh_{2}^{2}\cdot\bh_{17}}{12}$ 			\quad
						$\tr{\ff_{5}}{\bh_{1}\cdot\bh_{8}^{2}}{12}$ 			\quad
						$\tr{\ff_{5}}{\bh_{2}\cdot\bh_{8}\cdot\bh_{9}}{12}$ & \\ &
						$\tr{\ff_{5}}{\bh_{2}^{2}\cdot\bh_{18}}{12}$ 			\quad
						$\tr{\ff_{5}}{\bh_{1}\cdot\bh_{8}\cdot\bh_{9}}{12}$ \quad
						$\tr{\ff_{5}}{\bh_{2}\cdot\bh_{7}\cdot\bh_{8}}{12}$ & \\ &
						$\tr{\ff_{5}}{\bh_{2}\cdot\bh_{9}^{2}}{12}$ 			\quad
						$\tr{\ff_{5}}{\bh_{1}\cdot\bh_{9}^{2}}{12}$			\quad
						$\tr{\ff_{5}}{\bh_{2}\cdot\bh_{7}\cdot\bh_{9}}{12}$ & \\ &
						$\tr{\ff_{5}}{\bh_{1}\cdot\bh_{2}\cdot\bh_{18}}{12}$\quad
						$\tr{\ff_{5}}{\bh_{1}\cdot\bh_{7}\cdot\bh_{8}}{12}$ \quad
						$\tr{\ff_{5}}{\bh_{1}\cdot\bh_{7}\cdot\bh_{9}}{12}$ & \\ &
						$\tr{\ff_{5}}{\bh_{1}^{2}\cdot\bh_{18}}{12}$ 			\quad
						$\tr{\ff_{5}}{\bh_{2}\cdot\bh_{7}^{2}}{12}$			\quad
						$\tr{\ff_{10}}{\bh_{2}\cdot\bh_{21}}{10}$ 			\quad
						$\tr{\ff_{10}}{\bh_{1}\cdot\bh_{20}}{10}$ 			& \\ &
						$\tr{\ff_{11}}{\bh_{2}^{2}\cdot\bh_{9}}{12}$ 			\quad
						$\tr{\ff_{11}}{\bh_{1}\cdot\bh_{2}\cdot\bh_{8}}{12}$\quad
						$\tr{\ff_{10}}{\bh_{2}^{2}\cdot\bh_{6}}{10}$ 			& \\ &
						$\tr{\ff_{12}}{\bh_{2}^{2}\cdot\bh_{10}}{14}$ 		\quad
						$\tr{\ff_{10}}{\bh_{1}\cdot\bh_{2}\cdot\bh_{6}}{10}$\quad
						$\tr{\ff_{11}}{\bh_{1}^{2}\cdot\bh_{8}}{12}$ 			& \\ &
						$\tr{\ff_{10}}{\bh_{1}\cdot\bh_{21}}{10}$ 			\quad
						$\tr{\ff_{10}}{\bh_{2}\cdot\bh_{22}}{10}$ 			\quad
						$\tr{\ff_{12}}{\bh_{1}\cdot\bh_{2}\cdot\bh_{10}}{14}$ & \\ &
						$\tr{\ff_{11}}{\bh_{1}\cdot\bh_{2}\cdot\bh_{9}}{12}$ \quad
						$\tr{\ff_{11}}{\bh_{2}^{2}\cdot\bh_{7}}{12}$ 			\quad
						$\tr{\ff_{9}}{\bh_{9}^{2}}{8}$ 						& \\ &
						$\tr{\ff_{10}}{\bh_{1}\cdot\bh_{22}}{10}$ 			\quad
						$\tr{\ff_{11}}{\bh_{1}\cdot\bh_{2}\cdot\bh_{7}}{12}$ \quad
						$\tr{\ff_{11}}{\bh_{1}^{2}\cdot\bh_{9}}{12}$ 			& \\ &
						$\tr{\ff_{10}}{\bh_{1}^{2}\cdot\bh_{6}}{10}$ 			\quad
						$\tr{\ff_{10}}{\bh_{2}\cdot\bh_{19}}{10}$ \quad
						$\tr{\ff_{12}}{\bh_{1}^{2}\cdot\bh_{10}}{14}$ \quad
						$\tr{\ff_{21}}{\bh_{1}\cdot\bh_{2}^{2}}{12}$ & \\ &
						$\tr{\ff_{18}}{\bh_{2}\cdot\bh_{9}}{8}$ \quad
						$\tr{\ff_{17}}{\bh_{21}}{6}$ \quad
						$\tr{\ff_{17}}{\bh_{2}\cdot\bh_{6}}{6}$ \quad
						$\tr{\ff_{20}}{\bh_{2}\cdot\bh_{10}}{10}$ & \\ &
						$\tr{\ff_{19}}{\bh_{2}\cdot\bh_{10}}{10}$ \quad
						$\tr{\ff_{18}}{\bh_{1}\cdot\bh_{8}}{8}$ \quad
						$\tr{\ff_{17}}{\bh_{1}\cdot\bh_{6}}{6}$ \quad
						$\tr{\ff_{18}}{\bh_{1}\cdot\bh_{9}}{8}$ & \\ &
						$\tr{\ff_{17}}{\bh_{22}}{6}$ \quad
						$\tr{\ff_{20}}{\bh_{1}\cdot\bh_{10}}{10}$ \quad
						$\tr{\ff_{21}}{\bh_{1}^{2}\cdot\bh_{2}}{12}$ \quad
						$\tr{\ff_{19}}{\bh_{1}\cdot\bh_{10}}{10}$ & \\ &
						$\tr{\ff_{18}}{\bh_{2}\cdot\bh_{7}}{8}$ \quad
						$\tr{\ff_{29}}{\bh_{10}}{6}$ \quad
						$\tr{\ff_{30}}{\bh_{1}\cdot\bh_{2}}{8}$ \quad
						$\tr{\ff_{26}}{\bh_{9}}{4}$ & \\ &
						$\tr{\ff_{27}}{\bh_{9}}{4}$ \quad
						$\tr{\ff_{28}}{\bh_{10}}{6}$ \quad  									&\\ \hline
Degree 8\quad		& 	$\tr{\ff_{37}}{\bh_{9}}{4}$ \quad
						$\tr{\ff_{38}}{\bh_{9}}{4}$ \quad
						$\tr{\ff_{40}}{\bh_{10}}{6}$ \quad
						$\tr{\ff_{41}}{\bh_{10}}{6}$ & \\ &
						$\tr{\ff_{42}}{\bh_{1}\cdot\bh_{2}}{8}$ \quad
						$\tr{\ff_{29}}{\bh_{21}}{6}$ \quad
						$\tr{\ff_{30}}{\bh_{1}\cdot\bh_{8}}{8}$ \quad
						$\tr{\ff_{30}}{\bh_{2}\cdot\bh_{9}}{8}$ & \\ &
						$\tr{\ff_{31}}{\bh_{2}\cdot\bh_{10}}{10}$ \quad
						$\tr{\ff_{32}}{\bh_{2}\cdot\bh_{10}}{10}$ \quad
						$\tr{\ff_{33}}{\bh_{2}\cdot\bh_{10}}{10}$ \quad
						$\tr{\ff_{29}}{\bh_{22}}{6}$ & \\ &
						$\tr{\ff_{30}}{\bh_{1}\cdot\bh_{9}}{8} $ \quad
						$\tr{\ff_{30}}{\bh_{2}\cdot\bh_{7}}{8}$ \quad
						$\tr{\ff_{31}}{\bh_{1}\cdot\bh_{10}}{10}$ \quad
						$\tr{\ff_{32}}{\bh_{1}\cdot\bh_{10}}{10}$ & \\ &
						$\tr{\ff_{33}}{\bh_{1}\cdot\bh_{10}}{10}$ \quad
						$\tr{\ff_{20}}{\bh_{2}\cdot\bh_{22}}{10}$ \quad
						$\tr{\ff_{20}}{\bh_{1}\cdot\bh_{2}\cdot\bh_{6}}{10}$ \quad
						$\tr{\ff_{21}}{\bh_{1}^{2}\cdot\bh_{8}}{12}$ & \\ &
						$\tr{\ff_{21}}{\bh_{1}\cdot\bh_{2}\cdot\bh_{9}}{12}$ \quad
						$\tr{\ff_{21}}{\bh_{2}^{2}\cdot\bh_{7}}{12}$ \quad
						$\tr{\ff_{22}}{\bh_{1}\cdot\bh_{2}\cdot\bh_{10}}{14}$ \quad
						$\tr{\ff_{20}}{\bh_{1}\cdot\bh_{22}}{10}$ & \\ &
						$\tr{\ff_{20}}{\bh_{1}^{2}\cdot\bh_{6}}{10}$ \quad
						$\tr{\ff_{21}}{\bh_{1}^{2}\cdot\bh_{9}}{12}$ \quad
						$\tr{\ff_{21}}{\bh_{1}\cdot\bh_{2}\cdot\bh_{7}}{12}$ \quad
						$\tr{\ff_{22}}{\bh_{1}^{2}\cdot\bh_{10}}{14}$ & \\ &
						$\tr{\ff_{11}}{\bh_{2}\cdot\bh_{7}\cdot\bh_{9}}{12}$ \quad
						$\tr{\ff_{12}}{\bh_{1}^{2}\cdot\bh_{2}\cdot\bh_{6}}{14}$ \quad
						$\tr{\ff_{13}}{\bh_{1}^{2}\cdot\bh_{2}\cdot\bh_{10}}{18}$ \quad
						$\tr{\ff_{11}}{\bh_{2}\cdot\bh_{7}^{2}}{12}$ & \\ &
						$\tr{\ff_{12}}{\bh_{1^3}\cdot\bh_{6}}{14}$ \quad
						$\tr{\ff_{13}}{\bh_{1^3}\cdot\bh_{10}}{18}$ \quad
						$\tr{\ff_{11}}{\bh_{2}\cdot\bh_{9}^{2}}{12}$ & \\ &
						$\tr{\ff_{12}}{\bh_{1}\cdot\bh_{2}^{2}\cdot\bh_{6}}{14}$ \quad
						$\tr{\ff_{13}}{\bh_{1}\cdot\bh_{2}^{2}\cdot\bh_{10}}{18}$ \quad
						$\tr{\ff_{20}}{\bh_{2}\cdot\bh_{21}}{10}$ & \\ &
						$\tr{\ff_{20}}{\bh_{2}^{2}\cdot\bh_{6}}{10}$ \quad
						$\tr{\ff_{21}}{\bh_{1}\cdot\bh_{2}\cdot\bh_{8}}{12}$ \quad
						$\tr{\ff_{21}}{\bh_{2}^{2}\cdot\bh_{9}}{12}$ & \\ &
						$\tr{\ff_{22}}{\bh_{2}^{2}\cdot\bh_{10}}{14}$ \quad
						$\tr{\ff_{11}}{\bh_{2}\cdot\bh_{8}\cdot\bh_{9}}{12}$ \quad
						$\tr{\ff_{12}}{\bh_{2^3}\cdot\bh_{6}}{14}$ & \\ &
						$\tr{\ff_{13}}{\bh_{2^3}\cdot\bh_{10}}{18}$ \quad								&\\ \hline
Degree 9\quad 	&	$\tr{\ff_{1}\cdot\ff_{25}}{\bh_{2}\cdot\bh_{10}}{10}$ \quad
					$\tr{\ff_{43}}{\bh_{2}\cdot\bh_{10}}{10}$ \quad
					$\tr{\ff_{8}^{2}}{\bh_{1}\cdot\bh_{2}^{2}}{12}$ \quad
					$\tr{\ff_{1}\cdot\ff_{25}}{\bh_{1}\cdot\bh_{10}}{10}$ & \\ &
					$\tr{\ff_{8}^{2}}{\bh_{1}^{2}\cdot\bh_{2}}{12}$ \quad
					$\tr{\ff_{43}}{\bh_{1}\cdot\bh_{10}}{10}$ \quad
					$\tr{\ff_{3}\cdot\ff_{25}}{\bh_{10}}{6}$ \quad
					$\tr{\ff_{51}}{\bh_{10}}{6}$ & \\ &
					$\tr{\ff_{48}}{\bh_{9}}{4}$ \quad
					$\tr{\ff_{47}}{\bh_{9}}{4}$ \quad								&\\ \hline
Degree 10\quad	&	$\tr{\ff_{54}}{\bh_{6}}{2}$ \quad
					$\tr{\ff_{56}}{\bh_{9}}{4}$ 										&\\ \hline
Degree 11\quad	&	$\tr{\ff_{61}}{\bh_{6}}{2}$ \quad
					$\tr{\ff_{62}}{\bh_{6}}{2}$ \quad
					$\tr{\ff_{63}}{\bh_{9}}{4}$ \quad									& \\ \hline
\end{supertabular}
\end{center}

\appendix

\section{The Stroh formula and some corollaries}

The following general algebraic relation was obtained by Stroh~\cite{Str1888} (see also~\cite{GY2010}).

\begin{lem}\label{lem:StrohEg}
Let $u_1$, $u_2$ and $u_3$ be three commutative variables such that
\begin{equation*}
	u_1+u_2+u_3=0.
\end{equation*}
Then we have
\begin{multline}\label{eq:Stroh}
  (-1)^{k_{2}} \sum_{i=0}^{k_{1}} \binom{g}{i} \binom{k_{1} + k_{3} - i}{k_{3}} u_{3}^{g - i} u_{1}^{i} + (-1)^{k_{3}} \sum_{i=0}^{k_{2}} \binom{g}{i} \binom{k_{2} + k_{1} - i}{k_{1}} u_{1}^{g - i} u_{2}^{i} + \\
  (-1)^{k_{1}} \sum_{i=0}^{k_{3}} \binom{g}{i} \binom{k_{3} + k_{2} - i}{k_{2}} u_{2}^{g-i} u_{3}^{i} = 0,
\end{multline}
with $k_1+k_2+k_3=g-1$.
\end{lem}

This formula leads to new degree three relations on molecules. Let $V=\Sn{n}$ and $(e_{0},e_{1},e_{2})$ be three integers such that $e_{i}+e_{j}\leq n$ ($i\neq j$). Define:

\begin{align}\label{fig:DegreeThree}
	\di{D}(e_0,e_1,e_2)
	&:=\begin{tikzpicture}[scale=1.5,baseline={([yshift=-.5ex]current bounding box.center)}]
			\node[m] (P)at(0.2,1.1){$\alpha$};
			\node[m] (Q)at(1.7,1.1){$\beta$};
			\node[m] (R)at(0.95,0.1){$\gamma$};
			\draw[flecheo] (P)--(Q) node[midway,above] {$e_0$};
			\draw[flecheo] (Q)--(R) node[pos=0.6,right] {$e_1$};
			\draw[flecheo] (R)--(P) node[pos=0.3,left] {$e_2$};			
	\end{tikzpicture}
	&\text{ with weight } w=e_0+e_1+e_2,
\end{align}
Note that $\di{D}(e_0,e_1,e_2)\in \Hom_{\sldc}(\Sn{n}\otimes\Sn{n}\otimes \Sn{n},\Sn{3n-2w})$.

\begin{lem}\label{lem:Basis_Family}
Let $w\leq n$ and $m_1,m_2,m_3\geq 1$ be integers such that $m_1+m_2+m_3=w+1$. Then the molecule $\mathsf{D}(e_0,e_1,e_2)$ is a linear combination of
\begin{equation*}
	\di{D}(w-i_{1},i_{1},0),\quad \di{D}(0,w-i_{2},i_{2}),\quad \di{D}(i_{3},0,w-i_{3}),
\end{equation*}
with $i_s=0\dotsc m_s-1$,
\end{lem}

\begin{proof}[Sketch of proof]
Using Clebsch--Gordan decomposition, first observe that
\begin{equation*}
	\dim \Hom_{\sldc}(\Sn{n}\otimes\Sn{n}\otimes \Sn{n},\Sn{3n-2w})=w+1
\end{equation*}
Suppose that we have a linear relation
\begin{equation*}
	\sum_{i=0}^{w} \lambda_{i}\di{D}(w-i,i,0)=0.
\end{equation*}
Taking $\ff_{\alpha}=x_{\alpha}^{n},\ff_{\beta}=y_{\beta}^{n}$ and $\ff_{\gamma}=y_{\gamma}^{n}$ leads to $\lambda_{0}=0$; and by induction we get $\lambda_{i}=0$ for all $i$. Thus $\mathcal{F}_{1}:=\set{\di{D}(w-i,i,0),i=0\dotsc w}$ is a bases of $\Hom_{\sldc}(\Sn{n}\otimes\Sn{n}\otimes \Sn{n},\Sn{3n-2w})$. There is the same statement for $\mathcal{F}_{2}:=\set{\di{D}(0,w-i,i),i=0\dotsc w}$ and $\mathcal{F}_{3}:=\set{\di{D}(i,0,w-i),i=0\dotsc w}$.

Let
\begin{equation*}
	u_{1}=\Omega_{\alpha\beta}\sigma_{\gamma},\quad u_{2}=\Omega_{\beta\gamma}\sigma_{\alpha},\quad u_{3}=\Omega_{\gamma\alpha}\sigma_{\beta}
\end{equation*}
Those are commutative variables verifying $u_{1}+u_{2}+u_{3}=0$. Now, taking the family
\begin{equation*}
	\mathcal{F}:=\set{\di{D}(w-i_{1},i_{1},0),\di{D}(0,w-i_{2},i_{2}),\di{D}(i_{3},0,w-i_{3}),\quad i_s=0\dotsc m_s-1}
\end{equation*}
lemma~\ref{lem:StrohEg} with $k_{1}=m_{1},k_{2}=m_{2},k_{3}=m_{3}+1$ (for $m_{3}< w$) and $g=w+3$ induces that $\di{D}(m_{3}+1,0,w-m_{3}-1)\in \mathcal{F}_{3}$ is generated by the family $\mathcal{F}$. By induction, $\mathcal{F}_{3}$ and thus all molecules are generated by $\mathcal{F}$.
\end{proof}

\begin{lem}\label{lem:Stroh1}
Let $\mathsf{D}(e_0,e_1,e_2)$ be given by~\ref{fig:DegreeThree}.
\begin{enumerate}
\item If $w\leq n$ then
\begin{equation*}
\mathsf{D}(e_0,e_1,e_2) \text{ is of grade } r\geq \frac{2}{3}w.
\end{equation*}
\item If $w> n$ then
\begin{equation*}
\mathsf{D}(e_0,e_1,e_2) \text{ is of grade } r\geq n-\frac{w}{3}.
\end{equation*}
\end{enumerate}
\end{lem}

\begin{proof}[Sketch of proof]
The detailed proof is in~\cite{GY2010}. Just consider here the case when $w\leq n$ with $w=3k-1$. Taking $m_{1}=m_{2}=m_{3}=m$ in lemma~\ref{lem:Basis_Family} leads to a family $\mathcal{F}$ whose molecules are of grade at least $2k$. We use the same kind of arguments for $w=3k+2$ and $w=3k$.
\end{proof}

A special case of~\ref{lem:Stroh1} is:

\begin{lem}\label{lem:RedDeg3}
Let $\mathsf{D}(e_0,e_1,e_2)$ be given by~\ref{fig:DegreeThree} with $e_{i}+e_{j}\leq n$ ($i\neq j$). Suppose that
\begin{equation*}
e_0\leq \cfrac{n}{2} \text{ and } e_1+e_2>\cfrac{e_0}{2},
\end{equation*}
then
\begin{equation*}
\mathsf{D}(e_0,e_1,e_2) \text{ is of grade } e_{0}+1,
\end{equation*}
unless $e_0=e_1=e_2=\cfrac{n}{2}$.
\end{lem}


\section{Degree three covariant basis}

Take $n$, $p$ and $q$ be three non negative integers. By Clebsch--Gordan decomposition, we first know that we have an $\sldc$ decomposition
\ban
\Sn{n}\otimes \Sn{p}\simeq \bigoplus_{i=0}^{\min(n,p)} \Sn{n+p-2i}
\ean
thus
\ban
	\Sn{n}\otimes \Sn{p}\otimes \Sn{q}\simeq \bigoplus_{i=0}^{\min(n,p)} \Sn{n+p-2i}\otimes \Sn{q}
\ean
The same argument leads to
\ban
	\Sn{n}\otimes \Sn{p}\otimes \Sn{q}\simeq \bigoplus_{j=0}^{\min(p,q)} \Sn{n}\otimes \Sn{p+q-2j}
\ean
We define a triplet $(n,p,i)$ to be \emph{admissible} if $0\leq i \leq \min(n,p)$ : this means that the irreducible component $\Sn{n+p-2i}$ appears is the $\sldc$ decomposition of $\Sn{n}\otimes \Sn{p}$. 

\begin{lem}\label{lem:Deg_3_Cov}
Let $r$ be an integer, $i_1,i_2,j_1,j_2$ be integers such that 
\ban
\begin{cases}
(n,p,i_1),(n+p-2i_1,q,i_2),(p,q,j_1),(n,p+q-2j_1,j_2)& \text{ are admissible} \\
n+p+q-2(i_1+i_2)=r \\
n+p+q-2(j_1+j_2)=r
\end{cases}.
\ean
Then both sets 
\ban
	\phi_{i_1,i_2}:\ff \otimes \bg \otimes \bh\mapsto \tr{\ff}{\tr{\bg}{\bh}{i_1}}{i_2},\quad \Sn{n}\otimes \Sn{p}\otimes \Sn{q}\longrightarrow \Sn{n+p+q-2(i_1+i_2)}
\ean
and 
\ban
	\psi_{j_1,j_2}:\ff \otimes \bg \otimes \bh\mapsto \tr{\tr{\ff}{\bg}{j_1}}{\bh}{j_2},\quad \Sn{n}\otimes \Sn{p}\otimes \Sn{q}\longrightarrow \Sn{n+p+q-2(j_1+j_2)}
\ean
are vector basis of 
\ban
	\Hom_{\sldc}(\Sn{n}\otimes \Sn{p}\otimes \Sn{q},\Sn{r}).
\ean
\end{lem}

\begin{proof}
This lemma is the dual version of \cite[lemma 2.6.1]{CFS1995}, the only idea being that for finite dimensional linear $\sldc$ representations we have
\ban	
	\Hom_{\sldc}(E_{1}\oplus E_{2},F)\simeq \Hom_{\sldc}(E_{1},F)\oplus \Hom_{\sldc}(E_{2},F)
\ean
Now, given integers $i_1,i_2,r$ as in the hypothesis, we have
\ban
	\Sn{n}\otimes \Sn{p}\otimes \Sn{q}\simeq \bigoplus_{i_1=0}^{\min(n,p)} \Sn{n+p-2i_1}\otimes \Sn{q}
\ean
and $\phi_{i_1,i_2}$ is a non nul vector of the one dimensional space 
\ban
	\Hom_{\sldc} (\Sn{n+p-2i_1}\otimes \Sn{q},\Sn{r}).
\ean
We do similarly for integers $j_1,j_2,r$. 
\end{proof}


\section{Relatively complete families of a single binary form}

We give here results about reduction of some families modulo an ideal. We take a space $\Sn{n}$ of binary forms and Gordan's ideal $I_r$ (see definition~\ref{def:Gord_Ideal}). By \eqref{eq:Grade_Impair}, every molecular covariant of grade $1$ is thus in $I_{2}$, and then:

\begin{cor}\label{cor:InitGordan}
The family $\mathrm{A}_0:=\lbrace \ff \rbrace$ is relatively complete modulo $I_{2}$
\end{cor}

The following lemma is about degree three molecular covariants, and is used in the following:

\begin{lem}
Let $V$ be a space of binary forms, $\alpha$, $\beta$ and $\gamma$ be three atoms of respective valence $n$, $p$, $q$. Let $r$ be an integer such that $r\leq \min(n,p,q)$; then
\begin{align}\label{eq:binomMol}
	\begin{tikzpicture}[scale=1.2,baseline={([yshift=-.5ex]current bounding box.center)}]
		\node[m] (P)at(0.2,1.1){$\alpha$};
		\node[m] (Q)at(1.7,1.1){$\beta$};
		\node[m] (R)at(0.95,0.1){$\gamma$};
		\draw[flecheo] (P)--(Q) node[midway,above] {$r$};
	\end{tikzpicture}
	&
	=\sum_{i=0}^r\binom{r}{i}
	\begin{tikzpicture}[scale=1.2,baseline={([yshift=-.5ex]current bounding box.center)}]
		\node[m] (P)at(0.2,1.1){$\alpha$};
		\node[m] (Q)at(1.7,1.1){$\beta$};
		\node[m] (R)at(0.95,0.1){$\gamma$};
		\draw[flecheo] (P)--(R) node[midway,left] {$i$};
		\draw[flecheo] (R)--(Q) node[midway,right] {$r-i$};
	\end{tikzpicture}
\end{align}
\end{lem}

\begin{proof}
Starting with relation~\eqref{eq:DeterminantOmega}:
\begin{equation*}
	\Omega_{\alpha\beta}\sigma_{\gamma}=
	\Omega_{\alpha\gamma}\sigma_{\beta}
	+
	\Omega_{\gamma\beta}\sigma_{\alpha},
\end{equation*}
we get
\begin{equation*} \Omega_{\alpha\beta}^{r}\sigma_{\gamma}^{r}=\sum_{i=0}^r\binom{r}{i}\Omega_{\alpha\gamma}^{i}\Omega_{\gamma\beta}^{r-i}\sigma_{\beta}^{i}\sigma_{\gamma}^{r-i},
\end{equation*}
and we just have to multiply each side of the equation by $\sigma_{\alpha}^{n-r}\sigma_{\beta}^{p-r}\sigma_{\gamma}^{q-r}$.
\end{proof}

Recall here that, for $\ff\in \Sn{n}$ and for a given integer $k\ge 0$, we have 
\ban
\mathbf{H}_{2k}:=\tr{\ff}{\ff}{2k}.
\ean
\begin{lem}\label{lem:Ordre>n}
If $2n-4k>n$, where $2n-4k$ is the order of $\mathbf{H}_{2k}$, then the family $\mathrm{B}=\lbrace \mathbf{H}_{2k} \rbrace$ is relatively complete modulo $I_{2k+2}$.
\end{lem}

\begin{proof}
We have to consider molecular covariants containing
\begin{center}
$\Mdi{D}:=$
\begin{tikzpicture}[scale=1.2,baseline={([yshift=-.5ex]current bounding box.center)}]
		\node[m] (P)at(0.2,1.4){$\ff_{\alpha}$};
		\node[m] (Q)at(1.7,1.4){$\ff_{\beta}$};
		\node[m] (R)at(1.7,0){$\ff_{\gamma}$};
		\node[m] (S)at(0.2,0){$\ff_{\delta}$};
		\draw[flecheno] (P)--(Q) node[midway,above] {$2k$} ;
		\draw[flecheno] (S)--(R) node[midway,below] {$2k$} ;
		\draw[flecheo] (Q)--(R) node[midway,right] {$r$} ;
	\end{tikzpicture}
	with $1\leq r\leq 2k$
\end{center}
all symbol being equivalent. When $r>k$, the molecular covariant
\begin{align*}
	\begin{tikzpicture}[scale=1.5,baseline={([yshift=-.5ex]current bounding box.center)}]
			\node[m] (P)at(0.2,1.1){$\ff_{\alpha}$};
			\node[m] (Q)at(1.7,1.1){$\ff_{\beta}$};
			\node[m] (R)at(0.95,0.1){$\ff_{\gamma}$};
			\draw[flecheo] (P)--(Q) node[midway,above] {$e_0=2k$};
			\draw[flecheo] (Q)--(R) node[pos=0.6,right] {$e_1=r$};
	\end{tikzpicture}
\end{align*}
is of grade $2k+1$ by lemma~\ref{lem:RedDeg3}. Thus the molecular covariant associated to $\Mdi{D}$ is in $I_{2k+1}=I_{2k+2}$ by \eqref{eq:Grade_Impair}.

When $r<k$, by relation~\eqref{eq:binomMol}, $\Mdi{D}$ decomposes as a linear combination of
\begin{center}
\begin{tikzpicture}[scale=1.2,baseline={([yshift=-.5ex]current bounding box.center)}]
		\node[m] (P)at(0.2,1.4){$\ff_{\alpha}$};
		\node[m] (Q)at(1.7,1.4){$\ff_{\beta}$};
		\node[m] (R)at(1.7,0){$\ff_{\gamma}$};
		\node[m] (S)at(0.2,0){$\ff_{\delta}$};
		\draw[flecheno] (P)--(Q) node[midway,above] {$2k$} ;
		\draw[flecheo] (Q)--(R) node[pos=0.4,right] {$r$} ;
		\draw[flecheo] (S)--(P) node[pos=0.4,left] {$2k-i$} ;
		\draw[flecheo] (P)--(R) node[midway,above] {$i$} ;				
	\end{tikzpicture}
	with $0\leq i\leq 2k$
\end{center}
Now:
\begin{itemize}
\item[$\bullet$] If $i\geq k$, we consider the molecule
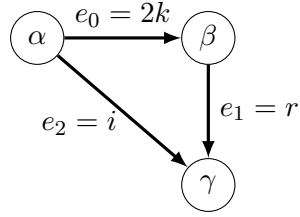
\begin{align*}
	\begin{tikzpicture}[scale=1.5,baseline={([yshift=-.5ex]current bounding box.center)}]
			\node[m] (P)at(0.2,1.4){$\alpha$};
			\node[m] (Q)at(1.7,1.4){$\beta$};
			\node[m] (R)at(1.7,0.1){$\gamma$};
			\draw[flecheo] (P)--(Q) node[midway,above] {$e_0=2k$};
			\draw[flecheo] (Q)--(R) node[midway,right] {$e_1=r$};
			\draw[flecheo] (P)--(R) node[pos=0.6,left=0.1] {$e_2=i$};
	\end{tikzpicture}
\end{align*}
of weight $w=2k+r+i\geq 3k+r>3k$. Since $2k+r+i\leq n$, by lemma~\ref{lem:Stroh1} this molecule is of grade $r\geq \cfrac{2}{3}w> 2k$;
\item[$\bullet$] If $i< k$, we consider the molecule
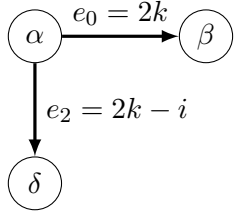
\begin{align*}
	\begin{tikzpicture}[scale=1.5,baseline={([yshift=-.5ex]current bounding box.center)}]
			\node[m] (P)at(0.2,1.4){$\alpha$};
			\node[m] (Q)at(1.7,1.4){$\beta$};
			\node[m] (R)at(0.2,0.1){$\delta$};
			\draw[flecheo] (P)--(Q) node[midway,above] {$e_0=2k$};
			\draw[flecheo] (P)--(R) node[midway,right] {$e_2=2k-i$};
	\end{tikzpicture}
\end{align*}
and we conclude using lemma~\ref{lem:RedDeg3}.
\end{itemize}
\end{proof}

In the same way:

\begin{lem}\label{lem:Ordre=n}
If $n=4k$, then $\mathbf{H}_{2k}$ is of order $n$ and the family $\mathrm{B}=\lbrace \mathbf{H}_{2k} \rbrace$ is relatively complete modulo $I_{2k+2}+\langle \Delta \rangle$ where $\Delta$ is the invariant given by:
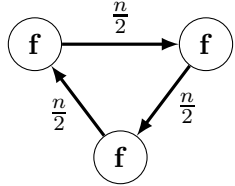
\begin{center}
	\begin{tikzpicture}[scale=1.5]
			\node[m] (P)at(0.2,1.1){$\mathbf{f}$};
			\node[m] (Q)at(1.7,1.1){$\mathbf{f}$};
			\node[m] (R)at(0.95,0.1){$\mathbf{f}$};
			\draw[flecheo] (P)--(Q) node[midway,above] {$\frac{n}{2}$};
			\draw[flecheo] (Q)--(R) node[pos=0.6,right=0.1] {$\frac{n}{2}$};
			\draw[flecheo] (R)--(P) node[pos=0.3,left=0.1] {$\frac{n}{2}$};			
	\end{tikzpicture}
\end{center}
\end{lem}


\section*{Acknowledgement}

The author wishes to thanks professors Abdelmalek Abdesselam and Andries E. Brouwer for their useful remarks concerning this paper, as well as Ana Paula Thomas and Alberto Vigneron-Tenorio for their important advices about linear Diophantine equations. 

\bibliographystyle{abbrv}
\bibliography{Invariant_theory} 
\end{document}